\documentclass[11pt]{amsart}
\usepackage{amsmath,amssymb,latexsym, amsthm, amscd, mathrsfs}
\usepackage[all]{xy}

\theoremstyle{definition}
\newtheorem{Def}[subsection]{Definition}%[section]
\newtheorem{rem}[subsection]{Remark}%[section]
\theoremstyle{plain}
\newtheorem{prop}[subsection]{Proposition}
\newtheorem{thm}[subsection]{Theorem}
\newtheorem{lem}[subsection]{Lemma}
\newtheorem{cor}[subsection]{Corollary}
\newtheorem{conj}[subsection]{Conjecture}
\newcommand{\mbf}{\mathbf}
\newcommand{\mbb}{\mathbb}
\newcommand{\mrm}{\mathrm}
\newcommand{\A}{\mathcal A}
\newcommand{\B}{\mbf B}
\newcommand{\C}{\mathscr C}
\newcommand{\D}{\mathcal D}

\newcommand{\F}{\mathcal F}
\newcommand{\G}{\mbf G}

\newcommand{\IC}{\mrm{IC}}

\newcommand{\M}{\mathscr M}
\newcommand{\Q}{\bar{\mbb Q}}

\newcommand{\T}{\mathcal T}
\newcommand{\U}{\mbf U}
\newcommand{\X}{\mbf X}
\newcommand{\Y}{\mbf Y}
\newcommand{\tF}{\widetilde{\mathcal F}}

\newcommand{\II}{\mathscr I}
\newcommand{\EE}{\mathscr E}
\newcommand{\FF}{\mathscr F}
\newcommand{\DD}{\mathscr D}

\newcommand{\KK}{\mathscr K}
\newcommand{\LL}{\mathscr L}
\newcommand{\BB}{\mathscr B}

\newcommand{\Supp}{\mbox{supp}}

\newcommand{\Hom}{\mrm{Hom}}

\title[Geometric  $\overset{.}{\U}$] {A geometric realization of  modified quantum  algebras}

\author{  Yiqiang Li}
\address{Department of Mathematics\\  University at Buffalo\\ State  University of New York\\ 244 Mathematics Building\\ Buffalo, New York 14260}
\email{yiqiang@buffalo.edu}

\date{\today}
\keywords{Modified quantum algebra,  canonical basis, equivariant derived category, equivariant perverse sheaf, framed representation variety} 
\subjclass{17B37; 14L30; 14F05; 14F43}

\begin{document}

\begin{abstract}
A   geometric construction of Lusztig's modified   quantum  algebra of symmetric type    is presented
 by using  certain localized equivariant derived categories of double framed representation varieties of quivers.
\end{abstract}

\maketitle

\section{Introduction}

\label{Introduction}

Let $\dot{\U}$ be Lusztig's   modified quantum  algebra  of symmetric type and $\dot{\B}$ its  canonical basis. 
This paper is an attempt to geometrize the pair $(\dot{\U}, \dot{\B})$ by using the geometry  of $double$ framed representation varieties $\mbf E_{\Omega}$ of quivers.
There are several  work in literature strongly suggesting  the existence of such a construction. 
Meanwhile, there are    obstacles in getting this  construction,  largely because $\mbf E_{\Omega}$ is not really in complete analog with  partial flag varieties used in  
Beilinson-Lusztig-MacPherson's construction of the pair $(\dot{\U},\dot{\B})$ of type $\mbf A$ in ~\cite{BLM90}. 
As a matter of fact,  partial flag varieties are geometric-invariant-theory (GIT)  quotients of certain open subvarieties of $\mbf E_{\Omega}$ of type $\mbf A$. 
Although the GIT quotients still make sense in other types,  they do not produce a construction of the pair $(\dot{\U},\dot{\B})$. We refer to ~\cite{Reineke08} and ~\cite{Li10} for the progress in this direction.
In another direction,  Zheng constructed in ~\cite{Zheng08} 
a set of  endomorphism functors between certain localized equivariant derived categories of framed representation varieties  and  showed  that it satisfies
the defining relations  of $\dot{\U}$.
It is natural to see if similar   localization  in  $\mbf E_{\Omega}$ helps us in obtaining the pair $(\dot{\U},\dot{\B})$.  
Indeed, it does as we obtain  the following results in this paper.

\begin{thm}
\label{all}
\begin{enumerate}
\item Similar localization $\M_d$ of equivariant derived categories of $\mbf E_{\Omega}$ admits 
 an associative  convolution product.

\item The convolution product is independent of the choices of orientations of quivers.

\item A set of bounded complexes in $\M_d$  satisfies the defining relations of the integral form of $\dot{\U}$. 

\item  There is a surjective algebra homomorphism from the integral form of  $\dot{\U}$ to the Grothedieck group of $\M_d$.  

\item Our construction is compatible with that of ~\cite{BLM90} and ~\cite{Zheng08}.

\item Up to an assumption, the complexes involved in this construction are semisimple. 

\item The canonical basis elements get sent to the simple perverse sheaves or zero, up to shifts,  
         arising in this construction in  the special  case  corresponding to the one studied in ~\cite{BLM90}. 

\item In the setting of item (7), the simple perverse sheaves are exactly those microlocal perverse sheaves (\cite{KS90}, \cite{GMV05}, \cite{W04}) 
         defined by using Nakajima's quiver varieties (\cite{Nakajima98}-\cite{Nakajima00}).
         
\end{enumerate}
 \end{thm}

The convolution product on the category $\M_d$ is obtained by using  left adjoints of the localization functors, whose existence is shown in this paper,  and 
 the general direct image functors with compact support   and general inverse image functors  defined by Bernstein and Lunts (\cite{BL94}, ~\cite{LO08b}). 
The assumption  mentioned in Theorem ~\ref{all} (7) assumes   that the complexes corresponding to the idempotent generators in $\dot{\U}$ are semisimple. 
It is proved in full generality in a closely related setting in ~\cite{W12}.
Based on Theorem ~\ref{all} (7) and (8), we conjecture that the assumption holds in our setting in full generality also and that the simple perverse sheaves arising from this construction corresponds to 
elements in $\dot{\B}$.

It will be very interesting to relate the construction  
to the Harish-Chandra-bimodule construction in ~\cite[Theorem47]{MS08} via localization,
 that in affine type $\mbf A$ cases
in  ~\cite{GV93} and ~\cite{Lusztig99}, 
the one in type $\mbf D$ cases  in ~\cite{Li10},
to the geometric realization of quantum affine algebras in  \cite{Nakajima00} (in view of ~\cite{T87} and  ~\cite{G94}), 
and to the categorification of $\dot{\U}$ in  ~\cite{KL08a}-\cite{KL08c},  ~\cite{R08} and ~\cite{MSV10}.

{\bf Acknowledgement.} 
We thank  Professor Z. Lin,  H. Nakajima, M. Shimozono,  W. Wang and B. Webster for interesting discussions. 
We thank Professor C. Stroppel for pointing out the connection with ~\cite{MS08}. 
A large part of the work was done while the author was  visiting Virginia Tech. 
This work is partially supported by NSF grant: DMS-1160351.

\setcounter{tocdepth}{1}
\tableofcontents

\section{Preliminary, I}

We shall recall the definition of  a  symmetric quantum modified algebra  from ~\cite{Lusztig93}.

\subsection{Symmetric Cartan datum,  root datum and graph}
\label{Cartan}

Let $I$ be a finite set, and 
\[
\cdot: \mbb Z[I]\times \mbb Z[I] \to \mbb Z
\]
be a symmetric bilinear form.  
The pair $(I, \cdot)$ is called a $symmetric$ $Cartan$ $datum$ if 
\[
i\cdot i =2  \quad \mrm{and} \quad 
i\cdot j \in \{ 0, -1, -2, \cdots \},  \quad \forall i \neq j  \in I. 
\]
We call a triple $(\X, \Y, (, ))$ a $root$ $datum$ of $(I, \cdot)$ if the following properties are satisfied:

\begin{itemize}
\item   $\X$ and $\Y$ are finitely generated free abelian groups and $(, ): \Y\times \X \to \mbb Z$ is a perfect pairing;
\item   there are two embeddings $I\hookrightarrow  \X$, $i\mapsto \alpha_i$ and $I\hookrightarrow \Y$, $i\mapsto \check{\alpha}_i$ 
            such that $(\check{\alpha}_i,\alpha_ j)=i\cdot j$ for any  $i, j\in I $.
\end{itemize}
 
 Let $\X^+=\{ \lambda \in \X| (\check{\alpha}_i,  \lambda)\in \mbb  N, \forall i\in I\} $ be the set of all dominant integral weights.
 
To a symmetric Cartan datum,  associated  a $graph$ $\Gamma$ consisting of the following data:
\[
\Gamma = (I, H; \; ', ''  : H\to I; \, \bar \, : H\to H) 
\]
where $I$ is the vertex set and $H$ is the edge set, the maps $'$ and $''$ are the source and target maps, respectively,
and the map $\;\bar{\empty}\; $ is the fixed-point-free involution such that 
\begin{eqnarray*}
h'\neq h'',  h' =(\bar h)'' \;\mbox{and} \; \#\{ h\in H | h'=i, h''=j\}= - i\cdot j, \quad \forall h\in H, i \neq j\in I.
\end{eqnarray*}

\subsection{Modified quantum   algebra}
\label{modified-quantum}
Let $\mbb Q(v)$ be the rational field with the indeterminate $v$.  Let 
\begin{equation*}
\begin{split}
 [[s]]=\frac{v^s-v^{-s}}{v-v^{-1}},\quad 
[[s]]^!=[[s]][[s-1]]\cdots [[1]],
\quad \mbox{for any} \; s\in \mbb N.
\end{split}
\end{equation*}

The  $modi\!f\!ied$ $quantum$  $algebra$ $\dot{\U}$   of Lusztig attached to a root datum $(\X, \Y, (,))$ of a symmetric Cartan datum $(I, \cdot)$ 
is  a $\mbb Q(v)$-algebra without unit determined by the following generators and relations.
The generators are 
\[
1_{\lambda}, \; E_{\lambda+\alpha_i, \lambda}\;\mbox{and}\; F_{\lambda-\alpha_i,\lambda}, \quad \forall i \in I, \lambda \in \X.
\]
The relations are
\begin{align*}
\tag{$\dot{\U}$a} & 1_{\lambda} 1_{\lambda'}=\delta_{\lambda,\lambda'} 1_{\lambda}, \quad \forall \lambda,\lambda'\in \X;\\
\tag{$\dot{\U}$b} & E_{\lambda+\alpha_i,\lambda} 1_{\lambda'}=\delta_{\lambda,\lambda'} E_{\lambda+\alpha_i,\lambda},
\quad
1_{\lambda'} E_{\lambda+\alpha_i,\lambda} =\delta_{\lambda',\lambda+\alpha_i} E_{\lambda+\alpha_i,\lambda}, \quad \forall i\in I, \lambda,\lambda'\in \X;\\
\tag{$\dot{\U}$c}  & F_{\lambda-\alpha_i,\lambda} 1_{\lambda'}=\delta_{\lambda,\lambda'} F_{\lambda-\alpha_i,\lambda}, 
\quad
1_{\lambda'} F_{\lambda-\alpha_i,\lambda} = \delta_{\lambda',\lambda-\alpha_i} F_{\lambda-\alpha_i,\lambda}, \quad \forall i\in I, \lambda, \lambda' \in \X;\\
\tag{$\dot{\U}$d}   &E_{\lambda-\alpha_j+\alpha_i,\lambda-\alpha_j} F_{\lambda-\alpha_j,\lambda}- F_{\lambda+\alpha_i-\alpha_j,\lambda+\alpha_i} E_{\lambda+\alpha_i,\lambda} 
= \delta_{st} [[(\check{\alpha}_i,\lambda)]]  1_{\lambda}, \quad \forall  i, j\in I, \lambda \in \X;\\
\tag{$\dot{\U}$e}   & 
\sum_{p=0}^m  (-1)^p     
E^{(m-p)}_{\lambda+m\alpha_i+\alpha_j,  \lambda+p\alpha_i+\alpha_j} 
E_{\lambda+p\alpha_i+\alpha_j, \lambda+p\alpha_i} 
E^{(p)}_{\lambda+p\alpha_i, \lambda} =0, \quad  \forall i\neq j \in I, \lambda \in \X;\\
\tag{$\dot{\U}$f} \quad &
\sum_{p=0}^m  (-1)^p     
F^{(m-p)}_{\lambda-m\alpha_i-\alpha_j,  \lambda-p\alpha_i-\alpha_j} 
F_{\lambda-p\alpha_i-\alpha_j, \lambda-p\alpha_i} 
F^{(p)}_{\lambda-p\alpha_i, \lambda} =0, \quad  \forall i\neq j \in I, \lambda \in \X;
\end{align*}
where we set $m=1-i \cdot j$ in the q-Serre relations ($\dot{\U}$e) and ($\dot{\U}$f) and 
\begin{equation*}
\begin{split}
E^{(n)}_{\lambda+n \alpha_i, \lambda} &=\frac{1}{[[n]]^!}  E_{\lambda+n\alpha_i, \lambda+(n-1) \alpha_i } \cdots E_{\lambda +2\alpha_i, \lambda+\alpha_i} E_{\lambda +\alpha_i, \lambda}, \quad\\
F^{(n)}_{\lambda-n \alpha_i, \lambda} &=\frac{1}{[[n]]^!}  F_{\lambda-n\alpha_i, \lambda-(n-1) \alpha_i } \cdots F_{\lambda -2\alpha_i, \lambda-\alpha_i} F_{\lambda -\alpha_i, \lambda},\quad \forall 
n\in \mbb N, i\in I, \;\mbox{and}\; \lambda \in \X.
\end{split}
\end{equation*}

Let $\mbb A=\mbb Z[v, v^{-1}]$
be the subring of Laurent polynomials in $\mbb Q(v)$. Let $_{\mbb A}\! \dot{\U}$ be the $\mbb A$-subalgebra of $\dot{\U}$ generated by the elements
$1_{\lambda}$, $E^{(n)}_{\lambda+n\alpha_i, \lambda}$ and $F^{(n)}_{\lambda-n\alpha_i, \lambda}$ for all $\lambda \in \X$, $n\in\mbb N$ and $i\in I$.
Let $\dot{\B}$ be the canonical basis of $\dot{\U}$ defined in ~\cite[25.2.4]{Lusztig93}.

\section{Preliminary, II}
\label{geometric}
We shall recall the equivariant derived category from ~\cite{BL94}, ~\cite{LMB00}, ~\cite{LO08a}-\cite{LO08b}. 
We will use the presentations in  ~\cite{Schn08} and ~\cite{WW09}.

\subsection{Derived category $\D^*(X)$}

Let $p$ be a prime number. Fix an algebraic closure $k$ of the finite field $\mbb F_p$ of $p$ elements.  
All algebraic varieties  in this paper will be over $k$.

Let $l$ be a prime number different from $p$.  Let $\bar{\mbb Q}_l$ be an algebraic closure of the field of $l$-adic numbers. 

We write $\D^b(X)$ for the $bounded$ $derived$ $category$ of complexes of  $\bar {\mbb Q}_l$-constructible sheaves on an algebraic variety $X$ defined in ~\cite[2.2.18]{BBD82}.
See also ~\cite[I, \S12] {FK88} and ~\cite[II 5] {KW01}.
We denote  $\D^{*}(X)$, $*=\{ \phi, +, - \}$,  for the similar  derived category with the term ``bounded'' replaced by 
``unbounded'', ``bounded below'', ``bounded above'', respectively.

The shift functor will be denoted by ``$[-]$''.  The functors $Rf_*$, $Rf_!$, $Lf^*$, and $Rf^!$ associated to a given morphism $f: Y\to X$ of varieties will be written as 
$f_*$, $f_!$, $f^*$ and $f^!$, respectively, in this paper.  We also write $K\to L\to M \to $ for any distinguished triangle $K\to L \to M\to K[1]$ in the derived categories just defined.

\subsection{Equivariant derived category $\D^*_{G}(X)$} 
\label{equivariant}

Let us recall the  equivariant derived categories following  ~\cite{BL94}, ~\cite{Schn08}, and ~\cite{WW09}.

A morphism $f: Y\to X$ of varieties over $k$ is called $n$-$acyclic$ if for any sheaf $\mathcal F$ over $X$, considered as a complex of sheaves concentrated on the zeroth degree,  then
\begin{align}
\label{acyclic}
\mathcal F \simeq \tau_{\leq n} f_* f^* (\mathcal F),
\end{align}
and the property (\ref{acyclic}) holds under base changes, where $\tau_{\leq n}: \D(X) \to \D^{\leq n}(X)$ is the truncation functor.

Let a linear algebraic group, $G$,   over $k$ act on $X$.

A $G$-equivariant morphism $E\overset{e} {\to } X$ is called a  $resolution$  $o\! f$  $X$  (with respect to $G$) if  the action of $G$ on $E$ is free and its quotient map  $E\to G\backslash E$ is a locally trivial fibration with fiber $G$.  
We thus have the following diagram
\[
\begin{CD}
X @<e<< E @>\bar e>> \bar E,
\end{CD}
\]
where $\bar E = G\backslash E$.
Let $[G\backslash X]$ be the category whose objects are all resolutions  of $X$ and 
morphisms between  resolutions of $X$ are defined to be $G$-equivariant  morphisms over $X$. 
By ~\cite[18.7.5]{LMB00}, 
$[G\backslash X]$ is an algebraic stack of Bernstein-Lunts. 

A resolution $ E\overset{e}\to X$ of $X$ is called $smooth$ if $e$ is smooth.
If, moreover, $\bar e$ is n-acyclic, we call the resolution $E\overset{e}\to X$ an $n$-$acyclic$ $smooth$ $resolution$ $o\! f$ $X$.

A  sequence, $(E_n, f_n)_{n\in \mbb N}$, of smooth resolutions of $X$:
\[
\begin{CD}
E_0  @>f_0>> E_1 @>f_1>> \cdots  @>>> E_n @>f_n>>  E_{n+1} @>>> \cdots,
\end{CD}
\]
is called $acyclic$ if  $e_n: E_n \to X$ is $n$-acyclic and $f_n$ is a  closed inclusion, $\forall n\in \mbb N$.

Given any smooth resolution $E\overset{e}{\to} X$, the $category$ $\D^+(\bar E| e)$ of $objects$ $from$ $X$ is the full subcategory of $\D^+(\bar E)$ consisting of all objects $K$ such that there is a complex
$L\in \D^+(X)$ satisfying $e^*(L) \simeq \bar e^*(K)$.

If $f: E\to E'$ is a morphism of smooth resolutions of $X$, then it induces  a map $\bar f: \bar E\to \bar E'$, which gives rise to the inverse image functor 
$\bar f^*: \D^+(\bar E'| e' ) \to \D^+(\bar E|e)$.

Given any acyclic sequence,  $(E_n, f_n)_{n\in \mbb N}$, of smooth resolutions of $X$, we then have a sequence of  functors:
\[
\begin{CD}
\D^+(\bar E_0| e_0)   @<\bar f_0^* << \D^+(\bar E_1|e_1)  @<\bar f_1^* << \cdots   \D^+ (\bar E_n|e_n)  @<\bar f_n^*<<  \D^+(\bar E_{n+1}|e_{n+1})  @<<< \cdots.
\end{CD}
\]
The $inverse$ $limit$ $category$ $\underleftarrow{\lim} \mathcal \D^+(\bar E_n|e_n)$ is defined to be the category whose objects are pairs $(A_n; \phi_n)_{n\in \mbb N}$,  with 
$A_n\in \D^+(\bar E_n|e_n) $ and $\phi_n: \bar f^*_n(A_{n+1})  \tilde{\to} A_n$ isomorphisms;  and whose morphisms 
$\alpha: (A_n; \phi_n)_{n\in \mbb N}  \to (B_n; \psi_n)_{n\in \mbb N}$ are  collections of morphisms $\alpha_n: A_n\to B_n$ such that 
$\alpha_n \phi_n = \psi_n  \bar f^*_n (\alpha_{n+1})$, $\forall n\in \mbb N$.

\begin{Def}
The $bounded$  $below$  $G$-$equivariant$ $derived$ $category$ $of$ $X$ is defined to be 
\[
\D^+_G(X) =\underleftarrow{\lim} \; \D^+(\bar E_n|e_n).
\]
\end{Def}
Similarly, we define the $bounded$ (resp. $bounded$ $above$; $unbounded$)  $G$-$equivariant$ $derived$ $category$ $\D^b_G(X)$ (resp. $\D^-_G(X)$; $\D_G(X)$),  
with $\underleftarrow{\lim} \; \D^+(\bar E_n| e_n)$ replaced by  
$\underleftarrow{\lim} \; \D^b(\bar E_n|e_n)$ (resp. $\underleftarrow{\lim} \; \D^-(\bar E_n|e_n)$; $\underleftarrow{\lim} \; \D(\bar E_n|e_n)$).

We call the objects in the categories just defined $equivariant$ $complexes$.

The following results are proved in ~\cite{Schn08} and  ~\cite{WW09}.

\begin{prop}
\label{Geometric-Grothendieck}
\begin{enumerate}
\item[(a)]  Acyclic sequences of  smooth resolutions of $X$  exist;
\item[(b)]  The definitions of equivariant derived categories are independent of the choices of the acyclic  sequences of smooth resolutions of $X$;
\item [(c)]  The category $\D^+_G(X)$ $($resp. $\D^b_G(X)$, $\D^-_G(X)$$)$ is a triangulated category. $\Sigma $ is a 
          distinguished triangle in $\D^+_G(X)$ $($resp. $\D^b_G(X)$, $\D^-_G(X))$ if and only if the projection of $\Sigma$ to $\D^+(\bar E_n|e_n)$ $($resp. $\D^b(\bar E_n|e_n)$, $\D^-(\bar E_n|e_n))$
          is a distinguished triangle for all $n\in \mbb N$;
\end{enumerate}
\end{prop}

For  any object $(A_n, \phi_n) \in \D_G^+( X)$, there exists $B_n \in \D^+(X)$ such that $\bar e_n^* A_n \simeq e_n^* B_n$ for any $n\in \mbb N$. 
So $e_n^* B_n \simeq f_n^* \bar e_{n+1}^* (A_{n+1})\simeq e_n^* B_{n+1}$. 
Thus  the stalks of $B_n$ and $B_{n+1}$ at any given points are isomorphic
to each other.  Hence $\mbox{supp} (B_n) =\mbox{supp} (B_{n+1})$, for all $n\in \mbb N$.
Therefore, we can define the support of the object $(A_n, \phi_n)$ as follows.

\begin{Def}
The $support$ of the object $(A_n, \phi_n)\in \D_G^+(X)$ is defined to be the support of the $B_n$'s.
\end{Def}
Similarly, we can define the supports of any objects in $\D_G^b(X)$, $\D_G^-(X)$, and $\D_G(X)$.

Finally, we should give another description of the bounded derived category $\D^b_G(X)$ in order to talk about the singular supports, 
which will be defined in Section ~\ref{TypeA}, of objects  in $\D^b_G(X)$ for varieties $X$ defined over a characteristic zero algebraically closed field. 

Let $X \overset{e}{\leftarrow} E \overset{\bar e}{\to} \bar E$ be a smooth resolution of $X$ with respect to $G$.  
We define the category  $\D^b_G (X, E)$ to be the category whose objects are 
\begin{itemize}
\item triples $K=(K_X, \bar K, \theta)$,  where $K_X \in \D^b(X)$, $\bar K \in \D^b (\bar E)$ and $\theta: e^* K_X \to \bar e^* \bar K$ is an isomorphism;
\end{itemize}
and whose morphisms are
\begin{itemize}
\item pairs $\alpha=(\alpha_X, \bar \alpha): K\to L$,  with $\alpha_X:K_X\to L_X\in \D^b(X)$ and $\bar \alpha : \bar K \to \bar L \in \D^b(\bar E)$ such that $\theta e^* (\alpha_X)=e^* (\bar \alpha)  \theta$.
\end{itemize}

Given any morphism $\nu : E\to E_1$ of smooth resolutions of $X$, it induces a morphism $\bar \nu: \bar E\to \bar E_1$, and a functor 
\[
\nu^*: \D^b_G(X, E_1) \to \D^b_G(X, E),  \quad (K_X, \bar K, \theta) \to (K_X, \bar \nu^* \bar K, \xi),
\]
where $\xi$ is the composition 
$e^* (K_X) =\nu^* e_1^* K_X \overset{\nu^* \theta}{\to} \nu^* \bar e_1^*(\bar K)=e^* \bar \nu^* \bar K$.

An acyclic sequence $(E_n, f_n)$ of smooth resolutions of $X$ gives rise to a sequence of functors
\[
\begin{CD}
\D^b_G(X, E_0) @<f_0^*<<  \D^b_G(X, E_1)  @<f_1^*<<  \cdots @<<< \D^b_G(X, E_n) @<<< \cdots.
\end{CD}
\]
The following proposition is  proved in ~\cite[5.3]{Schn08} and ~\cite[Part I, 2]{BL94}.

\begin{prop} 
$\D^b_G(X) \simeq \underleftarrow{\lim} \;  \D^b(X, E_n)$.
\end{prop}

Recall from ~\cite[Part I, 3.2]{BL94} that there exists a functor 
\begin{equation}
\label{tensorfunctor}
\otimes: \D^b_G(X) \times \D^b_G(X) \to \D^b_G(X)
\end{equation}
which is exact on each variable. Equipped with the tensor functor $\otimes$, the derived category $\D^b_G(X)$ is a tensor triangulated category (\cite{B05}).

We refer to ~\cite[I, 5]{BL94} for the definition of $equivariant$ $perverse$ $sheaves$.  
Recall that to any irreducible, $G$-invariant, closed subvariety $X_1$ in $X$, associated  the $equivariant$ $intersection$ $cohomology$, denoted by 
$\IC_G(X_1)$. Let 
\begin{equation}
\label{shiftedintersection}
\widetilde{\IC}_G(X_1) =\IC_G(X_1) [-\dim X_1]. 
\end{equation}
Note that if $X$ is smooth, $\widetilde{\IC}_G(X) =\bar {\mbb Q}_{l, X}$, the constant sheaf on $X$.

\begin{rem}
\label{geometric-identification}
As in ~\cite[18.7]{LMB00}, the equivariant derived category $\D^*_G(X)$ for $*=\{\phi, +, -, b\}$ is equivalent to the derived category $D^*_c([G\backslash X], \bar{\mbb Q}_l)$ of complexes of  
lisse-\'{e}tale $\bar{\mbb Q}_l$-sheaves on the quotient stack $[G\backslash X]$ of cartesian constructible cohomologies defined in ~\cite{LMB00}. 
{\em We shall identify these two categories throughout the paper.}
\end{rem}

\subsection{General inverse functor  and general direct image functor}
\label{generaldirectimage}

Throughout this paper, we assume that  for any homomorphism, say $\phi: H\to G$, of linear algebraic groups,
\begin{itemize}
\item $H=G\times G_1$ for some linear algebraic group $G_1$ and $\phi$ is the projection to $G$.
\end{itemize}

A morphism $f: Y\to X$ is called a 
$\phi$-$map$ if $ f(h.y)=\phi(h).f(y)$ for any $h\in H$, $y\in Y$. Any $\phi$-map $f: Y\to X$ gives rise to a morphism 
\begin{equation}
\label{general-f}
 Q\! f: [H\backslash Y] \to [G\backslash X]
 \end{equation}
of algebraic stacks (\cite{LMB00}) defined as follows.
For any given resolution $Y\overset{e}{\leftarrow} E$,  the composition $X \overset{fe} {\leftarrow} E$ 
factors through the quotient map $q: E \to G_1\backslash E$ due to the assumption that $G_1$ acts trivially on $X$.
In other words, there exists a unique morphism $X \overset{e_1}{\leftarrow} G_1\backslash E$ such that $e_1 q=fe$. Moreover, 
$H\backslash E = G\backslash (G_1\backslash E)$.
For any morphism $\nu: E \to E'$ in $[H\backslash Y]$,  it naturally induces a morphism $\bar \nu: \bar E \to \bar E'$ in $[G\backslash X]$.
The morphism  $Q\! f$ is then defined to be 
\begin{equation*}
\begin{split}
&Q \! f:   Y \overset{e}{\leftarrow} E  \mapsto   X\overset{e_1}{\leftarrow} G_1\backslash E, \quad \forall H \mbox{-resolution}  \; Y\overset{e}{\leftarrow} E \in [H\backslash Y];\\
&Q \! f:  E \overset{\nu}{\to} E' \mapsto   \bar E  \overset{\bar{\nu}}{\to} \bar E', \quad \forall \; \mbox{morphism} \; E \overset{\nu}{\to} E' \in [H\backslash Y].
\end{split}
\end{equation*}

By ~\cite{LMB00}, ~\cite{LO08b}, and Remark ~\ref{geometric-identification},   the morphism $Q\!f$ gives rise to the following functors:
\begin{equation}
\label{inducedfunctors}
\begin{split}
&Q \! f_*: \D^+_H(Y) \to \D^+_G(X), \quad Q\! f_!: \D^-_H(Y) \to \D^-_G(X),\\
&Q\! f^*: \D_G(X) \to \D_H(Y) , \quad Q \! f^!: \D_G(X) \to \D_H(Y).
\end{split}
\end{equation}
When $H=G$, we simply write $f$ for the morphism $Q\! f$, and  $f_*$, $f_!$, $f^*$, $f^!$ for the functors $Q\!f_*$, $Q\! f_!$, $Q\! f^*$ and $Q\! f^!$, respectively.
These functors satisfy the following properties. We refer to ~\cite{LMB00} and ~\cite{LO08b} for proofs of these properties.

\begin{lem}
The pairs $(Q \! f^*, Q\! f_*)$ and  $(Q\! f_!, Q\! f^!)$ are  adjoint pairs.
\end{lem}

\begin{lem}
\label{general-distribution}
$Q\!f^*(A\otimes A') =Q\!f^* A \otimes Q\! f^* A'$, for any $A, A'\in \D^b_{G} (X)$.
\end{lem}

\begin{lem}
\label{general-composition}
If, moreover, $\psi: I\to H$ is a morphism of linear  algebraic groups, and $g: Z\to Y$ a $\psi$-map,  we have
\[
Qg^*  Q\! f^*\simeq  Q(fg)^* \quad \mbox{and} \quad Q\! f_* Qg_* \simeq Q(fg)_*.
\]
\end{lem}

\begin{lem}
\label{general-projection}
$Q\! f_! (A \otimes Q\! f^* (B) )  \simeq Q\! f_!(A) \otimes B$, for any $A\in \D^-_H(Y)$ and $B\in \D^-_G(X)$. 
\end{lem}

\begin{lem}
\label{composition}
If, moreover, $\psi: I\to H$ is a morphism of  linear algebraic groups, and $g: Z\to Y$ a $\psi$-map,  we have
\[
Qg^!  Q\! f^!\simeq  Q(fg)^! \quad \mbox{and} \quad Q\! f_! Qg_! \simeq Q(fg)_!.
\]
\end{lem}

Suppose that 
\[
I=G\times G_1 \times G_2, \quad H=G\times G_1, \quad \mbox{and} \quad H_1=G\times G_2.
\]
We denote the projections of linear algebraic groups as follows.
\[
\begin{CD}
I @>\psi>> H @>\phi>> G, \quad \mbox{and}  \quad 
I @>\psi_1>> H_1 @>\phi_1>> G.
\end{CD}
\]
Assume, further,  that we have a  cartesian diagram 
\[
\begin{CD}
Z @>g>> Y\\
@Vf_1VV @VfVV\\
Y_1 @>g_1>> X,
\end{CD}
\]
and the morphisms $f$, $g$,   $f_1$ and $g_1$ are $\phi$, $\psi$, $\psi_1$ and $\phi_1$ maps, respectively.
This  cartesian  diagram then gives rise to the following  cartesian diagram 
\begin{equation}
\label{cartesiandiagram}
\begin{CD}
[I\backslash Z] @>Q g>> [H\backslash Y]\\
@VQ\!f_1VV @VQ\! fVV\\
[H_1\backslash Y_1] @>Qg_1>> [G\backslash X].
\end{CD}
\end{equation}

In fact, one can define a morphism, 
\[
\mathscr X \overset{\xi}{\to} [I\backslash Z],\]  
from the fiber product $\mathscr X:=[H_1\backslash Y_1] \times_{[G\backslash X]} [H\backslash Y]$ to $[I\backslash Z]$ as follows.
To a triple $T=( E_1\to Y_1, E\to Y, \alpha:  Qg_1( E_1\to Y_1) \simeq Q\! f( E\to Y))$ in the fiber product $\mathscr X$,   we form the fiber product 
$E_1 \times_{(G_1\backslash E)} E$, on which there is a free $I$-action induced from the free $H_1$ (resp. $H$, $G$) action on $E_1$ (resp. $E$, $G_1\backslash E$). 
It is clear that $I\backslash E_1\times_{(G_1\backslash E)} E\simeq H\backslash E$.
By the universal property of the cartesian diagram $(g, f; f_1, g_1)$, we see that there is a unique $I$-equivariant morphism $ E_1\times_{(G_1\backslash E)} E \to Z$. 
The morphism  $\xi$ is  defined by sending  the triple $T$ 
to $ E_1\times_{(G_1\backslash E)} E \to Z$. 
By the universal property of $\mathscr X$, there is a morphism $\xi_1: [I\backslash Z]\to \mathscr X$. It is not difficult to show that $\xi \xi_1 = \mbox{id}_{[I\backslash Z]}$ and
$\xi_1\xi =\mbox{id}_{\mathscr X}$. 
So $[I\backslash Z]$ is isomorphic to the fiber product $\mathscr X$.  Therefore, the diagram above is cartesian.

The following lemma holds
from the above cartesian diagram (\ref{cartesiandiagram}) of algebraic stacks.

\begin{lem}
\label{general-basechange}
$Q\! f^* Qg_{1!} \simeq Qg_! Q\!f_1^*$.
\end{lem}

\subsection{Compatibility with localization}

Let $\mathcal A, \mathcal B$ and $ \mathcal C$ be three triangulated categories with exact functors
\[
\mathcal A \overset{f_1} {\to} \mathcal B \overset{f_2}{\to} \mathcal C.
\]
Suppose that $\T_{\mathcal A}$, $\T_{\mathcal B}$ and $\T_{\mathcal C}$ are thick subcategories of $\mathcal A$, $\mathcal B$,
and $\mathcal C$, respectively.
Then we have three ``exact sequences'' of categories (\cite[1.4.4]{BBD82})
\begin{align*}
\begin{split}
0 \to \T_{\mathcal A} \overset{\iota}{\to} \mathcal A \overset{Q}{\to} \mathcal A/\T_{\mathcal A} \to 0, \quad 
0 \to \T_{\mathcal B} \overset{\iota}{\to} \mathcal B\overset{Q}{\to} \mathcal B/\T_{\mathcal B} \to 0,\quad
0 \to \T_{\mathcal C} \overset{\iota}{\to} \mathcal C \overset{Q}{\to} \mathcal C/\T_{\mathcal C} \to 0
\end{split}
\end{align*}
Moreover, we assume that each localization functor $Q$ admits a left  adjoint $Q_!$ and a right adjoint $Q_*$. 
This implies that the functor $\iota$ admits a left adjoint $\iota^*$ and a right adjoint $\iota^!$ (\cite[\S 2]{Verdier76}) .
We form the following functors
\begin{align*}
\begin{split}
F_1= Q \circ f_1 \circ Q_!:  & \A/\T_{\A} \to  \mathcal B/\T_{ \mathcal B}; \quad 
F_2 =Q\circ f_2 \circ Q_!: \mathcal B/ \T_{\mathcal B} \to \mathcal C/ \T_{\mathcal C};\\
&F_3= Q\circ f_2f_1 \circ Q_!: \A/\T_{\A} \to \mathcal C/\T_{\mathcal C}.
\end{split}
\end{align*}

\begin{lem}
\label{localization-composition}
If $\iota^* f_1 Q_!=0$  or $Q \circ  f_2\circ  \iota=0$, then
$F_2 \circ F_1 =F_3$.
\end{lem}

\begin{proof}
From ~\cite[6.7]{Verdier76}, we have a distinguished triangle of functors
\[
Q_ ! Q\to  \mrm{Id}  \to \iota  \iota^* \to. 
\]
Thus 
we have a distinguished triangle of functors
\[
Q \circ f_2  (Q_! Q) f_1 Q_! \to Q\circ  f_2   (\mrm{Id}) f_1 Q_* \to Q\circ f_2 (\iota \iota^* ) f_1 Q_!\to  .
\]
But the third term is zero by the assumption. Lemma follows.
\end{proof}

Suppose that $\D$ is a fourth  triangulated category with a thick subcategory $\T_{\D}$ such that the localization functor $Q: \D \to \D/\T_{\D}$ admits a right adjoint $Q_*$ and a left adjoint $Q_!$.
Assume, moreover, that we have a pair of exact functors $g_1: \A \to \D,  g_2: \D \to \mathcal C$ such that $f_2 f_1\simeq g_2 g_1$.   We can form the following functors:
\[
G_1= Q\circ  g_1 \circ Q_! \quad \mbox{and} \quad  G_2 =Q \circ g_2 \circ  Q_!.
\]

\begin{lem}
\label{cartesian}
If  $Q\circ f_2\circ  \iota=0$ and $\iota^* g_1 Q_! =0$, then $F_2  F_1 \simeq G_2 G_1$.
\end{lem}

This is because , we have 
$F_2 F_1 =  Q\circ f_2f_1 \circ Q_!  =  Q\circ g_2g_1 \circ Q_! = G_2 G_1$,
by  Lemma ~\ref{localization-composition}. 

Recall from ~\cite{B05} that a tensor triangulated category is a triangulated category equipped with a symmetric monoindal functor which is exact on each variable. 
A thick subcategory $\T$ of  a tensor triangulated category $\A$ is called a $thick$ $tensor$ $ideal$ if for any $T\in \T$ and $A\in \A$, we have $T\otimes A\in \T$. 
If $\T$ is a thick tensor ideal of $\A$, then the quotient categories
$\A/\T$ is again a tensor triangulated category. Moreover, the tensor structures are compatible with the localization functor in the sense that 
$Q(A\otimes A') \simeq  Q(A) \otimes Q(A')$ for any $A, A'\in \A$.

\begin{lem}
\label{distribution}
Assume that $\A$ and $\mathcal B$ are tensor triangulated categories and $\T_{\A}$ and $\T_{\mathcal B}$ are thick tensor ideals.
 Suppose that $f_1 : \A \to \mathcal B$ is an exact functor  such that $f_1(A\otimes A')=f_1 A\otimes f_1 A'$ and $Q\circ f_1\circ \iota=0$, then 
\[F_1(Q(A)\otimes Q(A'))\simeq F_1(A) \otimes F_1(A').\]
\end{lem}

\begin{proof}
We have $F_1(Q(A)\otimes Q(A')) \simeq F_1 Q  (A\otimes A') =Q\circ f_1 Q_! Q(A\otimes A')$.   Consider the distinguished triangle
\[
F_1 (Q(A)\otimes Q(A') \to Q \circ f_1 (A\otimes A') \to Q\circ f_1 \iota \iota^* (A\otimes A') \to.
\]
By assumption, the third term of the above distinguished triangle is zero. So   

$F_1 (Q(A)\otimes Q(A')) \simeq  Q \circ f_1 (A\otimes A')=F( Q A) \otimes F(QA')$.
\end{proof}

\begin{lem}
\label{reflective}
Assume that $\A$ and $\mathcal B$ are tensor triangulated categories and $\T_{\A}$ and $\T_{\mathcal B}$ are thick tensor ideals.
Assume, further,  that $Q_! (A\otimes Q(A'))\simeq Q_!(A)\otimes A'$ for any $A$, $A'$ in $\A$.
Suppose that  $h_1: \mathcal B \to \mathcal A$ is an exact functor.
Let $H_1 = Q\circ  h_1 \circ  Q_!$. 
If $Q\circ  h_1 \iota =0$, $Q_! \circ F_1 = f_1 \circ Q_!$,  and $f_1(A\otimes h_1(B)) \simeq f_1 (A) \otimes B$, for $A\in \A$ and $B\in \mathcal B$, then 
\[
F_1 (Q(A) \otimes H_1(Q(B))) \simeq F_1 (Q(A)) \otimes Q(B).
\]
\end{lem}

\begin{proof}
By assumption, $H_1 \circ Q \simeq Q\circ h_1$. So 
\begin{equation*}
\begin{split}
F_1(Q(A)&\otimes H_1(Q(B))) 
= F_1(Q(A)\otimes Q\circ h_1(B)) \simeq Q\circ f_1 \circ Q_! (Q(A) \otimes Q\circ h_1(B))\\
&= Q\circ f_1 (Q_! Q (A) \otimes h_1(B)) = Q\circ f_1 (Q_! Q (A) \otimes h_1(B) \\
&= Q ( f_1 Q_! Q(A) \otimes h_1(B) =Q( Q_! \circ F_1  Q(A) \otimes B) = F_1(Q(A))\otimes Q(B).
\end{split}
\end{equation*}
The lemma follows.
\end{proof}

\section{Convolution product }

\subsection{Framed representation variety}
\label{framed}
Recall from section ~\ref{Cartan} that  $\Gamma=(I, H, ', '', \bar\empty \; )$ is  a graph  attached to a Cartan datum ($I, \cdot$ ).  
Fix an orientation  $\Omega$ of $\Gamma$, i.e., $\Omega$ is a subset of $H$ such that $\Omega\sqcup \bar \Omega =H$. 
We call the pair $(\Gamma, \Omega)$ a $quiver$. 
To an $I$-graded vector space $V=\oplus_{i\in I} V_i$ over $k$, we set
\[
G_V=\prod_{i\in I}  \mrm{GL}(V_i),
\]
the product of the general linear groups $\mrm{GL}(V_i)$. 
To a pair ($D$, $V$) of finite dimensional  $I$-graded vector spaces over the field $k$, attached the $framed$ $representation$ $variety$  of the quiver $(\Gamma, \Omega)$:
\[
\mbf E_{\Omega}(D, V) = \oplus_{h\in \Omega} \mrm{Hom} ( V_{h'} , V_{h''}) \oplus \oplus_{i\in I} \mrm{Hom} (V_i, D_i).
\] 
Elements in $\mbf E_{\Omega}(D, V)$ will be denoted by $X=(x,q)$ where $x$ (resp. $q$) is in the first (resp. second)  component.
The group $G_D\times G_V$ acts on $\mbf E_{\Omega}(D, V)$ by conjugation:
\[
(f, g) .(x, q)  =(x', q'), \quad \mbox{where} \; x'_{h} = g_{h''} x_h g_{h'}^{-1}, \;  q'_i= f_i q_i g_i^{-1},  \quad  \forall h\in \Omega, i\in I,
\]
for any $(f, g)\in G_D\times G_V$,   and $(x, q) \in \mbf E_{\Omega}(D, V)$.
To each $i\in I$, we set
\[
X(i)=q_i+\sum_{h\in \Omega: h'=i} x_h : V_i \to D_i\oplus \bigoplus_{h\in \Omega: h'=i} V_{h''} .
\]
To a triple $(D, V, V')$ of $I$-graded vector spaces, we set
\[
 \mbf E_{\Omega}(D, V, V') = \mbf E_{\Omega}(D,V) \oplus  \mbf E_{\Omega}(D, V').
\]
We write $\mbf E_{\Omega}$ for $\mbf E_{\Omega} (D, V, V')$ if it causes no ambiguity. 
The group 
\[
\G=G_D\times G_V \times G_{V'}\]
acts on $\mbf E_{\Omega}$ 
by
$(f, g, g'). (X, X')=( (f, g). X, (f, g') . X')$, for any $ (f, g, g')\in \G, (X, X') \in \mbf E_{\Omega}$.

Similarly, to a quadruple $(D, V, V', V'')$ of $I$-graded vector spaces, we set
\[
\mbf E_{\Omega}(D, V, V', V'')= \mbf E_{\Omega}(D, V) 
\oplus  \mbf E_{\Omega}(D, V') \oplus \mbf E_{\Omega}(D, V'').
\]
The group $\mbf H=G_D\times G_V \times G_{V'}\times G_{V''}$ acts on 
$\mbf E_{\Omega}(D, V, V', V'')$ by 
\[
(f, g, g', g''). (X, X', X'')=( (f, g). X, (f, g') . X', (f, g'').X''), 
\]
for any 
$(f, g, g', g'')\in \mbf H, (X, X', X'') \in \mbf E_{\Omega}(D, V, V', V'')$.

\subsection{Fourier-Deligne transform}
\label{Fourier}
 
Let $\Omega'$ be another orientation of the graph $\Gamma$. 
The various varieties defined in ~\ref{framed} can be defined with respect to $\Omega'$ and $\Omega\cup \Omega'$. 
Define a pairing $u_s: \mbf E_{\Omega\cup \Omega'}(D, V^s)\to k$ by 
$u_s(X^s) =\sum_{h\in \Omega\backslash \Omega'} \mbox{tr}(x^s_hx^s_{\bar h})$ for any $X^i\in \mbf E_{\Omega\cup\Omega'}(D, V^s)$
where $\mbox{tr}(-)$ is the trace of the endomorphism in the parenthesis.
Fix  a non-trivial character $\chi$ from the field $ \mbb F_p$ of $p$ elements to  
$ \Q_l^*:=\Q_l\backslash \{0\}$.
Denote by $\mathcal L_{\chi}$  the local system on $k$ corresponding to $\chi$.
Let 
\begin{equation}
\label{L}
\mathcal L_s =u_s^* \mathcal L_{\chi}.
\end{equation}
Let $u_{st}: \mbf E_{\Omega \cup \Omega'}(D, V^i, V^j)\to k$  be the pairing defined by $u_{st} (X^s, X^t) =-u_s(X^s) + u_t(X^t)$ 
for any $(X^s, X^t)\in \mbf E_{\Omega \cup \Omega'}(D, V^s, V^t)$. 
Via this pairing, we may regard $\mbf E_{\Omega'}$ as the dual bundle of the vector bundle $\mbf E_{\Omega}$ over $\mbf E_{\Omega \cap \Omega'}$.
Note that this pairing is $\G$-invariant, i.e., $u_{st}( g. (X^s, X^t))= u_{st}  (X^s, X^t)$, for any  $g\in \G$.
We set
\begin{equation}
\label{L2}
\mathcal L_{st}=u_{st}^* \mathcal L_{\chi}.
\end{equation}
Consider the diagram 
\[
\begin{CD}
\mbf E_{\Omega}(D, V^s, V^t) @<m_{st} << \mbf E_{\Omega\cup \Omega'} (D, V^s, V^t) @>m_{st}'>> \mbf E_{\Omega'}(D, V^s, V^t),
\end{CD}
\]
where the morphisms are obvious projections.
It is clear that $m_{st}$ and $m'_{st}$ are $\G$-equivariant morphisms.
The $Fourier$-$Deligne$ $trans\! f\!orm$ for the vector bundle $\mbf E_{\Omega} $ over $\mbf E_{\Omega\cap \Omega'}$, associated with the character $\chi$, is the triangulated functor
\begin{equation}
\label{Fourier-functor}
\Phi_{\Omega}^{\Omega'}: \D^b_{\G}(\mbf E_{\Omega}(D, V^s, V^t) )\to \D^b_{\G}(\mbf E_{\Omega'}(D, V^s, V^t))
\end{equation}
is defined to be $\Phi_{\Omega}^{\Omega'}(K) =m'_{st!} (m_{st}^*(K) \otimes\mathcal L_{st})[r_{st}]$, where $r_{st}$ is the rank of the vector bundle 
$\mbf E_{\Omega}\to \mbf E_{\Omega\cap \Omega'}$. Note that $r_{st}=\sum_{h\in \Omega\backslash \Omega'} \dim V^s_{h'} \dim V^s_{h''} + \dim V^t_{h'} \dim V^t_{h''}$.

Let $a$ be the map of multiplication by $-1$ along the fiber of the vector bundle $\mbf E_{\Omega} $ over $\mbf E_{\Omega\cap \Omega'}$ . 

\begin{thm}
The transform $\Phi_{\Omega}^{\Omega'}$ is an equivalence of triangulated categories.  Moreover, 
$\Phi_{\Omega'}^{\Omega} \Phi_{\Omega}^{\Omega'} \simeq a^*$.
\end{thm}

\subsection{Localization}
\label{general-localization}
To each $i\in I$, we fix an orientation $\Omega_i$ of the graph $\Gamma$ such that $i$ is a $source$, 
i.e., all arrows $h$ incident to $i$ having $h'=i$.
Let $F_i$ be the closed subvariety of $\mbf E_{\Omega_i}$ consisting of all elements $(X, X')$ such that either
$X(i)$ or $X'(i)$ is not injective.  Let $U_i$ be its complement. Thus we have a decomposition
\begin{align}
\label{partition}
\gamma_i: F_i \hookrightarrow \mbf E_{\Omega_i} \hookleftarrow U_i \;:\beta_i.
\end{align}
Notice that $F_i$ and $U_i$ are $\G$-invariant.
Following Zheng ~\cite{Zheng08}, let $\mathcal  N_i$ be the thick subcategory  of $\D^b_{\G}(\mbf E_{\Omega})$
generated by the objects
$K\in \D^b_{\G}(\mbf E_{\Omega})$ such that  the support of the complex
$\Phi_{\Omega}^{\Omega_i} (K) $ is contained in the subvariety $F_i$. 
Let $\mathcal N$ be the thick subcategory of  $\D^b_{\G}(\mbf E_{\Omega})$ generated by $\mathcal N_i$ for all $i\in I$.
We define 
\[
 \mathscr D^b_{\G} (\mbf E_{\Omega}) = \D^b_{\G}(\mbf E_{\Omega})/ \mathcal N
\]
to be the localization of $\D^b_{\G} (\mbf E_{\Omega})$ with respect to the thick subcategory $\mathcal N$ (\cite{Verdier76}, ~\cite{KS90}).
Let 
\[
Q: \D^b_{\G}(\mbf E_{\Omega}) \to  \mathscr D^b_{\G}(\mbf E_{\Omega}) 
\]
denote the localization functor. 

If $\dim V_i^a > \dim D_i + \sum_{h\in \Omega: h'=i} \dim V^a_{h''}+ \sum_{h\in \Omega: h''=i} \dim V^a_{h'}$ for $a=1$ or $2$ and some $i\in I$, we have
$\mathcal N_i =  \D^b_{\G}(\mbf E_{\Omega})$. In this case, we have $\mathscr D^b_{\G}(\mbf E_{\Omega}) =0$. 

Recall that 
$\Supp (K\otimes L ) = \Supp(K) \cap \Supp(L)$ for any complexes $K$ and $L$. 
From these properties and the fact that $\otimes$ is exact on each variable,  
we have that $\mathcal N$ is a thick tensor ideal of $\D^b_{\G}(\mbf E_{\Omega})$ if $\D^b_{\G}(\mbf E_{\Omega})$ is equipped with the derived tensor functor $\otimes$.
We see that $ \mathscr D^b_{\G} (\mbf E_{\Omega}) $ is a tensor triangulated category with the tensor structure inherited from that of $\D^b_{\G}(\mbf E_{\Omega})$.

\begin{lem}
\label{adjoint}
The localization functor $Q$ admits a right adjoint $Q_*$ and a left adjoint $Q_!$. Moreover, $Q_*$ and $Q_!$ are fully faithful.
\end{lem}

\begin{proof}
Let $Q_i:  \D^b_{\G}(\mbf E_{\Omega}) \to  \D^b_{\G}(\mbf E_{\Omega})/\mathcal N_i$ be the localization functor with respect to the thick subcategory $\mathcal N_i$.   
It is well-known, for example, \cite[1.4]{BBD82},  ~\cite[Theorem 3.4.3]{BL94}, 
that 
the functor 
$\beta_i^*: \D^b_{\G}(\mbf E_{\Omega_i}) \to \D^b_{\G}(U_i)$  admits a  fully faithful right adjoint $\beta_{i*}$ and a fully faithful  left adjoint $\beta_{i!}$. 
Now that the functor $Q_i$ can be identified with the functor $\beta_i^*$ via the transform $\Phi_{\Omega}^{\Omega_i}$, it then admits 
a fully faithful right adjoint  $Q_{i*}$ and a fully faithful left adjoint $Q_{i!}$.

For simplicity, let us assume that the graph $\Gamma$ consists of only two vertices $i$ and $j$.
By the universal property of the localization functor $Q_i$,   the functor $Q$ factors through $Q_i$, i.e., there exists a functor
$\hat Q_i: \D^b_{\G}(\mbf E_{\Omega}) /\mathcal N_i \to \D^b_{\G}(\mbf E_{\Omega})/ \mathcal N$ such that 
$Q=\hat Q_i \circ Q_i$. 
Similarly, we have $Q=\hat Q_j \circ Q_j$ for some $\hat Q_j:  \D^b_{\G}(\mbf E_{\Omega}) /\mathcal N_j \to \D^b_{\G}(\mbf E_{\Omega})/ \mathcal N$.
Let us consider the following diagram
\[
\begin{CD}
\D^b_{\G}(\mbf E_{\Omega})  @>Q_j>>  \D^b_{\G}(\mbf E_{\Omega}) /\mathcal N_j @>\hat Q_j>> \D^b_{\G}(\mbf E_{\Omega}) /\mathcal N\\
@. @VQ_i Q_{j*} VV\\
\empty @. \D^b_{\G}(\mbf E_{\Omega}) /\mathcal N_i.
\end{CD}
\]
Given any $K\in \mathcal N$, we have  $Q_i Q_{j*} Q_j(K)=0$. Indeed, if $K\in \mathcal N_j$, then $Q_j(K) =0$. So $Q_i Q_{j*} Q_j(K)=0$.
If $K\in \mathcal N_i$, we have $\mbox{supp} ( \Phi_{\Omega}^{\Omega_i}  Q_{j*} Q_j(K)) = \mbox{supp}(\Phi_{\Omega}^{\Omega_i} K) \subseteq F_i$. So $Q_{j*} Q_j(K)\in \mathcal N_i$. 
Thus,  $Q_i Q_{j*} Q_j(K)=0$.
If $K_1\to K \to K_2 \to$ is a distinguished triangle such that $K_1$ and $K_2$ are in either $\mathcal N_i$ or $\mathcal N_j$, then $Q_i Q_j Q_{j^*}(K)=0$ because the functor $Q_i Q_j Q_{j^*}$ is exact. 
Since any object in $\mathcal N$ is generated by objects from $\mathcal N_i$  we see that
$Q_i Q_{j*} Q_j(K)=0$ for any $K \in \mathcal N$.

 From this, we see that the functor $Q_i Q_{j*} Q_j$ factors through $Q$, i.e., there is a unique exact functor 
$\hat Q_{i*} : \D^b_{\G}(\mbf E_{\Omega})/ \mathcal N \to \D^b_{\G}(\mbf E_{\Omega})/ \mathcal N_i$ such that 
\begin{align}
\label{factor}
\hat Q_{i*} Q = Q_i Q_{j*} Q_j.
\end{align}
By applying $Q_{j*}$ to (\ref{factor}) and using the fact that $Q_j Q_{j*}\simeq \mbox{id}$, we have  $\hat Q_{i*}\hat Q_j = Q_i Q_{j*}$.

Let $\eta_j: \mbox{Id} \to Q_{j*}Q_j$ and $\epsilon_j: Q_j Q_{j*} \to \mbox{Id}$ be the associated natural transformations of the adjoint pair $(Q_j, Q_{j*})$.
We set $\hat \eta_i: \mbox{Id} \to \hat Q_{i*} \hat Q_i$ be the composition 
\[
Q_i  \overset{Q_i \eta_j}{\to}  Q_i Q_{j*} Q_j = \hat Q_{i*} \hat Q_j Q_j = \hat Q_{i*} \hat Q_i Q_i,
\]
and $\hat \epsilon_i: \hat Q_i \hat Q_{i*} \to \mbox{Id}$ be the composition 
\[
\hat Q_i \hat Q_{i*} \hat Q_j = \hat Q_i Q_i Q_{j*} = \hat Q_j Q_j Q_{j*}  \overset{\hat Q_j \epsilon_j}{\to} \hat Q_j. 
\]
The fact that $(Q_j, Q_{j*})$ is an adjoint pair implies that the compositions of natural transformations
\begin{align}
\label{Qj*}
Q_{j*} \overset{\eta_j Q_{j*} } {\to } Q_{j*} Q_j Q_{j*} \overset{Q_{j*} \epsilon_j} {\to} Q_{j*}, \quad
Q_j \overset{Q_j \eta_j} {\to} Q_j Q_{j*} Q_j \overset{\epsilon_j Q_j}{\to} Q_j
\end{align}
are the identity transformation $\mbox{Id}: Q_{j*} \to Q_{j*}$ and $\mbox{id} : Q_j \to Q_j$, respectively. 
By applying $Q_i$ to the first term in (\ref{Qj*}) and using the definitions of $\hat \epsilon_i$ and $\hat \eta_i$, we get that the following composition of natural transformations is again identity transformation
\[
\begin{CD}
Q_i Q_{j*} @> Q_i \eta_j Q_{j*} >> Q_iQ_{j*}Q_j Q_{j*}  @> Q_iQ_{j*}\epsilon_j >> Q_iQ_{j*}\\
@| @| @|\\
\hat Q_{i*} \hat Q_j   @>  \hat \eta_i  \hat Q_{i*} \hat Q_j >> \hat Q_{i*} \hat Q_i \hat  Q_{i*}  \hat Q_{j}  @> \hat Q_{i*} \hat \epsilon_i  \hat Q_j >> \hat Q_{i*} \hat Q_j,
\end{CD}
\]
where the vertical middle term comes from the following identity:
\[
Q_iQ_{j*}Q_j Q_{j*}= \hat Q_{i*} \hat Q_j Q_j  Q_{j*} =  \hat Q_{i*} \hat Q_i Q_i Q_{j*} = \hat Q_{i*} \hat Q_i \hat Q_{i*} \hat Q_j .
\]
 In other words, we have 
  \begin{align}
  \label{adjunction}
   ( \hat Q_{i*} \hat \epsilon_i )(\hat \eta_i  \hat Q_{i*} )= \mbox{Id}_{\hat Q_{i*}}, \quad
   (\hat \epsilon_i \hat Q_i ) ( \hat Q_i \hat \eta_i) = \mbox{Id}_{\hat Q_i},
   \end{align}
where the second identity can be obtained  by applying $\hat Q_j$ to the second term in (\ref{Qj*}).
It is well known, for example ~\cite{MacLane},  that the pair $(\hat Q_i ,  \hat Q_{i*})$ together with the transformations $(\hat \epsilon_i, \hat \eta_i)$  and the above two identities (\ref{adjunction})  make  the pair $(\hat Q_i ,  \hat Q_{i*})$ an adjoint pair.
Moreover,
the fact that $\epsilon_j$ is a natural isomorphism implies that $\hat \epsilon_i$ is a natural isomorphism, 
which is equivalent to say that $\hat Q_{i*}$ is fully faithful.
The existence and the fully-faithfulness of the left adjoint $\hat Q_{i!}$ of $\hat Q_i$ can be proved in a similar way.

Since $Q_{i*}$ and $\hat Q_{i*}$ are fully faithful right adjoints of $Q_i$ and $\hat Q_i$, respectively,  the composition  functor $Q_* = Q_{i*} \circ \hat Q_{i*}$ is the fully faithful right adjoint of $Q= \hat Q_i Q_i$. 
Similary, $ Q_!= Q_{i!} \circ \hat Q_{i!}$  is  the fully faithful left adjoint of $Q$. The lemma is proved for the graph with only two vertices $i$ and $j$.

In general,  let us order the vertex set $I$ as $i_1, \cdots, i_n$. Define a sequence of thick subcategories:
\[
\mathcal N_1 \subseteq  \cdots \subseteq \mathcal N_m \subseteq \cdots \subseteq  \mathcal N_n=\mathcal N,
\]
where $\mathcal N_m$ is the thick subcategory generated by the subcategories $\mathcal N_{i_1}, \cdots, \mathcal N_{i_m}$. 
Then the functor $Q$ is the composition  of the following functors
\[
\D^b_{\G}(\mbf E_{\Omega}) \to \D^b_{\G}(\mbf E_{\Omega})/\mathcal N_1  \to \cdots \to
\D^b_{\G}(\mbf E_{\Omega})/\mathcal N_m \overset{Q_m}{\to} \D^b_{\G}(\mbf E_{\Omega})/\mathcal N_{m+1} 
\to \cdots \to \D^b_{\G}(\mbf E_{\Omega})/\mathcal N.
\]
Each $Q_m$ admits a fully faithful right adjoint $Q_{m*}$ and a fully faithful left adjoint $Q_{m!}$ due to the fact that the functor 
$Q_{i_m}: \D^b_{\G}(\mbf E_{\Omega}) \to \D^b_{\G}(\mbf E_{\Omega})/\mathcal N_{i_m}$ admits fully faithful right and left adjoint functors. 
This can be proved inductively in a similar manner as the proof in the case when the graph has only two vertices, 
with the pair $(\mathcal N_i, \mathcal N_j)$ replaced by $(\mathcal N_m, \mathcal N_{i_{m+1}})$.
From this observation, we see that  the functors
\[
Q_*= Q_{1*} \circ \cdots \circ Q_{n*} \quad \mbox{and} \quad Q_!= Q_{1!}\circ \cdots \circ Q_{n!}
\]
are the fully faithful right and left adjoint functors of $Q$, respectively.
\end{proof}

What did not state in the above proof of Lemma ~\ref{adjoint} for $\Gamma =\{i, j\}$  is that $\hat Q_j$ admits a fully faithful right adjoint $\hat Q_{j*}$ and the composition $Q_{j*} \hat Q_{j*}$ is the right adjoint of $Q$.  By the uniqueness of the right adjoint of $Q$, we have 
$ Q_{j*} \hat Q_{j*} = Q_{i*} \hat Q_{i*}$.
Now we have 
 \begin{align*}
 Q_{j*} \hat Q_{j*} Q = Q_{j*} \hat Q_{j*} \hat Q_i Q_i = Q_{j*} Q_j Q_{i*} Q_i,\quad
 Q_{i*} \hat Q_{i*} Q = Q_{i*} \hat Q_{i*} \hat Q_j Q_j = Q_{i*} Q_i Q_{j*} Q_j.
\end{align*}
So $Q_* Q(K)  = Q_{j*} Q_j \circ Q_{i*} Q_i(K) =Q_{i*} Q_i \circ  Q_{j*} Q_j(K)$.
So it makes sense to state the identities as follows:
\[
Q_* (Q(K)) = \prod_{i\in I} \Phi_{\Omega_i}^{\Omega} \beta_{i*} \beta_i^* \Phi_{\Omega}^{\Omega_i} (K)
\quad \mbox{and} \quad
Q_!(Q(K)) = \prod_{i\in I} \Phi_{\Omega_i}^{\Omega} \beta_{i!} \beta_i^* \Phi_{\Omega}^{\Omega_i} (K)
\]
for any $K\in  \D^b_{\G}(\mbf E_{\Omega})$.  It is clear that such an expression works for an arbitrary graph too.
We also have
\begin{equation}
\label{Q-projection}
Q_! (A\otimes Q(B)) \simeq Q_!(A) \otimes B, \quad \forall A\in \D^b_{\G}(\mbf E_{\Omega})/\mathcal N, B\in  \D^b_{\G}(\mbf E_{\Omega}).
\end{equation}
This is because each pair ($Q_{a!}, Q_{a}$) in the proof of Lemma ~\ref{adjoint} has such a property.
Moreover,
\begin{lem}
\label{B}
\begin{enumerate}
\item[(a)] The inclusion $\iota: \mathcal N \to \D^b_{\G} (\mbf E_{\Omega})$ admits a left adjoint $\iota^*$ and a right adjoint $\iota^!$.
\item[(b)] One has $Q \iota=0$, $\iota^* Q_!=0$, $\iota^! Q_{*}=0$, and,  $\forall A\in\mathcal N$, $B\in \D^b_{\G}(\mbf E_{\Omega})/\mathcal N$,
         \[\mrm{Hom}(Q_! B, \iota A)=0 \quad  \mbox{and} \quad \mrm{Hom}(\iota A, Q_*B)=0;\]
\item[(c)] For any $K\in \D^b_{\G}(\mbf E_{\Omega})$, there are distinguished triangles
        \[ Q_! Q(K) \to K \to \iota \iota^*(K) \to \quad \mbox{and} \quad \iota \iota^! (K) \to K \to Q_* Q(K)\to;\]
\item[(d)] The functors $\iota, Q_*, Q_!$ are fully faithful, i.e.,  the following adjunction  are isomorphic:
        \[ \iota^* \iota \to \mbox {Id} \to \iota^! \iota \quad \mbox{and} \quad QQ_* \to \mbox{Id} \to Q Q_!.\]                 
\end{enumerate}
\end{lem}

\begin{proof}
These results  follow from Lemma ~\ref{adjoint} and results from ~\cite[Prop. 6.5, 6.6, 6.7]{Verdier76} and ~\cite[Ch. 7, 10]{KS06}.
\end{proof}

\begin{rem} 
(1). Comparing Lemma ~\ref{adjoint}-\ref{B} with ~\cite[1.4.3]{BBD82}, ~\cite{BL94}, we may regard $\D^b_{\G}(\mbf E_{\Omega})/\mathcal N$ as the equivariant derived category
$\D^b_{\G}(U) $ of  an ``imaginary'' open subvariety $U$ of $\mbf E_{\Omega}$. It is not clear to the author if the existence of such a variety can be justified by the method in
~\cite{B05}.

(2). One may define $\mathscr D^*_{\G} (\mbf E_{\Omega} (D, V, V'))$ for $*=\{ +, -\}$ in a similar way and Lemmas ~\ref{adjoint}, ~\ref{B} and (\ref{Fourier-category}) still hold,
where the thick subcategory $\mathcal N_i$ can be defined to be the one consisting of all objects $K$ such that the support of the cohomology sheaf $H^a(K)$ is in $F_i$ for any $a$.
\end{rem}

\subsection{Convolution product}
\label{convolution}

Let 
\[
p_{st} : \mbf E_{\Omega} (D, V^1, V^2, V^3)  \to \mbf E_{\Omega} (D, V^s, V^t), \quad \phi: \mbf H \to \G,
\]
denote the projections to the $(s, t)$-components,  where the spaces $\mbf E_{\Omega}$'s,  the groups $\mbf H$  and $\G$ are defined in  Section ~\ref{framed}.
It is clear that  $p_{st}$ is a $\phi$-map. 
So, by Section ~\ref{generaldirectimage} (\ref{general-f}), we have a morphism of algebraic stacks:
\[
Qp_{st} : [\mbf H\backslash  \mbf E_{\Omega} (D, V^1, V^2, V^3)]   \to [\G\backslash \mbf E_{\Omega} (D, V^s, V^t)].
\]
From  Section ~\ref{generaldirectimage} (\ref{inducedfunctors}), we have the following functors:
\begin{align}
\begin{split}
&Qp_{st}^*: \D^b_{\G}(\mbf E_{\Omega}(D, V^s, V^t)) \to \D^b_{\mbf H} (\mbf E_{\Omega}(D, V^1, V^2, V^3));\\
&(Qp_{st})_!: \D^-_{\mbf H}(\mbf E_{\Omega}(D, V^1, V^2, V^3)) \to \D^-_{\G}(\mbf E_{\Omega}(D, V^s, V^t)).
\end{split}
\end{align}

\begin{lem}
We have 
\begin{equation}
\label{convolution-A'}
Q\circ Qp_{st}^*\circ \iota=0 
\quad \mbox{and} \quad 
\iota^*\circ Qp_{st!} \circ Q_!=0.
\end{equation}
\end{lem}

\begin{proof}
For the first identity, it suffices to show that 
$Qp^*_{st} (\mathcal N_{i}) \subseteq \mathcal N_i$ for any $i\in I$. Fix an orientation $\Omega_i$ such that $i$ is a source in $\Omega_i$. 
The morphism $p_{st}$ is a morphism of vector bundles over the base space $\mbf E_{\Omega\cap \Omega_i}(D, V^s, V^t)$. 
Let $p'_{st}: \mbf E_{\Omega} (D, V^s, V^t)' \to  \mbf E_{\Omega} (D, V^1, V^2, V^3)'$ be the transpose of $p_{st}$. 
We identify $ \mbf E_{\Omega} (D, V^s, V^t)'$ with $ \mbf E_{\Omega_i} (D, V^s, V^t)$ under the pairing $u_{st}$
and $\mbf E_{\Omega} (D, V^1, V^2, V^3)'$ with $ \mbf E_{\Omega_i} (D, V^s, V^t) \times \mbf E_{\overline \Omega} (D, V^u)$ where $u$ is the number  such that
$\{s, t, u\} =\{1, 2, 3\}$. Then the transpose $p_{st}'$ is noting but the inclusion of 
$ \mbf E_{\Omega_i} (D, V^s, V^t)$ into $  \mbf E_{\Omega_i} (D, V^s, V^t) \times \mbf E_{\overline \Omega} (D, V^u)$. 
Let $\psi: \G\to \mbf H$ be the obvious imbedding. It is clear that
$p'_{st}$ is a $\psi$-map.  This induces a morphism of algebraic stacks:
\[
Qp_{st}': [\G \backslash  \mbf E_{\Omega_i} (D, V^s, V^t)] \to [\mbf H \backslash \mbf E_{\Omega_i} (D, V^s, V^t) \times \mbf E_{\overline \Omega} (D, V^u)].
\]
Let $\mbf H$ act on $\mbf E_{\Omega_i} (D, V^s, V^t)$ by declaring that $G_{V^u}$ acts trivially. Then 
$Qp_{st}'$ factors through $[\mbf H\backslash  \mbf E_{\Omega_i} (D, V^s, V^t)] $. So the functor 
$Qp'_{st!}$ is a composition of  the integration functor  
\[
\mbox{Ind}_!: \D_{\G} ( \mbf E_{\Omega_i} (D, V^s, V^t)) \to \D_{\mbf H}( \mbf E_{\Omega_i} (D, V^s, V^t) )
\]
 from ~\cite[3.7]{BL94} and  the functor 
 \[
 p'_{ij!}: \D_{\mbf H}( \mbf E_{\Omega_i} (D, V^s, V^t) ) \to \D_{\mbf H}(  \mbf E_{\Omega_i} (D, V^s, V^t) \times \mbf E_{\overline \Omega} (D, V^u)).
 \]

Since $G_{V^u}$ acts trivially on $\mbf E_{\Omega_i} (D, V^s, V^t)$, we see that the supports of the complexes $K$ and $\mbox{Ind}_!(K)$ are the same. 
Moreover, we have $\mbox{supp} (p'_{ij!} (K) ) = p'_{ij} \mbox{supp} (K)$.  
So  the decomposition $p'_{st!} \mbox{Ind}_!$ of $Qp'_{st!}$ implies that 
\begin{align}
\label{convolution-A'-a}
Qp'_{st!} (\mathcal N_i) \subseteq \mathcal N_i. 
\end{align}
Note that the notion of $\mathcal N_i$  also makes sense on $\D_{\mbf H}  ( \mbf E_{\Omega_i} (D, V^s, V^t) \times \mbf E_{\overline \Omega} (D, V^u))$.  
Now we have
\begin{align}
\label{convolution-A'-b}
Qp_{st}^* (\Phi_{\Omega'}^{\Omega} K)  = \Phi_{\Omega'}^{\Omega}  Qp'_{st!} (K) [f_1-f_2],
\end{align}
where $f_1$ and $f_2$ are the ranks of $  \mbf E_{\Omega} (D, V^s, V^t)$ and $ \mbf E_{\Omega} (D, V^1, V^2, V^3)$ over $  \mbf E_{\Omega\cap \Omega_i} (D, V^s, V^t)$, respectively.  
This  can be proved by a similar argument as the proof of Theorem 1.2.2.4 in ~\cite{L87}, see also  ~\cite[Theorem 13.2]{KW01}.
By combining (\ref{convolution-A'-a}) and (\ref{convolution-A'-b}), we get the statement  $Qp^*_{st} (\mathcal N_{i}) \subseteq \mathcal N_i$ for any $i\in I$.

The second identity can be shown similarly
 by replacing (\ref{convolution-A'-a}) by
$\iota_i^* (Qp'_{st})^* Q_{i!} =0$ where  $\iota_i^*$ is the left adjoint of the inclusion functor $\iota_i$ of $\mathcal N_i$ into the ambient category and 
$Q_{i!}$ is the left adjoint to the quotient functor $Q_i$ with respect to $\mathcal N_i$.
\end{proof}

We set
\begin{align}
\label{convolution-functors}
\begin{split}
&P_{st}^* = Q\circ Qp_{st}^* \circ Q_! : 
\mathscr D^b_{\G}(\mbf E_{\Omega}(D, V^s, V^t)) \to \mathscr D^b_{\mbf H} (\mbf E_{\Omega}(D, V^1, V^2, V^3));\\
& P_{st!}= Q \circ (Qp_{st})_!  \circ Q_!: \mathscr D^-_{\mbf H}(\mbf E_{\Omega}(D, V^1, V^2, V^3)) \to \mathscr D^-_{\G}(\mbf E_{\Omega}(D, V^s, V^t)).
\end{split}
\end{align}

By  Lemma ~\ref{B} (c) and (\ref{convolution-A'}) in the above Lemma, we have 
\begin{lem}
\label{convolution-A}
$P_{st}^* \circ  Q =Q\circ Qp_{st}^*$ and  $Q_! P_{st!} = Qp_{st!} Q_!$. 
\end{lem}

To any objects $K \in \mathscr D^-_{\G}(\mbf E_{\Omega} (D, V^1, V^2))$ and $L \in \mathscr D^-_{\G}(\mbf E_{\Omega}(D, V^2, V^3))$, we associate
\begin{equation}
\label{cdot}
K \cdot L= P_{13!} (P^*_{12}(K) \otimes P^*_{23}(L)) \quad \in \mathscr D^-_{\G}(\mbf E_{\Omega}(D, V^1, V^3)).
\end{equation}
If, in addition, $M \in \mathscr D^-_{\G}(\mbf E_{\Omega} (D, V^3, V^4))$, we have

\begin{prop}
\label{associative}
$(K\cdot L)\cdot M \simeq K \cdot (L\cdot M)$.
\end{prop}

\begin{proof}
Let 
$X_s= \mbf E_{\Omega} (D, V^s)$
and  $G_s=G_{V^s}$, $ \forall s=1, 2, 3, 4$.
Let 
\begin{align*}
\begin{split}
&q_{st}: X_1 \times X_3 \times X_4\to X_s\times X_t, \quad  \phi_{st}: G_D\times \prod_{a=1, 3, 4} G_a \to G_D\times \prod_{a=s, t} G_a;\\ 
&r_{stu}: X_1\times X_2\times X_3\times X_4 \to X_s\times X_t\times X_u, \quad \phi_{stu}: G_D \times \prod_{a=1, 2, 3, 4} G_a \to G_D\times\prod_{a=s, t, u} G_a; \\
&s_{st}: X_1 \times X_2 \times X_3 \times X_4 \to X_i \times X_j, \quad \phi'_{st}:  G_D \times \prod_{a=1, 2, 3, 4 } G_a \to G_D\times G_s\times G_t;
\end{split}
\end{align*}
be the self-explained projections. It is clear that $q_{st}$, $r_{stu}$ and $s_{st}$ are $\phi_{st}$-map, $\phi_{stu}$-map, and $\phi'_{st}$-map, respectively. 
Similar to the functors $P_{st!}$ and $P_{st}^*$, 
we define the functors $Q_{st!}$, $Q_{st}^*$ (resp. $R_{stu!}$, $R_{stu}^*$; $S_{st!}$, $S_{st}^*$) for the map
$q_{st}$ (resp. $r_{stu}$, $s_{st}$). By definition,
\[
(K\cdot L) \cdot M=
Q_{14!} (Q_{13}^* (K\cdot L) \otimes Q_{34}^* (M))
= Q_{14!} ( Q_{13}^* P_{13!} ( P_{12}^*(K) \otimes P_{23}^* (L) ) \otimes Q_{34}^*(M)).
\]
Since the square  $(r_{123}, p_{13}; r_{134}, q_{13})$ is cartesian and by Lemma  ~\ref{general-basechange}, we have 
\begin{equation}
\label{associative-1}
Qq_{13}^* Q p_{13!} \simeq Qr_{134!} Qr_{123}^*.
\end{equation}
By (\ref{convolution-A'}),  we have 
$Q\circ Qr_{123}^*\circ \iota=0$ and $\iota^*\circ Q p_{13!} \circ Q_!=0$.
So by Lemma ~\ref{cartesian} and (\ref{associative-1}),   
\begin{equation}
\label{associative-2}
Q_{13}^* P_{13!} = R_{134!} R_{123}^*.
\end{equation}
Thus 
\[
(K\cdot L) \cdot M\simeq Q_{14!} ( R_{134!} R_{123}^* ( P_{12}^*(K) \otimes P_{23}^* (L) ) \otimes Q_{34}^*(M)).
\]
By Lemma ~\ref{general-projection}, Lemma ~\ref{reflective}, and the fact that $Q\circ  Qr_{134}^* \iota=0$,  
\begin{equation}
\label{associative-3}
(K\cdot L) \cdot M
\simeq Q_{14!}  R_{134!} (R_{123}^* ( P_{12}^*(K) \otimes P_{23}^* (L) ) \otimes  R_{134}^* Q_{34}^*(M)).
\end{equation}
By Lemma ~\ref{composition}, Lemma ~\ref{localization-composition}, and the fact that $\iota^* \circ Qr_{134!}  Q_! =0$, we have
\begin{equation}
\label{associative-4}
Q_{14!}  R_{134!} = S_{14!}.
\end{equation}
By Lemma ~\ref{general-distribution}, Lemma ~\ref{distribution} and the fact that $Q\circ Qr_{123}^* \iota =0$, we have
\begin{equation}
\label{associative-5}
R_{123}^* (P_{12}^* K \otimes P_{23}^* L) = R_{123}^* P_{12}^* K \otimes R_{123}^* P_{23}^* L.
\end{equation}
By Lemma ~\ref{general-composition} and  Lemma ~\ref{localization-composition}, we have 
\begin{equation}
\label{associative-6}
R_{123}^* P_{12}^* =S_{12}^*,\quad 
R_{123}^* P_{23}^*=S_{23}^*,\quad \mbox{and} \quad 
R_{134}^* Q_{34}^*=S_{34}^*.
\end{equation}
Combining (\ref{associative-3})-(\ref{associative-6}), we get
\begin{align*}
\begin{split}
(K\cdot L) \cdot M
&\simeq S_{14!} ( S_{12}^*(K) \otimes S_{23}^*(L) \otimes S_{34}^*(M)).
\end{split}
\end{align*}
Similarly, we can show that 
$K\cdot (L\cdot M) \simeq S_{14!} ( S_{12}^*(K) \otimes S_{23}^*(L) \otimes S_{34}^*(M))$.
The lemma follows.
\end{proof}

\begin{rem}
The proofs of the equalities (\ref{associative-2})-(\ref{associative-6}) can be generalized to similar situations.  
In the sequel, we simply state the similar equalities without mentioning how they are proved. 
\end{rem}

\subsection{Independence of choice of orientation}

Let $\Omega'$ be another orientation of the graph $\Gamma$. 
Let $\mathcal N_{\Omega'}$ be the thick subcategory of $  \D_{\G} (\mbf E_{\Omega'})  $ defined in the same way as $\mathcal N$.
One has, by definition, $\Phi_{\Omega}^{\Omega'} (\mathcal N) = \mathcal N_{\Omega'}$.  So we have an equivalence, 
induced by   $\Phi_{\Omega}^{\Omega'}$,
\begin{align}
\label{Fourier-category}
\Phi_{\Omega}^{\Omega'}: \mathscr D^b_{\G}(\mbf E_{\Omega})  \simeq \mathscr D^b_{\G}(\mbf E_{\Omega'}).
\end{align}

Similarly, one can define the category
$\mathscr D^b_{\mbf H} (\mbf E_{\Omega} (D, V^1, V^2, V^3))$
and the equivalence of categories
$\Phi_{\Omega}^{\Omega'}: \mathscr D^b_{\mbf H}(\mbf E_{\Omega}(D, V^1, V^2, V^3))  \simeq \mathscr D^b_{\mbf H}(\mbf E_{\Omega'}(D, V^1, V^2, V^3))$.

We  define a  convolution product
``$\cdot_{\Omega'}$'' on the categories  $\DD^-_{\G}(\mbf E_{\Omega'}(D, V^s, V^t))$ similar to the one in Section ~\ref{convolution}.  
The following proposition shows that the convolution products are compatible with the Fourier-Deligne transform.

\begin{prop}
\label{Fourier-convolution}
$\Phi_{\Omega}^{\Omega'} (K\cdot L) = \Phi_{\Omega}^{\Omega'} (K) \cdot_{\Omega'} \Phi_{\Omega}^{\Omega'}(L)$,
for any objects $K \in \mathscr D^-_{\G}(\mbf E_{\Omega} (D, V^1, V^2))$ and $L \in \mathscr D^-_{\G}(\mbf E_{\Omega}(D, V^2, V^3))$.
\end{prop}

\begin{proof}
Due to the fact that $m_{13}^*$ is a fully faithful functor and that the condition to define the thick subcategory $\mathcal N$ on 
$\mbf E_{\Omega}$ and $\mbf E_{\Omega\cup \Omega'}$ are the same, one can deduce that $\iota^* m_{13}^* Q_!=0$. From this fact,   
the functor $\Phi_{\Omega}^{\Omega'}: \DD^-_{\G}(\mbf E_{\Omega} (D, V^1, V^3)) \to \DD^-_{\G}(\mbf E_{\Omega'}(D, V^1, V^3))$ can be rewritten as
\[
\Phi_{\Omega}^{\Omega'} (K) =M'_{13!} (M_{13}^* (K) \otimes \mathcal L_{13})[r_{13}].
\]
where  the notations $\mathcal L_{13}$ and $r_{13}$ are defined in Section ~\ref{Fourier} and $M'_{13!}$ and $M_{13}^*$ are obtained from $m_{13!}'$ and $m_{13}^*$ as in
(\ref{convolution-functors}).
By definition, we have 
\begin{equation*}
\Phi_{\Omega}^{\Omega'} (K\cdot L) =\Phi_{\Omega}^{\Omega'} (P_{13!} (P_{12}^*(K) \otimes P_{23}^* (L)))
=M'_{13!} (M_{13}^* P_{13!} (P_{12}^*(K) \otimes P_{23}^* (L)) \otimes \mathcal L_{13})[r_{13}].
\end{equation*}
Consider the following cartesian diagram
\[
\begin{CD}
\mbf E_{\Omega \cup \Omega'} (D, V^1, V^3) \times \mbf E_{\Omega} (D, V^2) @>s>> \mbf E_{\Omega\cup \Omega'}(D, V^1, V^3)\\
@VrVV @Vm_{13} VV\\
\mbf E_{\Omega}(D, V^1, V^2, V^3) @>p_{13}>> \mbf E_{\Omega}(D, V^1, V^3).
\end{CD}
\]
By an argument similar to the proof of (\ref{associative-2}), we have $M_{13}^* P_{13!} =S_! R^*$. So
\begin{equation}
\label{Fourier-convolution-LHS}
\begin{split}
\Phi_{\Omega}^{\Omega'} (K\cdot L)&=M'_{13!} (S_! R^*(P_{12}^*(K) \otimes P_{23}^* (L)) \otimes \mathcal L_{13})[r_{13}]\\
&=M'_{13!} S_! (R^*P_{12}^*(K) \otimes R^* P_{23}^* (L) \otimes S^* \mathcal L_{13})[r_{13}].
\end{split}
\end{equation}

On the other hand, we have
\begin{equation*}
\begin{split}
\Phi_{\Omega}^{\Omega'}(K) & \cdot_{\Omega'} \Phi_{\Omega}^{\Omega'} (L) 
=P'_{13!} ((P'_{12})^* (\Phi_{\Omega}^{\Omega'}(K))\otimes (P'_{23})^* (\Phi_{\Omega}^{\Omega'}(L)))\\
&=P'_{13!} ( (P'_{12})^* M'_{12!} (M_{12}^* (K) \otimes \mathcal L_{12})[r_{12}] \otimes (P'_{23})^* M'_{23!} (M_{23}^*(L)\otimes \mathcal L_{23})[r_{23}]),
\end{split}
\end{equation*}
where  $P_{st!}'$ and $(P_{st}')^*$ are obtained from the projections 
$p_{st}'$ from  $\mbf E_{\Omega'}(D, V^1, V^2, V^3)$ to  $\mbf E_{\Omega'}(D, V^s, V^t)$, $\mathcal L_{st}$ and $r_{st}$ are from Section ~\ref{Fourier}, and
the functors $M'_{st!}$ and $M_{st}^*$ are obtained from the projections $m_{st}$ in Section ~\ref{Fourier} as $P_{st!}$ 
and $P_{st}^*$ from $p_{st}$ in (\ref{convolution-functors}).
Consider the following cartesian diagrams
\[
\begin{CD}
\mbf E_{\Omega \cup \Omega'} (D, V^1, V^2) \times \mbf E_{\Omega'}(D, V^3) @>s_1' >> \mbf E_{\Omega'} (D, V^1, V^2, V^3) \quad\\
@Vr'_1 VV @V p'_{12} VV \\
\mbf E_{\Omega \cup \Omega'} (D, V^1, V^2) @>m'_{12} >> \mbf E_{\Omega'} (D, V^1, V^2),
\end{CD}
\]
and
\[
\begin{CD}
\mbf E_{\Omega'}(D, V^1)\times \mbf E_{\Omega\cup \Omega'} (D, V^2, V^3) @>s_2'>> \mbf E_{\Omega'}(D, V^1, V^2, V^3)\\
@Vr'_2VV @Vp'_{23} VV\\
\mbf E_{\Omega\cup \Omega'} (D, V^2, V^3) @>m'_{23}>> \mbf E_{\Omega'}(D, V^2, V^3).
\end{CD}
\]
From these cartesian diagrams and similar to the proof of  (\ref{associative-2}), we have
$(P_{12}')^* M_{12!}' = S'_{1!} (R'_1)^*$ and 
$(P'_{23})^* M'_{23!} = S'_{2!} (R'_2)^*$.
So 
\begin{equation*}
\begin{split}
\Phi&_{\Omega}^{\Omega'}(K)  \cdot_{\Omega'} \Phi_{\Omega}^{\Omega'} (L) =
P'_{13!} ( S'_{1!} (R'_1)^* (M_{12}^* (K) \otimes \mathcal L_{12}) \otimes S'_{2!} (R'_2)^* (M_{23}^*(L)\otimes \mathcal L_{23}))[r_{12}+r_{23}]\\
&=P'_{13!} ( S'_{1!}  ((R'_1)^* M_{12}^* (K) \otimes  (R'_1)^* \mathcal L_{12})\otimes S'_{2!} ((R'_2)^* M_{23}^*(L)\otimes (R'_2)^*( \mathcal L_{23})))[r_{12}+r_{23}]\\
&=P'_{13!} S'_{1!}  ((R'_1)^* M_{12}^* (K) \otimes  (R'_1)^* \mathcal L_{12} \otimes  (S'_1)^*S'_{2!} ((R'_2)^* M_{23}^*(L)\otimes (R'_2)^*( \mathcal L_{23})))[r_{12}+r_{23}].
\end{split}
\end{equation*}
We form the following cartesian diagram
\[
\begin{CD}
\mbf E_{\Omega \cup \Omega'} (D, V^1, V^2,  V^3) 
\times \mbf E^2_{\Omega\backslash \Omega'} (D, V^2)
 @>t'_2>> \mbf E_{\Omega\cup \Omega'}(D, V^1, V^2) \times \mbf E_{\Omega'}(D,V^3)\\
@Vt'_1VV @Vs'_1VV \\
\mbf E_{\Omega'}(D, V^1)\times \mbf E_{\Omega\cup \Omega'} (D, V^2, V^3) @>s'_2>> \mbf E_{\Omega'}(D,V^1, V^2, V^3).
\end{CD}
\]
This cartesian diagram gives rise to the identity, $(S'_1)^* S'_{2!} = T'_{2!} (T'_1)^*$. So
\begin{equation*}
\begin{split}
&\Phi_{\Omega}^{\Omega'}(K)  \cdot_{\Omega'} \Phi_{\Omega}^{\Omega'} (L) =\\
&=P'_{13!} S'_{1!}  ((R'_1)^* M_{12}^* (K) \otimes  (R'_1)^* \mathcal L_{12} \otimes  T'_{2!} (T'_1)^*( (R'_2)^* M_{23}^*(L)\otimes (R'_2)^*( \mathcal L_{23})))[r_{123}]\\
&=P'_{13!} S'_{1!}  T'_{2!} ( (T'_2)^*(R'_1)^* M_{12}^* K  \otimes  (T'_2)^* (R'_1)^* \mathcal L_{12} \otimes(T'_1)^* (R'_2)^* M_{23}^*L\otimes  (T'_1)^*
(R'_2)^* \mathcal L_{23})[r_{123}]\\
&=P'_{13!} S'_{1!}  T'_{2!} 
( (T'_2)^*(R'_1)^* M_{12}^* K \otimes(T'_1)^* (R'_2)^* M_{23}^*L 
\otimes  (T'_2)^* (R'_1)^* \mathcal L_{12}\otimes  (T'_1)^* (R'_2)^* \mathcal L_{23})[r_{123}].
\end{split}
\end{equation*}
where $r_{123}=r_{12}+r_{23}$. Let $\mbf F_1=\mbf E_{\Omega\cup \Omega'}(D, V^1, V^3)$ and 
\[
t'_3: \mbf F\equiv \mbf E_{\Omega \cup \Omega'} (D, V^1, V^2,  V^3) 
\times \mbf E^2_{\Omega\backslash \Omega'} (D, V^2)
 \to \mbf Z\equiv \mbf F_1\times \mbf E_{\Omega}(D, V^2)\times \mbf E^2_{\Omega\backslash \Omega'}(D, V^2)
\]
be the obvious projection.
Note that in the component $\mbf E_{\Omega \cup \Omega'} (D, V^2)$, there is a copy of $\mbf E_{\Omega\backslash \Omega'}(D, V^2)$, 
denoted by $\mbf E^1_{\Omega}(D, V^2)$. 
Observe that  
\[
p'_{13} s'_1 t'_2=w  t'_3, \quad
m_{12} r'_1 t'_2= y_2 t'_3,\quad \mbox{and} \quad 
m_{23} r'_2 t'_1= y_1 t'_3,
\]
where $w$ is the projection from $\mbf Z$ to $\mbf E_{\Omega'}(D, V^1, V^3)$, and $y_1$ and $y_2$ are the projections 
from $\mbf Z$ to  $\mbf E_{\Omega}(D, V^2, V^3)$ (for $y_1$) and $\mbf E_{\Omega}(D, V^1, V^2)$, respectively. 
 Note that there are two choices for the projections for each $y_i$, but we choose the unique one such that the above identities hold.
 So
\begin{equation*}
\begin{split}
\Phi_{\Omega}^{\Omega'}(K)  \cdot_{\Omega'} \Phi_{\Omega}^{\Omega'} (L) 
&=W_! T'_{3!} ( (T'_3)^* Y_2^*  K\otimes (T'_3)^* Y_1^*L \otimes (T'_2)^* (R'_1)^* \mathcal L_{12} \otimes  (T'_1)^* (R'_2)^* \mathcal L_{23})[r_{123}]\\
&=W_! ( Y_2^* K \otimes Y_1^* L \otimes T'_{3!} ((T'_2)^* (R'_1)^* \mathcal L_{12} \otimes  (T'_1)^* (R'_2)^* \mathcal L_{23}))[r_{123}].
\end{split}
\end{equation*}
Let 
$\mbf F_2=\mbf E_{\Omega \cup \Omega'} (D, V^2)\times \mbf E_{\Omega\backslash \Omega'}(D, V^2)$. Thus, 
$\mbf F=\mbf F_1\times \mbf F_2$.
Each component $\mbf E^i_{\Omega}(D, V^2)$ defines a projection,
$\pi_{s,s+1}:\mbf F_2\to \mbf E_{\Omega\cup \Omega'} (D, V^2)$, for any $ s=1, 2$.
We have 
\[
r_1' t_2' =w_1 (1\times \pi_{23}) 
\quad \mbox{and} \quad 
r_2' t_1' =w_2 (1\times \pi_{12}),
\]
where $1: \mbf F_1\to \mbf F_1$ is the identity map and  $w_s$ is the projection of  $\mbf E_{\Omega\cup \Omega'}(D, V^1, V^2, V^3)$ 
to $\mbf E_{\Omega\cup \Omega'}(D, V^s, V^{s+1})$ for any $s=1, 2$. Hence,
\begin{equation*}
\Phi_{\Omega}^{\Omega'}(K)  \cdot_{\Omega'} \Phi_{\Omega}^{\Omega'} (L)=
W_! ( Y_2^* K \otimes Y_1^* L \otimes T'_{3!}
((1\times \Pi_{23})^* W_1^*  \mathcal L_{12} \otimes  (1\times \Pi_{12})^* W_2^* \mathcal L_{23}))[r_{123}].
\end{equation*}
Observe that  $W_1^* \mathcal L_{12} = P^*_1  \mathcal L_1^* \otimes P_2^*  \mathcal L_2$ 
and $W_2^* \mathcal L_{23}=P_3^* \mathcal L_3 \otimes P_2^* \mathcal L_2^*$
where $\mathcal L_s^*$ is the dual of the local system $\mathcal L_s$ in  ~\ref{Fourier}, 
and  $p_s$ are the projections from $\mbf E_{\Omega\cup\Omega'}(D, V^1, V^2,  V^3)$ to $\mbf E_{\Omega\cup \Omega'}(D, V^s)$. So
\begin{equation*}
\begin{split}
T'_{3!}&((1\times \Pi_{23})^* W_1^*  \mathcal L_{12} \otimes  (1\times \Pi_{12})^* W_2^* \mathcal L_{23})=\\
&
=T'_{3!} (((1\times \Pi_{23})^* (P_1^* \mathcal L_1^* \otimes P_2^* \mathcal L_2)  \otimes (1\times \Pi_{12})^*  P_2^*( \mathcal L^*_2 \otimes P_3^*\mathcal L_3))\\
&=T'_{3!} ( (T'_3)^* P^* \mathcal L_{13} \otimes  (1\times \Pi_{23})^*  P_2^* \mathcal L_2\otimes (1\times \Pi_{12})^*  P_2^*( \mathcal L^*_2 ))\\
&=P^* \mathcal L_{13} \otimes T'_{3!} (1\times \Pi_{23})^*  P_2^* \mathcal L_2\otimes (1\times \Pi_{12})^*  P_2^*( \mathcal L^*_2 )),
\end{split}
\end{equation*}
where $P^*$ comes from the projection $p: \mbf Z \to \mbf E_{\Omega\cup \Omega'} (V^1, V^3)$.
By a similar argument as ~\cite[p. 44]{KW01}, we have 
\[
T'_{3!} (1\times \Pi_{23})^*  P_2^* \mathcal L_2\otimes (1\times \Pi_{12})^*  P_2^*( \mathcal L^*_2 ))=
\Delta_! \bar {\mbb Q}_{l, \mbf Z_1} [-2 r'].
\]
where $\Delta_!$ is from the diagonal map 
$\delta: \mbf Z_1\equiv \mbf E_{\Omega\cup \Omega'} (D, V^1, V^3) \times \mbf E_{\Omega}(D, V^2)  \to \mbf Z$ 
and $r'$ is the rank of $t_3'$, which is equal to $\sum_{h\in \Omega\backslash \Omega'} \dim V^2_{h'}\dim V^2_{h''}$. So 
\[
P^* \mathcal L_{13} \otimes T'_{3!} (1\times \Pi_{23})^*  P_2^* \mathcal L_2\otimes (1\times \Pi_{12})^*  P_2^*( \mathcal L^*_2 ))=
P^* \mathcal L_{13} \otimes\Delta_! \bar {\mbb Q}_{l, \mbf Z_1} [-2 r']=
\Delta_! S^* \mathcal L_{13}[-2r'].
\]
Therefore,
\begin{equation}
\label{Fourier-convolution-RHS}
\begin{split}
\Phi_{\Omega}^{\Omega'}(K)  \cdot_{\Omega'} \Phi_{\Omega}^{\Omega'} (L) &=
W_! ( Y_2^* K \otimes Y_1^* L \otimes\Delta_! S^* \mathcal L_{13}) [r_{123}-2r']\\
&=W_! \Delta_! (\Delta^* Y_2^* K\otimes \Delta^* Y_1^* L \otimes S^* \mathcal L_{13}) [r_{13}]\\
&=M'_{13!} S_! (R^*P_{12}^*(K) \otimes R^* P_{23}^* (L) \otimes S^* \mathcal L_{13})[r_{13}],
\end{split}
\end{equation}
where the last identity comes from the observation that $w\delta=m_{13}' s$, $y_2\delta = p_{12} r$ and $y_1 \delta  =p_{23} r$. 
The proposition follows  from (\ref{Fourier-convolution-LHS}) and (\ref{Fourier-convolution-RHS}).
\end{proof}

\section{Defining relation}

\subsection{Generator}
\label{Generator}
Given any pair $(X^1, X^2 ) \in \mbf E_{\Omega}(D, V^1, V^2)$, we write 
\[
\mbox{``}\; X^1\hookrightarrow X^2\;\mbox{''}
\] 
if there exists an $I$-graded inclusion  $\rho: V^1\to V^2$
such that  $\rho_{h''} x^1_h =x^2_h \rho_{h'}, \; q^1_i=q^2_i \rho_i$, for any  $h$ in $\Omega$ and $i$ in $I$.
We also write ``$\rho: X^1\hookrightarrow X^2$''  for such a $\rho$ and ``$X^1\overset{\rho}{\hookrightarrow} X^2$'' for the triple $(X^1, X^2, \rho)$.
Consider the smooth variety
\begin{equation}
\label{Z}
\mbf Z_{\Omega} \equiv \mbf Z_{\Omega}(D, V^1, V^2)=\{ (X^1, X^2, \rho ) | (X^1, X^2) \in \mbf E_{\Omega}(D, V^1, V^2)\; \mbox{and}\; \rho: X^1 \hookrightarrow X^2\}.
\end{equation}
We have a diagram
\begin{equation}
\label{Z-1}
\begin{CD}
\mbf E_{\Omega}(D, V^1) @<\pi_1<< \mbf Z_{\Omega}(D, V^1, V^2) @>\pi_2>> \mbf E_{\Omega}(D, V^2),\\
@|  @V\pi_{12}VV @|\\
\mbf E_{\Omega}(D, V^1)@<p_1<< \mbf E_{\Omega}(D, V^1, V^2) @>p_2>> \mbf E_{\Omega}(D, V^2),
\end{CD}
\end{equation}
 where $\pi_1$ and $p_1$ are projections to the first components,  $\pi_2$ and $p_2$ are the projections to the second components,  
 and $\pi_{12}$ is the projection to $(1, 2)$  components.

Similar to $\mbf Z_{\Omega}(D, V^1, V^2)$, let
$\mbf Z^t_{\Omega}(D, V^1, V^2)=\{(X^1, X^2, \rho) | \rho: X^2\hookrightarrow X^1\}$. 
Let $\pi_{12}$ denote the projection $\mbf Z^t_{\Omega}(D, V^1, V^2) \to \mbf E_{\Omega}(D, V^1, V^2)$.
Note that $\mbf Z_{\Omega}(D, V^1, V^2)\simeq \mbf Z^t_{\Omega}(D, V^2, V^1)$.

 We set  the following complexes in $\mathscr D^b_{\G}( \mbf E_{\Omega})$ with $\mu=\lambda-\nu$ and $n\in \mbb N$:
\begin{equation}
\label{Generator-I}
\begin{split}
\II_{\mu}  &= Q \left (\pi_{12!} (\bar{\mbb Q}_{l,  \mbf Z_{\Omega}}) \right ),  \quad \quad \quad \;  \mbox{ if}\;  \dim V^1=\dim V^2=\nu; \\
\EE^{(n)}_{\mu, \mu-n \alpha_i}  &=Q \left ( \pi_{12!}  (\bar{\mbb Q}_{l, \mbf Z_{\Omega}}) [e_{\mu,n\alpha_i}]  \right ), 
\quad  \mbox{if}\;    \dim V^1=\nu \;\mbox{and}\;  \dim V^2=\nu+ ni;\\ 
\FF^{(n)}_{\mu, \mu +n \alpha_i} &=  Q\left ( \pi_{12!}  (\bar{\mbb Q}_{l, \mbf Z^t_{\Omega}}) [f_{\mu,n\alpha_i}] \right ),
\quad  \mbox{if}  \dim V^1=\nu\;  \mbox{and}\; \dim V^2=\nu-ni;
\end{split}
\end{equation}
where 
\begin{equation}
\label{coefficient} 
e_{\mu, n\alpha_i}=n \left (d_i +\sum_{h\in \Omega: h'=i} \nu_{h''} -(\nu_i+n) \right )
\quad \mbox{and}\quad  
f_{\mu, n\alpha_i}=n\left ((\nu_i -n)-\sum_{h\in \Omega: h''=i} \nu_{h'}\right ).
\end{equation}
We set $\II_{\mu}$, $\EE^{(n)}_{\mu,\mu-n\alpha_i}$ and $\FF^{(n)}_{\mu,\mu+n\alpha_i}$ to 
be zero if $\mu\in \mbf X$ can not be written as the form $\mu=\lambda-\nu$ for some $\nu\in \mbb N[I]$.
We have 

\begin{lem}
\label{generator-Fourier}
$\Phi_{\Omega}^{\Omega'} (\II_{\mu}) =\II_{\mu}$, 
$\Phi_{\Omega}^{\Omega'} \left (\EE^{(n)}_{\mu, \mu-n \alpha_i} \right ) =\EE^{(n)}_{\mu, \mu-n \alpha_i} $
and 
$\Phi_{\Omega}^{\Omega'} \left (\FF^{(n)}_{\mu, \mu +n \alpha_i} \right ) =\FF^{(n)}_{\mu, \mu +n \alpha_i}$
where the elements on the right-hand sides are complexes in $\DD^b_{\G}(\mbf E_{\Omega'})$ defined in a similar way as the complexes on the left-hand sides.
\end{lem}
\begin{proof}
We shall show that $\Phi_{\Omega}^{\Omega'} \left (\EE^{(n)}_{\mu, \mu-n \alpha_i} \right ) =\EE^{(n)}_{\mu, \mu-n \alpha_i} $. The others can be shown in a similar way. 
Consider the following diagram
\[
\begin{CD}
\mbf Z_{\Omega} @<<< \hat {\mbf Z} @>c>> \hat {\mbf Z}_{\Omega'} \\
@V\pi_{12} VV @V\beta VV @VVV\\
\mbf E_{\Omega} @<m_{12}<< \mbf E_{\Omega\cup \Omega'} @>m_{12}'>> \mbf E_{\Omega'},
\end{CD}
\]
where the morphisms $m_{12}$ and $m_{12}'$ are defined in Section ~\ref{Fourier}, $\pi_{12}$ is defined in (\ref{Z-1}), 
$\hat{\mbf Z}=\mbf Z_{\Omega}\times_{\mbf E_{\Omega}} \mbf E_{\Omega\cup \Omega'}$, the morphism $c$ is 
the map by forgetting the components $x^1_h$ and $x^2_h$ for $h\in\Omega\backslash \Omega'$, and the rest of the morphisms are clearly defined.

By definition, we see that $\mbf Z_{\Omega}$ is a locally closed subvariety  in the   variety $\mbf E_{\Omega}\times \prod_{i\in I} \Hom( V^1_i, V^2_i)$, 
which, in turn, can be embedded as an open subvariety into a certain variety $\mbf E_{\Omega}\times \mbf P$, where $\mbf P$ is a certain projective variety, by 
~\cite{Zelevinsky85}. Then the morphism $\pi_{12}$ is a compactifiable morphism defined in ~\cite[p. 86]{FK88}. From this and the fact that the left square of the above diagram is cartesian, we may apply the base change for  compactifiable morphisms in ~\cite[Theorem 8.7]{FK88} to get the following identity
\begin{equation*}
\begin{split}
\Phi_{\Omega}^{\Omega'} (\pi_{12!} (\Q_{l, \mbf Z_{\Omega}}) )
=m'_{12!} (m_{12}^* \pi_{12!} (\Q_{l, \mbf Z_{\Omega}}) \otimes \mathcal L_{12})[r_{12}]\\
=m'_{12!} ( \beta_! (\Q_{l, \hat{\mbf Z}})\otimes \mathcal L_{12})[r_{12}]
=m'_{12!}\beta_!(\beta^*\mathcal L_{12})[r_{12}],
\end{split}
\end{equation*}
where $\mathcal L_{12}$ and $r_{12} $ are defined in ~\ref{Fourier}.

Let $\hat {\mbf Z}_1$ be the subvariety of $\hat{ \mbf  Z}$ defined by the condition 
$x^1_h\overset{\rho }{\hookrightarrow} x^2_h$ for any $h \in\Omega'\backslash \Omega$.
Observe that that the map $c$ is a vector bundle  with fiber dimension 
\[
f_c=\sum_{h\in \Omega\backslash \Omega':h''\neq i} \nu^2_{h'}\nu^2_{h''} +\sum_{h\in \Omega\backslash \Omega': h''=i} \nu^1_{h'} \nu^1_{h''},
\]
 and 
the restriction of the map $u_{12}\beta$, where $u_{12}$ is defined in Section ~\ref{Fourier}, to the fiber $c^{-1}(X,\rho)$ of $c$ is $0$ if
$c^{-1}(X, \sigma)\subseteq \hat{\mbf Z}_1$ and a non-constant affine linear function,  otherwise. 
By arguing in exactly the same way as the proof of Proposition 10.2.2 in ~\cite{Lusztig93}, we get
\begin{align*}
\begin{split}
\Phi_{\Omega}^{\Omega'} (\pi_{12!}  & (\Q_{l,\mbf Z_{\Omega}}) )
\simeq \pi_{12!}^{\Omega'} (\Q_{l, \mbf Z_{\Omega'}}[r_{12} -2f_c],
\end{split}
\end{align*}
where $\pi_{12}^{\Omega'}$ is defined similar to $\pi_{12}$. 
It is clear that 
\[
r_{12} -2f_c=-n \sum_{h\in \Omega\backslash \Omega':h'= i} \nu_{h''}
+ n\sum_{h\in \Omega\backslash \Omega': h''=i} \nu_{h'}.
\]
By combining the above analysis, we have 
\begin{equation*}
\begin{split}
\Phi_{\Omega}^{\Omega'} & \left (\EE^{(n)}_{\mu, \mu-n \alpha_i} \right )
=Q \left ( \pi_{12!}^{\Omega'} (\Q_{l,\mbf Z_{\Omega'}}[r_{12} -2f_c] [n(d_i +\sum_{h\in \Omega: h'=i} \nu_{h''} -(\nu_i+n))]  \right ) =\EE^{(n)}_{\mu,\mu-n\alpha_i}, 
\end{split}
\end{equation*}
where the second equality is due to the following  computation.
\[
r_{12} -2f_c + n(d_i +\sum_{h\in \Omega: h'=i} \nu_{h''} -(\nu_i+n))
=n(d_i +\sum_{h\in \Omega': h'=i} \nu_{h''} -(\nu_i+n)).
\]
The lemma follows.
\end{proof}

\subsection{Defining relation} 

We shall show that the complexes $\II_{\mu}$, $\EE^{(n)}_{\mu, \mu-n \alpha_i}$ and $\FF^{(n)}_{\mu, \mu +n \alpha_i}$ satisfy the defining relations of $_{\mbb A}\! \dot{\U}$.

\begin{lem}
\label{I-I}
$\II_{\mu}\II_{\mu'} =\delta_{\mu, \mu'} \II_{\mu}$ where $\mu'=\lambda -\nu'$ and  $\nu'\in \mbb N[I]$.
\end{lem}

\begin{proof}
Assume that $V^1=V^2=V^3$ has dimension $\nu$. Consider the following cartesian diagram
\[
\begin{CD}
\mbf Z_{\Omega}(D, V^1, V^2) @< << \mbf Z_1 @>s_1>> \mbf Z\\
@V\pi_{12}VV @Vr_1VV @Vs_2VV\\
\mbf E_{\Omega}(D, V^1, V^2) @<p_{12}<< \mbf E_{\Omega}(D, V^1, V^2, V^3) @< r_2<< \mbf Z_2,
\end{CD}
\]
where $\mbf Z_1=\mbf Z_{\Omega}(D, V^1, V^2) \times \mbf E_{\Omega}(D, V^3)$, $\mbf Z_2=\mbf E_{\Omega}(D, V^1)\times \mbf Z_{\Omega}(D, V^2, V^3)$,
\[
\mbf Z=\mbf Z_1\times_{\mbf E_{\Omega} (D, V^1, V^2, V^3)} \mbf Z_2= \{(X^1, X^2, X^3, \rho_1, \rho_2) | 
X^1\overset{\rho_1}{\hookrightarrow} X^2 \overset{\rho_2}{\hookrightarrow} X^3\}.
\]
and  the  morphisms  are the obvious projections. 
The cartesian square on the left gives rise to the identity 
$P_{12}^* \Pi_{12!} (\bar{\mbb Q}_{l,\mbf Z_{\Omega}(D, V^1, V^2)})= R_{1!} (\bar{\mbb Q}_{l, \mbf Z_1})$.
Similarly, $P_{23}^* \Pi_{12!} (\bar{\mbb Q}_{l,\mbf Z_{\Omega}(D, V^2, V^3)})= R_{2!} (\bar{\mbb Q}_{l, \mbf Z_2})$. 
So 
\begin{equation*}
\begin{split}
\II_{\mu} \cdot \II_{\mu} & = P_{13!} (P_{12}^* (\II_{\mu}) \otimes P_{23}^* (\II_{\mu})) 
= P_{13!}(R_{1!} (\Q_{l, \mbf Z_1})\otimes R_{2!} (\Q_{l, \mbf Z_2}))\\
&=P_{13!} R_{1!} (\Q_{l, \mbf Z_1}\otimes R_1^* R_{2!} (\Q_{l, \mbf Z_2}))
=P_{13!} R_{1!} (R_1^* R_{2!} (\Q_{l, \mbf Z_2}) ).
\end{split}
\end{equation*}
The right cartesian square in the above diagram   implies that $R_1^* R_{2!} = S_{1!} S_2^*$.
Thus
\[
\II_{\mu} \cdot \II_{\mu}=P_{13!} R_{1!} (R_1^* R_{2!} (\Q_{l, \mbf Z_2}) ) =
P_{13!} R_{1!}S_{1!} S_2^* (\Q_{l, \mbf Z_2})= P_{13!} R_{1!}S_{1!} (\Q_{l, \mbf Z}).
\]
Consider the following commutative  diagram 
\[
\begin{CD}
\mbf Z @>t>> \mbf Z_{\Omega}(D, V^1, V^3)\\
@Vr_1s_1VV @V\pi_{12}VV \\
\mbf E_{\Omega}(D, V^1, V^2. V^3) @> p_{13}>> \mbf E_{\Omega}(D, V^1, V^3),
\end{CD}
\]
where $t$ sends $X^1\overset{\rho_1}{\hookrightarrow} X^2 \overset{\rho_2}{\hookrightarrow} X^3$ to $X^1\overset{\rho_2 \rho_1}{\hookrightarrow} X^3$. Then we have
\[
\II_{\mu} \cdot \II_{\mu}=P_{13!} R_{1!}S_{1!} (\Q_{l, \mbf Z})= \Pi_{12!} T_! (\Q_{l, \mbf Z}).
\]
Observe that $t$ is the quotient map of $\mbf Z$ by the group  $G_{V^2}$, thus the induced morphism $Qt: [\mbf H\backslash \mbf Z] \to [\G\backslash \mbf Z_{\Omega}(D, V^1, V^3)]$ is an isomorphism.
Hence $T_! (\Q_{l, \mbf Z})=\Q_{l,\mbf Z_{\Omega}(D, V^1, V^3)}$. Therefore,
\[
\II_{\mu} \cdot \II_{\mu}=\Pi_{12!} (\Q_{l,\mbf Z_{\Omega}(D, V^1, V^3)} )= \pi_{12!} (\Q_{l,\mbf Z_{\Omega}(D, V^1, V^3)} )=\II_{\mu}.
\]
It is clear that $\II_{\mu} \cdot \II_{\mu'}=0$ if $\mu\neq \mu'$ from the above argument. The lemma follows.
\end{proof}

\begin{lem} 
For any $\mu$ and $\mu'$, we have  
\begin{alignat*}{3}
 \EE^{(n)}_{\mu,\mu-n\alpha_i} \II_{\mu'}&=\delta_{\mu-n\alpha_i,\mu'} \EE^{(n)}_{\mu, \mu-n\alpha_i},\quad 
& \II_{\mu'} \EE^{(n)}_{\mu, \mu-n\alpha_i} &=\delta_{\mu', \mu} \EE^{(n)}_{\mu, \mu-n\alpha_i}; \\
\FF^{(n)}_{\mu, \mu +n\alpha_i} \II_{\mu'}&=\delta_{\mu+n\alpha_i,\mu'} \FF^{(n)}_{\mu, \mu+n\alpha_i}, \quad
&\II_{\mu'} \FF^{(n)}_{\mu, \mu+n \alpha_i} &= \delta_{\mu', \mu} \FF^{(n)}_{\mu, \mu+n\alpha_i}.
\end{alignat*}
\end{lem}

This lemma  can be  proved in exactly the same way as the proof of Lemma ~\ref{I-I}.

\begin{lem}
\label{E-i-j}
$\EE_{\mu -\alpha_j+\alpha_i,\mu -\alpha_j} \FF_{\mu -\alpha_j,\mu}=  \FF_{\mu +\alpha_i-\alpha_j,\mu +\alpha_i} \EE_{\mu+\alpha_i,\mu} $, for any $i\neq j$.
\end{lem}

\begin{proof}
Fix four $I$-graded vector spaces $V^s$ for $s=1, 2, 3, 4$ such that
\begin{equation}
\label{E-F-i-j}
\dim V^1 =\nu +j-i, \quad \dim V^2 =\nu +j, \quad \dim V^3 = \nu \quad\mbox{and} \quad 
\dim V^4 =\nu-i.
\end{equation}
Let 
\begin{equation*}
\mbf Z_1
=\{ (X^1, X^2, X^3, \rho_1, \rho_2) | (X^1, X^2, X^3)\in \mbf E_{\Omega}(D, V^1, V^2, V^3), X^1 \overset{\rho_1}{\hookrightarrow} X^2 \overset{\rho_2}{\hookleftarrow} X^3\},
\end{equation*}
and $\pi_1$ be the projection from $\mbf Z_1$ to $\mbf E_{\Omega}(D, V^1, V^2, V^3)$. 
Then an argument similar to the proof of Lemma ~\ref{I-I} yields that 
\[
\EE_{\mu -\alpha_j+\alpha_i,\mu -\alpha_j} \FF_{\mu -\alpha_j,\mu}=P_{13!} \Pi_{1!} (\Q_{l,\mbf Z_1}) [m],
\]
where
$m=e_{\mu-\alpha_j+\alpha_i, \alpha_i}+f_{\mu-\alpha_j, \alpha_j}$.
Denote by $\check V$ the $I\backslash \{j\}$-graded vector space obtained from $V$ by deleting the component $V_j$. Let 
$\mbf Z$ be the variety of quadruples  $( X^1, X^3,  \check \rho: \check V^1\hookrightarrow  \check V^3, \sigma_j: V^3_j \hookrightarrow V^1_j)$, where  
$(X^1, X^3)\in \mbf E_{\Omega}(D, V^1, V^3)$, such that all the diagrams  incurred in the quadruples are commutative, i.e.,
\[
x^3_h \check \rho_{ h'}=\check \rho_{ h''} x^1_h, \; \mbox{if}\;  \{h', h''\} \neq j;\;
x_h^1 = \sigma_j x^3_h \check \rho_{ h'},\; \mbox{if} \; h''=j;\; \mbox{and}\; 
\sigma_j x^1_h = x^3_h \check \rho_{h'},\; \mbox{if} \; h'=j.
\]
Define a morphism of varieties
\[
r_{13}: \mbf Z_1\to \mbf Z
\]
by $r_{13}(X^1, X^2, X^3, \rho_1, \rho_2) = (X^1, X^3, \check \rho, \sigma_j)$ where 
$\check \rho_{k} = \rho_{2, k}^{-1} \rho_{1, k}$ for any $k\in I\backslash\{ j\}$ and $\sigma_j= \rho_{1, j}^{-1} \rho_{2,j}$.
Then, we have $p_{13} \pi_2= \pi r_{13}$, where $\pi$ is the projection from $\mbf Z$ to $\mbf E_{\Omega}(D, V^1, V^3)$. Moreover,
we observe that $r_{13}$ is a quotient map of $\mbf Z_1$ by $G_{V^2}$. From these facts, we have
\begin{equation}
\label{E-F-LHS}
\EE_{\mu -\alpha_j+\alpha_i,\mu -\alpha_j} \FF_{\mu -\alpha_j,\mu}=P_{13!} \Pi_{1!} (\Q_{l,\mbf Z_1})[m]
=\Pi_! R_{13!} (\Q_{l,\mbf Z_1}) [m]=\Pi_! (\Q_{l, \mbf Z})[m].
\end{equation}

On the other hand, let 
\begin{equation*}
\mbf Z_2
=\{ (X^1, X^4, X^3, \rho_1, \rho_2) | (X^1, X^4, X^3)\in \mbf E_{\Omega}(D, V^1, V^4, V^3), X^1 \overset{\rho_1}{\hookleftarrow} X^4 \overset{\rho_2}{\hookrightarrow} X^3\}
\end{equation*}
and $\pi_2$ be the projection from $\mbf Z_2$  to $\mbf E_{\Omega}(D, V^1, V^4, V^3)$. Define a morphism $\tilde r_{13}: \mbf Z_2\to \mbf Z$ by 
$\tilde r_{13} (X^1, X^4, X^3, \rho_1, \rho_2)=(X^1, X^3, \check \rho, \sigma_j)$ where 
$\check \rho_{ k} = \rho_{2,k} \rho_{1,k}^{-1}$ for any $k\in I\backslash \{j\}$ and 
$\sigma_j= \rho_{1,j} \rho_{2,j}^{-1}$.
Then we have $\tilde p_{13} \pi_2 = \pi \tilde r_{13}$, where $\tilde p_{13}$ is the projection from $\mbf Z_2$ to $\mbf E_{\Omega}(D, V^1, V^3)$, 
moreover $\tilde r_{13}$ is quotient map of $\mbf Z_2$ by $G_{V^4}$. From these facts, we get
\begin{equation}
\label{E-F-RHS}
 \FF_{\mu +\alpha_i-\alpha_j,\mu +\alpha_i} \EE_{\mu+\alpha_i,\mu} =
 \tilde P_{13!} \Pi_{2!} (\Q_{l,\mbf Z_2})[m']  = \Pi_! \tilde R_{13!} (\Q_{l,\mbf Z_2}) [m']=\Pi_! (\Q_{l, \mbf Z})[m'],
\end{equation}
where $m'=f_{\mu+\alpha_i-\alpha_j,\alpha_j} +e_{\mu+\alpha_i,\alpha_i}$.
By (\ref{E-F-LHS}), (\ref{E-F-RHS}) and the fact that $m=m'$, we have
the lemma.
\end{proof}

\begin{lem}  
\label{E-F-i} 
Let $\nu(i)=d_i +\sum_{h\in H: h'=i} \nu_{h''}$. For any vertex $i\in I$, 
\[
\EE_{\mu, \mu-\alpha_i} \FF_{\mu-\alpha_i, \mu} \oplus \bigoplus_{p=0}^{\nu_i -1} \II_{\mu} [\nu(i) -1-2p] =
\FF_{\mu, \mu+\alpha_i} \EE_{\mu+\alpha_i,\mu} \oplus \bigoplus_{p=0}^{\nu(i) -\nu_i-1}\II_{\mu}[\nu(i)-1-2p].
\]

\end{lem}

\begin{proof}
Due to Lemma ~\ref{generator-Fourier},  we may assume that $i$ is a source in $\Omega$.  Let us fix four $I$-graded vector spaces, $V^a$, for $a=1, 2, 3, 4$,  such that
\begin{equation}
\label{E-i-i}
\dim V^1=\dim V^3=\nu, \quad \dim V^2=\nu+i \quad \mbox{and} \quad \dim V^4=\nu-i.
\end{equation}
Then  we have 
\[
\EE_{\mu,\mu -\alpha_i} \FF_{\mu -\alpha_i,\mu}=P_{13!} \Pi_{1!} (\Q_{l,\mbf Z_1}) [m]
\quad \mbox{and} \quad
\FF_{\mu +\alpha_i-\alpha_j,\mu +\alpha_i} \EE_{\mu+\alpha_i,\mu} =
 \tilde P_{13!} \Pi_{2!} (\Q_{l,\mbf Z_2})[m],
\]
where $m=\nu(i) -1$ and  
the other notations on the right-hand sides are defined  in the proof of Lemma ~\ref{E-i-j}  with the condition (\ref{E-F-i-j}) replaced by 
(\ref{E-i-i}).

Let $\mbf Z_1^s$ be the open subvariety of $\mbf Z_1$ defined by the condition that $X^1(i)$, $X^2(i)$ and $X^3(i)$ are injective.
Similarly, we define the open subvariety $\mbf E_{\Omega}^s$ in $\mbf E_{\Omega}(D, V^1, V^3)$.

Denote by $\check X$  the element obtained from $X\in \mbf E_{\Omega}(D, V)$ by deleting any component $x_h$ such that $ h'=i$.
Let $\check {\mbf Z}_1^s$ be the variety of tuples $( X^1,\check X^2, X^3, \rho_1, \rho_2)$.
Let $\mbf Y_1$ be the variety of  tuples $(\check X^1, \check X^2, \check X^3, \check \rho_1, \check \rho_2, \mathcal V_1,\mathcal  V_2, \mathcal V_3)$, 
where $\mathcal V_1, \mathcal V_2, \mathcal V_3\subseteq D_i \oplus \oplus_{h\in \Omega: h'=i}V_{h''}$, 
such that $\mathcal V_1,\mathcal  \mathcal \mathcal V_3\subseteq \mathcal V_2$,
$\dim \mathcal V_1=\dim \mathcal V_3=\nu_i$ and
$\dim \mathcal V_2=\nu_i+1$.
Similarly, we define the variety $\X_1$ of tuples $(\check X^1, \check X^2, \check X^3, \check \rho_1, \check \rho_2, \mathcal V_1, \mathcal V_3)$ and 
the variety $\mbf W$ of tuples $(\check X^1,\check X^3, \mathcal V_1, \mathcal V_3)$.
Then we have the following cartesian diagram
\[
\begin{CD}
\mbf Z_1^s @>r_1 >>  \check {\mbf Z}_1^s @>r_2>> \mbf E_{\Omega}^s\\
@Vs_1 VV @Vs_2 VV @Vs_3VV\\
\mbf Y_1 @>r_3>> \mbf X_1 @>r_4 >> \mbf W,
\end{CD}
\]
where the $r_a$'s  are the obvious projections, and $s_2$ and $s_3$ are induced from $s_1$, which is defined by
$s_1 ( X^1, X^2, X^3,\rho_1,\rho_2) = (\check X^1, \check X^2, \check X^3, \check \rho_1, \check \rho_2, \mathcal V_1,\mathcal  V_2, \mathcal V_3)$
with 
\[
\mathcal V_2 =\mbox{im} \; (q_i^2+\sum_{h\in \Omega: h'=i} x_h^2),\quad \mbox{and}\quad
\mathcal V_a=\mbox{im} \; (q_i^a+\sum_{h\in \Omega: h'=i} \rho_{a,h''}x_h^a),\quad \forall a=1, 3.
\]
Observe that $s_1$ and $s_2$ are the quotient maps of $\mbf Z_1^s$ and $\check{\mbf Z}^s_1$ by the group 
$G_{V_i^1}\times G_{V_i^2}\times G_{V_i^3}$, respectively, and $s_3$ is the quotient map of $\mbf E^s_{\Omega}$ by  $G_{V_i^1}\times G_{V_i^3}$. Thus we have
\begin{equation}
R_{2!} R_{1!} (\Q_{l,\mbf Z_1^s}) 
=S_3^* R_{4!} R_{3!} (\Q_{l,\mbf Y_1}). 
\end{equation}
Let $\mbf Y_1^c$ be the closed subvariety of $\mbf Y_1$ defined by the condition $\mathcal V_1=\mathcal V_3$ and $\mbf Y_1^o$ be its complement.
Let $i_1: \mbf Y_1^c\to \mbf Y_1$ and $j_1:\mbf  Y_1^o\to \mbf Y_1$ be the inclusions.
Sine $r_3$ is proper and $\mbf Y_1$ is smooth,  the complex $R_{3!} (\Q_{l,\mbf Y_1})$ is semisimple. So we have 
\[
r_{3!} (\Q_{l,\mbf Y_1}) = j_{1!*}^o r^o_{3!}  (\Q_{l, \mbf Y_1^o}) \oplus r_{3!} i_{1!} (\Q_{l,\mbf Y_1^c}),
\]
where $j^o_1$ and $r_3^o$ are the morphisms $\mbf Y_1^o\overset{r_3^o} {\to} \mbf X_1^o  \overset{j_1^o}{\to} \mbf X_1$ with $\mbf X_1^o$ the image of $\mbf Y_1^o$
under $r_3$.
Observe that the morphism $r_3 i_1$ is a projective bundle of relative dimension $\nu(i)-\nu_i-1$. Thus
$r_{3!} i_{1!} (\Q_{l,\mbf Y_1^c})= \oplus_{p=0}^{\nu(i)-\nu_i -1} \Q_{l, \mbf X_1^c}[-2p]$ 
where $\mbf X_1^c$ is the closed subvariety of $\mbf X_1$ defined by the condition 
$\mathcal V_1=\mathcal V_3$.
Then, we have 
\[
R_{3!} (\Q_{l,\mbf Y_1})= J_{!*}^o R^o_{3!}  (\Q_{l, \mbf Y_1^o}) \oplus  \oplus_{p=0}^{\nu(i)-\nu_i -1} \Q_{l, \mbf X_1^c}[-2p].
\]

By combining the above analysis, we see that the restriction of the complex $\EE_{\mu,\mu -\alpha_i} \FF_{\mu -\alpha_i,\mu}$ to $\mbf E_{\Omega}^s$ 
is equal to 
\begin{equation}
\label{E-i-i-LHS}
R_{2!}R_{1!} (\Q_{l,\mbf Z_1^s}) = S_3^* R_{4!} R_{3!} (\Q_{l,\mbf Y_1})=
S_3^* R_{4!} J_{1!*}^o R^o_{3!}  (\Q_{l, \mbf Y_1^o}) \oplus  \oplus_{p=0}^{\nu(i)-\nu_i -1} S_3^* R_{4!}\Q_{l, \mbf X_1^c}[m-2p].
\end{equation}

On the other hand, we may define the  open subvarieties $\mbf Z_2^s$ of $\mbf Z_2^s$ similar to the subvariety $\mbf Z_1^s$ of $\mbf Z$. 
Then the following varieties $\check{\mbf Z}_2$,
$\mbf Y_2$ , $\mbf Y_2^c$, $\mbf Y_2^o$ and $\mbf X_2$ in the diagram below are defined in a way similar to the varieties having  subscript $1$:
\[
\begin{CD}
\mbf Z_2^s @>t_1 >>  \check {\mbf Z}_2^s @>t_2>> \mbf E_{\Omega}^s\\
@Vw_1 VV @Vw_2 VV @Vs_3VV\\
\mbf Y_2 @>t_3>> \mbf X_2 @>t_4 >> \mbf W,
\end{CD}
\]
where the morphism $w_1$ is defined by
$w_1 ( X^1, X^4, X^3,\sigma_1,\sigma_2) = (\check X^1, \check X^4, \check X^3, \check \sigma_1, \check \sigma_2, \mathcal V_1,\mathcal  V_4, \mathcal V_3)$
with $\mathcal V_4 =\mbox{im} \; (q_i^4+\sum_{h\in \Omega: h'=i} x_h^4)$  and
$\mathcal V_a=\mbox{im} \; (q_i^a+\sum_{h\in \Omega: h'=i} \sigma_{a, h''}^{-1}x_h^a)$ for any $ a=1, 3$.

Let $i_2: \mbf Y_2^c\hookrightarrow \mbf Y_2\hookleftarrow \mbf Y_2^o: j_2 $ be the inclusions. Then we see that $t_3i_2$ is a projective bundle of relative dimension $\nu_i-1$.
An argument similar to the proof of (\ref{E-i-i-LHS}) shows that the restriction of the complex $\FF_{\mu +\alpha_i-\alpha_j,\mu +\alpha_i} \EE_{\mu+\alpha_i,\mu}$
to $\mbf E^s_{\Omega}$ is equal to 
\begin{equation}
\label{E-i-i-RHS}
T_{2!}T_{1!} (\Q_{l,\mbf Z_2^s}) = S_3^* T_{4!} T_{3!} (\Q_{l,\mbf Y_2})=
S_3^* T_{4!}  J^o_{2!*} T^o_{3!}  (\Q_{l, \mbf Y_2^o}) \oplus  \oplus_{p=0}^{\nu_i -1} S_3^* T_{4!}\Q_{l, \mbf X_2^c}[m-2p],
\end{equation}
where $j^o_2$ and $t_3^o$ are the morphisms $\mbf Y_2^o\overset{t_3^o} {\to} \mbf X_2^o  \overset{j_2^o}{\to} \mbf X_2$ with $\mbf X_2^o$.
Finally, observe that there are isomorphisms $\mbf Y_1^o\simeq \mbf Y_2^o$, $\mbf X^c_1\simeq \mbf X^c_2$ and moreover the complex
$S_3^* T_{4!}(\Q_{l, \mbf X_2^c})$ is the restriction of $\II_{\mu}$ to $\mbf E^s_{\Omega}$. 
The lemma follows by comparing (\ref{E-i-i-LHS}) and (\ref{E-i-i-RHS}) and using 
the observations.
\end{proof}

\begin{lem} 
\label{Serre}
For any $i\neq j\in I$,  let $m=1-i\cdot j$. We have
\begin{equation*}
\begin{split}
&\bigoplus_{\substack{0\leq p\leq m \\ p \; even }}  
\EE^{(m-p)}_{\mu^3, \mu^2} 
\EE_{\mu^2, \mu^1} 
\EE^{(p)}_{\mu^1, \mu} 
=\bigoplus_{\substack{0\leq p\leq m\\ p \; odd} }  
\EE^{(m-p)}_{\mu^3,  \mu^2} 
\EE_{\mu^2, \mu^1} 
\EE^{(p)}_{\mu^1, \mu};\\
&\bigoplus_{\substack{0\leq p\leq m\\p \; even}  }  
\FF^{(p)}_{\mu,  \mu^1} 
\FF_{\mu^1,  \mu^2} 
\FF^{(m-p)}_{\mu^2,  \mu^3} 
=
\bigoplus_{\substack{0\leq p\leq m\\ p \; odd}  }  
\FF^{(p)}_{\mu,  \mu^1} 
\FF_{\mu^1,  \mu^2} 
\FF^{(m-p)}_{\mu^2,  \mu^3};
\end{split}
\end{equation*}
where 
$\mu^1=\mu+p\alpha_i$, 
$\mu^2= \mu+p\alpha_i+\alpha_j$, and
$\mu^3=\mu+m\alpha_i+\alpha_j$.
\end{lem}

\begin{proof}
Without lost of generality, we assume that $i$ is a source in $\Omega$. For $a=1, 2, 3, 4$, let $V^a$  be the $I$-graded vector spaces such that
\[
\dim V^1 =\nu-mi-j, \quad 
\dim V^2 =\nu -p i-j,\quad
\dim V^3=\nu-pi \quad \mbox{and} \quad
\dim V^4= \nu.
\]
Let  $\mbf Z$  be the variety of the data $(X^1 \overset{\rho_1}{\hookrightarrow} X^2\overset{\rho_2}{\hookrightarrow} X^3 \overset{\rho_3}{\hookrightarrow} X^4)$ where $X^a\in \mbf E_{\Omega}(D, V^a)$ for $a=1, 2, 3, 4$.  
Let $\pi: \mbf Z\to  \mbf E_{\Omega}(D, V^1,V^4)$ be the obvious projection.
Then 
\[
\EE^{(m-p)}_{\mu^3, \mu^2} 
\EE_{\mu^2, \mu^1} 
\EE^{(p)}_{\mu^1, \mu} 
=\Pi_! (\Q_{l, \mbf Z})[s_{m-p}],
\]
where 
$s_{m-p}=m(\nu(i)-\nu_i) +(d_j +\sum_{h\in \Omega: h'=j}\nu_{h''} -\nu_j)+(m-p) ( 1-(m-p))$.
Moreover, $\pi$ factors through $\mbf Z_{\Omega}(D, V^1, V^4)$, where the map $r$ from $\mbf Z$ to $\mbf Z_{\Omega}(D, V^1, V^4)$ is given by
$r (X^1 \overset{\rho_1}{\hookrightarrow} X^2\overset{\rho_2}{\hookrightarrow} X^3 \overset{\rho_3}{\hookrightarrow} X^4)=(X^1\overset{\rho_3\rho_2\rho_1}{\hookrightarrow} X^4)$.
Let $B_{m-p}=R_! (\Q_{l, \mbf Z})$.
Thus,
\[
\EE^{(m-p)}_{\mu^3, \mu^2} 
\EE_{\mu^2, \mu^1} 
\EE^{(p)}_{\mu^1, \mu} 
=\Pi_{!} (\Q_{l, \mbf Z_p})[s_{m-p}]
=\Pi_{12!} B_{m-p}[s_{m-p}].
\]
The identity for the $\EE$'s  is reduced to show that 
\begin{equation}
\label{even-odd}
\bigoplus_{\substack{0\leq p\leq m\\ p\; even}} B_{m-p}[(m-p) ( 1-(m-p))] =\bigoplus_{\substack{0\leq p\leq m \\ p \; odd}} B_{m-p}[(m-p)(1-(m-p))].
\end{equation}
This is shown in ~\cite[2.5.8]{Zheng08}.
For the sake of completeness, let us reproduce here.
Let $\mbf Z^s$ be the open subvariety of $\mbf Z$ defined by the condition that $X^a(i)$ are injective for $a=1, 2, 3, 4$.  The variety $\mbf Z^s_{\Omega}(D, V^1, V^4)$ is defined similarly.  Let $s: \mbf Z^s\to \mbf Z^s_{\Omega}(D, V^1, V^4)$ be the restriction of $r$ to $\mbf Z^s$ and 
$C_{m-p}:= S_!(\Q_{l,\mbf Z^s})$.
 Then the condition on the localization implies that we only need to show the
identity (\ref{even-odd}) when we restrict the complexes involved  to the variety $\mbf Z^s_{\Omega}(D, V^1, V^4)$, 
i.e., to show that  (\ref{even-odd}) holds with the complexes $B_{m-p}$ replaced by 
the complexes $C_{m-p}$.

Observe that the group $G_{V^2}\times G_{V^3}\times \mrm{GL}(V^1_i)\times \mrm{GL}(V^4_i)$ acts freely on $\mbf Z^s$ and the quotient  variety $\mbf Y$ is the variety 
of tuples $(\check X^1\overset{\rho}{\hookrightarrow} \check X^4, \mathcal V_1, \mathcal V_2, \mathcal V_3)$ 
where $\mathcal V_1,\mathcal V_2 \subseteq V^1(i)$, $ \mathcal V_3\subseteq V^4(i))$ such that 
$\mathcal V_1\subseteq V_2$ and $\rho(\mathcal V_2)\subseteq \mathcal V_3$; and $\dim \mathcal V_1=\nu_i -m$, $\dim \mathcal V_2=\nu_i -p$ and $\mathcal V_3=\nu_i$.
Moreover, the group $\mrm{GL}(V^1_i)\times \mrm{GL}(V^4_i)$ acts freely on $\mbf Z^s_{\Omega}(D, V^1, V^4)$ and its quotient variety $\mbf X$ 
is the variety obtained from $\mbf Y$ 
by deleting $\mathcal V_2$ and replacing the condition $\rho(\mathcal V_2)\subseteq \mathcal V_3$ by $\rho(\mathcal V_1)\subseteq \mathcal V_3$.
Let $t: \mbf Y\to \mbf X$ be the projection. Let $A_{m-p}=T_!(\Q_{l, \mbf Y})$. To show (\ref{even-odd}), it reduces to show that the identity holds 
with the complexes $B_{m-p}$ replaced by 
the complexes $A_{m-p}$.

Now define a partition $(\mbf X_n)_{n=0}^m$ of $\mbf X$ such that elements in $\mbf X_n$ satisfying  the condition that 
$\dim \rho(V^1(i))\cap \mathcal V_3= \nu_i -m+n$.  Let $\mbf Y_n= t^{-1}(\mbf X_n)$, and $t_n:\mbf  Y_n\to\mbf  X_n$ be the restriction of $t$ to $\mbf Y_n$. 
Then the restriction of $r$ to $\mbf Y_n$ has fiber at any point of $\mbf X_n$
isomorphic to the Grassmannian $\mrm{Gr}(m-p, n)$ of $(m-p)$-subspaces in  $n$-space.
By the property of the  cohomology of $\mrm{Gr}(m-p, n)$, we have  
\[
T_{n!} (\Q_{l,\mbf Y_n}) = \oplus_{\kappa} \Q_{l,\mbf X_n} [-2\sum_{a=1}^{m-p} (\kappa_a-a)],
\] 
where 
$\kappa$ runs through the sequences $(1\leq \kappa_1< \kappa_2<\cdots < \kappa_{m-p} \leq n)$.
Since the complexes $A_{m-p}$ are semisimple, it suffices to show that  (\ref{even-odd})  holds when restricts to the strata $X_n$ for all $n$, which is left to show that 
\[
\bigoplus_{\substack{ 0\leq p\leq n\\ p\; even}} \oplus_{\kappa} \Q_{l,\mbf X_n} [-2\sum_{a=1}^{p} (\kappa_a-a)+p(1-p)]=
\bigoplus_{\substack{0\leq p\leq n \\ p\; odd}} \oplus_{\kappa} \Q_{l,\mbf X_n} [-2\sum_{a=1}^{p} (\kappa_a-a)+p(1-p)].
\]
To any sequence $\kappa=(1\leq \kappa_1< \cdots <\kappa_p\leq n)$ of odd length, attached a sequence $\kappa'$ of even length  by 
$\kappa'_a=\kappa_{a+1} $ for $a=1,\cdots, p$,  if $\kappa_1=1$;
$\kappa'_1=1$ and  $\kappa'_{a+1}=\kappa_a$ for $a=1,\cdots, p$ if $\kappa_1\neq 1$.
This defines a  bijection between the set of sequences $\kappa$ of even length and the set of sequences $\kappa$ of odd length and it is clear that the shifts on both sides
are the same under this bijection. Thus the identity holds.
The identity for the $\FF$'s can be proved similarly.
\end{proof}

\begin{lem} For any $i\in I$ and $m\in \mbb N$,  we have
\label{EE=E}
\begin{equation*}
\begin{split}
&\EE_{\mu+(m+1)\alpha_i, \mu+m\alpha_i} \EE^{(m)}_{\mu+m \alpha_i, \mu} = \bigoplus_{0\leq p\leq m} \EE^{(m+1)}_{\mu+(m+1)\alpha_i, \mu} [m-2p];\\
&\FF_{\mu-(m+1)\alpha_i, \mu-m\alpha_i} \FF^{(m)}_{\mu-m \alpha_i, \mu} = \bigoplus_{0\leq p\leq m} \FF^{(m+1)}_{\mu-(m+1)\alpha_i, \mu} [m-2p].
\end{split}
\end{equation*}
\end{lem}

\begin{proof}
Fix three $I$-graded vector spaces $V^a$ for $a=1, 2,3$ such that $\dim V^1 =\nu-(m+1)i$, $\dim V^2=\nu-mi$ and $\dim V^3=\nu$. Let
$\mbf Z_1=\mbf Z_{\Omega}(D, V^1, V^3)$ and  $\mbf Z$ be the variety of tuples $(X^1, X^2, X^3, \rho_1, \rho_2)$, where $(X^1, X^2, X^3)\in \mbf E_{\Omega}(D, V^1, V^2, V^3)$, such that $X^1\overset{\rho_1}{\hookrightarrow} X^2 \overset{\rho_2}{\hookrightarrow} X^3$.
Let $t: \mbf Z \to \mbf Z_1$ be the map defined by $t(X^1, X^2, X^3, \rho_1, \rho_2)=(X^1, X^3, \rho_2\rho_1: X^1\to X^3)$. As in the proof of Lemma \ref{E-F-i}, we have 
$T_! (\Q_{l,\mbf Z} )=\oplus_{p=0}^{m} \Q_{l,\mbf Z_1}[-2p]$.  So
\begin{equation*}
\begin{split}
\EE&_{\mu+(m+1)\alpha_i, \mu+m\alpha_i}  \EE^{(m)}_{\mu+m \alpha_i, \mu}=\Pi_{12!} T_! (\Q_{l,\mbf Z}) [e_{\mu+(m+1)\alpha_i, \alpha_i}+e_{\mu+m\alpha_i, m \alpha_i} ]\\
&=  \oplus_{p=0}^m    \Pi_{12!}  (\Q_{l,\mbf Z_1})[e_{\mu+(m+1)\alpha_i, \alpha_i}+e_{\mu+m\alpha_i, m \alpha_i}-2p]=
\oplus_{p=0}^m \EE^{(m+1)}_{\mu+(m+1)\alpha_i, \mu}  [m-2p].
\end{split}
\end{equation*}
The proof for the $\FF$'s is similar.
\end{proof}

By specializing the shift $[z]$ to $v^z$ for any $z\in \mbb Z$, the identities  in Lemmas ~\ref{I-I}-\ref{EE=E} 
become  the defining relations of the integral form $_{\mbb A}\! \dot{\U}$.
In short, we have

\begin{prop}
\label{U-relations}
The complexes  $\II_{\mu}$, $\EE^{(n)}_{\mu, \mu-n \alpha_i}$ and $\FF^{(n)}_{\mu, \mu +n \alpha_i}$ 
satisfy the defining relations of the integral form $_{\mbb A}\! \dot{\U}$ defined in Section ~\ref{modified-quantum}.
\end{prop}

\begin{rem}
In many respects, if not all,  the proof of Proposition  ~\ref{U-relations} is very similar to that of ~\cite[Theorem  2.5.2]{Zheng08}
(see also Proposition ~\ref{Psi-relation} in this paper).
\end{rem}

\section{Algebra $\KK_d$}

\subsection{Complex $K_{\bullet}$}
\label{Kdot}

Consider the complexes of the form 
\begin{equation}
\label{complex}
K_{\bullet} =K_1 \cdot K_2  \cdot  ... \cdot K_m \quad \in \mathscr D^-_{\G}(\mbf E_{\Omega}(D, V^1, V^2)),
\end{equation}
where the $K_a$'s are either $\EE^{(n)}_{\mu', \mu}$ or $\FF^{(n)}_{\mu', \mu}$ in Section ~\ref{Generator}.

\begin{prop}
\label{boundedness}
The complexes $K_{\bullet}$ in (\ref{complex}) are bounded.
\end{prop}

\begin{proof}
For  any pair $(\mbf{i, a})$ of sequences, where $\mbf i=(i_m,\cdots,i_1)\in I^m$ and $\mbf a=(a_m,\cdots, a_1)\in \mbb N^m$, we write
\[
\EE_{(\mbf{i, a}), \mu} =\EE^{(a_m)}_{\mu,\mu^{m-1}} \cdots \EE^{(a_2)}_{\mu^2,\mu^1} \EE^{(a_1)}_{\mu^1,\mu^0}
\quad \mbox{and}\quad 
\FF_{\mu, (\mbf{i, a})}=
\FF^{(a_1)}_{\mu^0,\mu^1} \cdots \FF^{(a_{m-1})}_{\mu^{m-2},\mu^{m-1}} \FF^{(a_m)}_{\mu^{m-1},\mu},
\]
such that 
$\mu^l -\mu^{l-1}= a_l \alpha_{i_l}$ for $l=1,\cdots, m$. By Lemmas ~\ref{E-i-j} and ~\ref{E-F-i}, 
it suffices to show the boundedness of the complex $K_{\bullet}$ if
$K_{\bullet}$ is of the form $\FF_{\mu, (\mbf{j, b})}\EE_{(\mbf{i,a}), \mu}$ for any two pairs $(\mbf{i, a})$ and $(\mbf{j, b})$.
An argument similar to the proof of Lemma ~\ref{Serre} yields that 
\begin{align}
\label{boundedness-1}
\FF_{\mu, (\mbf{j, b})}\EE_{(\mbf{i,a}), \mu}=\Pi_! (\Q_{l, \mbf Z})[m], 
\end{align}
for some $m$, where 
$\pi$ is the projection from $\mbf Z$ to the variety
$\mbf E_{\Omega}(D, V^1, V^3)$ with the dimensions of $V^1 $ and $V^3$ determined by the pairs of sequences, and 
$\mbf Z$ is the variety of the data $( X^1 \overset{\rho_1}{\hookleftarrow} X^2 \overset{\rho_2}{\hookrightarrow} X^3)$ together with 
a pair $(U^s, W^t)_{1\leq s\leq m-1, 1\leq t\leq n-1} $ of $I$-graded partial flags of $V^1$ and $V^3$,
where the dimensions of $U^s$ and $W^t$ are determined by the $\mu$'s in the pairs $(\mbf{i, a}))$ and $(\mbf {j, b})$, respectively, 
 such that 
\[
\rho_1(X^2) \subseteq U^{n-1} \subseteq U^{n-2}\subseteq \cdots \subseteq U^1\subseteq V^1,
\quad \rho_1(X^2) \subseteq W^{n-1} \subseteq \cdots \subseteq W^1\subseteq V^3,\]
and  $U^s$ and $W^t$ are invariant under  $X^1$ and $X^3$, respectively, for any $s$ and $t$.
The morphism $\pi$ factors through the following varieties
\[
\mbf Z \overset{\pi_1}{\to} \mbf Z_1 \overset{\pi_2}{\to} \mbf Z_2 \overset{\pi_3}{\to} \mbf E_{\Omega} (D, V^1, V^3),
\]
where 
$\mbf Z_1$ is the variety obtained from $\mbf Z$ by forgetting the maps $\rho_2$, the variety $\mbf Z_2$ is the quotient variety of $\mbf Z_1$ by the group $G_{V^2}$
and the morphisms are clearly defined. 
It is clear that the functors $\Pi_{1!}$, $\Pi_{2!}$ and $\Pi_{3!}$ send bounded complexes to bounded complexes and  $\Pi_!=\Pi_{3!} \Pi_{2!} \Pi_{1!}$. The proposition follows. 
\end{proof}

\begin{lem}
\label{K-semisimple}
If $\II_{\mu}$ is semisimple, then $K_{\bullet}$ is semisimple. 
\end{lem}

\begin{proof}
 
Since $\FF_{\mu, (\mbf{j, b})}\EE_{(\mbf{i,a}), \mu} = \FF_{\mu, (\mbf{j, b})}\II_{\mu} \EE_{(\mbf{i,a}), \mu}$, we still have 
(\ref{boundedness-1}) with $\mbf Z$ replaced by 
\[
\mbf Y=\{  (X^1 \overset{\rho_1}{\hookleftarrow} X^2 \overset{\rho}{\hookrightarrow} X^4  \overset{\rho_2}{\hookrightarrow} X^3; (U^s, W^t)_{s,t})\},
\]
where  $\rho$ is an isomorphism and $(X^2, X^4)\in \mbf E_{\Omega}(D, V^2, V^4)$. 
Consider the following cartesian diagram
\[
\begin{CD}
\mbf Y @>>> \mbf Z_{\Omega}\\
@V\pi'_1 VV @V\pi_{12} VV\\
\mbf Y_1 @>p>> \mbf E_{\Omega} (D, V^2, V^4),
\end{CD}
\]
where $\mbf Y_1 =\{  (X^1 \overset{\rho_1}{\hookleftarrow} X^2 , X^4  \overset{\rho_2}{\hookrightarrow} X^3;  (U^s, W^t)_{s,t}) \}$ and the horizontal maps are projections. In particular, 
$p(X^1 \overset{\rho_1}{\hookleftarrow} X^2 , X^4  \overset{\rho_2}{\hookrightarrow} X^3; (U^s, W^t)_{s,t}) =  (X^2,  X^4)$.  By Base change theorem, we have 
\[
\pi'_{1!} (\Q_{l, \mbf Y}) = p^* (\II_{\mu}).
\]
Since $p$ is smooth and, by assumption, $\II_{\mu}$ is semisimple, we see that $\pi'_{1!} (\Q_{l, \mbf Y}) $ is semisimple.

Let $\mbf Y_2= \mbf Y_1/G_{V^2}\times G_{V^4} $. Then $\pi$ is the composition of the following morphisms
\[
\mbf Y \overset{\pi'_1}{\to} \mbf Y_1\overset{\pi'_2}{ \to} \mbf Y_2 \overset{\pi_3'}{\to} \mbf E_{\Omega}(D, V^1, V^3). 
\]
where $\pi'_2$ is a quotient map and $\pi_3'$ is a proper map. Again $\Pi_!$ is the composition of $\Pi_{1!}'$, $\Pi_{2!}'$ and $\Pi_{3!}'$. 
From this and that $\pi'_{1!} (\Q_{l, \mbf Y}) $ is semisimple, we see that $\FF_{\mu, (\mbf{j, b})}\EE_{(\mbf{i,a}), \mu} $ is semisimple. 
\end{proof}

\begin{conj}
\label{I-semisimple}
$\II_{\mu}$ is a simple perverse sheaf, up to a shift,  for any $\mu$. 
\end{conj}

This conjecture is partially proved in Section ~\ref{TypeA}. It is proved in full generality in ~\cite{W12} in a closely related setting.

\subsection{Category $\M_d$}

Let $\M_{d, \nu^1,\nu^2}$ be the full subcategory of $\DD^b_{\G}(\mbf E_{\Omega}(D, V^1, V^2) )$ whose objects are finite direct sums of shifts of the complexes $K_{\bullet}$
in  Section ~\ref{Kdot}. We set 
\[
\M_d =\oplus_{\nu^1,\nu^2\in\mbb N[I]} \M_{d, \nu^1,\nu^2}.
\]
It is clear that $\M_d$ is closed under the convolution product  ``$\cdot$'' in (\ref{cdot}).
Let 
$\LL_d$ be the Grothendieck group of the category $\M_d$. By definition,  this 
is free $\mbb A$-module  spanned  by the isomorphism classes of objects in $\M_d$ and subject to the relation that
$[A\oplus B]=[A] + [B]$  and $[A[1]] =v[A]$ for any objects  $A, B\in \M_d$.
By Lemma ~\ref{C-convolution} and Proposition ~\ref{associative},  the space $\LL_d$ is an associative algebra, where 
the multiplication on $\LL_d$ is descended from the convolution product ``$\cdot$'' in (\ref{cdot}).
By  Proposotion ~\ref{U-relations},

\begin{thm} 
\label{main}
There  is a unique surjective  $\mbb A$-algebra homomorphism
\[
\Psi_d: \, _{\mbb A}\! \dot{\U} \to \LL_d,
\]
sending
$1_{\mu}$,  $E^{(n)}_{\mu+\alpha_i, \mu}$ and  $F^{(n)}_{\mu -\alpha_i,\mu}$, 
to 
$\II_{\mu}$, $\EE^{(n)}_{\mu,\mu-n\alpha_i}$ and $\FF^{(n)}_{\mu,\mu+n\alpha_i}$, respectively,
 for any  $i \in I$ and  $\mu \in \X$.
 \end{thm}
 
 If $\Omega'$ is a second orientation of the graph $\Gamma$, we can define a similar category $\M_d^{\Omega'}$ and its Grothendieck group  $\LL_d^{\Omega'}$.
We have an equivalence of categories:
\[
\Phi_{\Omega}^{\Omega'}: \M_d\to\M_d^{\Omega'}.
\]
 The following proposition follows from  Lemma ~\ref{Fourier-convolution} and Lemma ~\ref{generator-Fourier},

 \begin{prop} 
 We have an  isomorphism of algebras
 $\Phi_d: \LL_d\to \LL_d^{\Omega'}$ sending  $\II_{\mu}$, $\EE^{(n)}_{\mu,\mu-n\alpha_i}$ and $\FF^{(n)}_{\mu,\mu+n\alpha_i}$ to the respective elements
 in $\LL^{\Omega'}_d$. 
 \end{prop}

The algebra $\LL_d$ should be a generalized $q$-Schur algebra (\cite{D03}) when the graph $\Gamma$ is of finite type.

\subsection{Category  $\C_d$}

Let $\BB_{d, \nu^1,\nu^2}$ be the set of isomorphism classes of simple perverse sheaves on 
$\DD^b_{\G}(\mbf E_{\Omega}(D, V^1, V^2) )$ appeared as simple subquotients in $^p\!H^s (K_{\bullet})$ for any $s\in \mbb Z$, where
$K_{\bullet}$ is in (\ref{complex}) and $^p\!H^i(-)$ is the perverse cohomology functor. 
We set 
\[
\BB_d =\sqcup_{\nu^1,\nu^2\in \mbb N[I]} \BB_{d,\nu^1,\nu^2}.
\]
Let $\C_{d, \nu^1,\nu^2}$ be the full subcategory of $\DD^b_{\G}(\mbf E_{\Omega}(D, V^1, V^2) )$ consisting of  semisimple objects $K$ such that 
the isomorphism classes of  simple summands of $^p\!H^s (K)$ are in $\BB_{d,\nu^1,\nu^2}$.
We set
\[
\C_d =\oplus_{\nu^1,\nu^2\in\mbb N[I]} \C_{d, \nu^1,\nu^2}.
\]
By Proposition ~\ref{K-semisimple},  we have 

\begin{lem}
\label{C-convolution}
Assume that $\II_{\mu}$ is semisimple for any $\mu$, then $\M_d\subseteq \C_d$ and, moreover, 
the category $\C_d$ is closed under the convolution product ``$\cdot$'' in (\ref{cdot}).
\end{lem}

Let $\KK_d$ be the Grothendieck group of the category $\C_d$.  
We see that $\BB_d$ is a basis of $\KK_d$.
By  Lemma ~\ref{C-convolution},

\begin{cor} 
\label{main-corollary}
Assume that $\II_{\mu}$ is semisimple for any $\mu$, then $\KK_d$ is an associative algebra with multiplication induced from the convolution product
``$\cdot$'' in (\ref{cdot}). Moreover,  the algebra $\LL_d$ is a subalgebra of $\KK_d$.
 \end{cor}

\begin{conj}
\label{conjecture-algebra}
We conjecture that $\LL_d=\KK_d$
and the image of $\dot{\B}$ under $\Psi_d$ is $\mathscr B_d$ with possible shifts, i.e.,
for any $b\in \dot{\B}$, $\Psi^{\Omega}_d (b) =v^{s(b)} [K]$ for some $s(b)\in \mbb Z$ and $[K]\in \mathscr B^{\Omega}_d$.
\end{conj}

We shall prove this conjecture partially in Section ~\ref{TypeA}.
In general, the identity $\LL_d=\KK_d$  is likely to be proved by the analysis of the geometry on the sink vertex and in a way similar to the proof of 
Lemma 3.10 in ~\cite{L03}.

Let $V_{\underline \lambda}=V_{\lambda_1}\otimes\cdots\otimes V_{\lambda_n}$ be the tensor product of the irreducible integrable representations of $\dot{\U}$ 
with highest weights $\lambda_1$, $\cdots$,  $\lambda_n$ in $\X^+$. Denote by $\DD_{\underline \lambda}$  the full subcategory of 
$\oplus_{V} \DD^-_{\G}(\mbf E_{\Omega}(D, V) )$ such that its  Grothendieck group $\mathscr V_{\underline \lambda}$  
is isomorphic to the integral form of  $V_{\underline \lambda}$ (see ~\cite{Zheng08}).
Then the bifunctor (\ref{action}) gives rise to a bifunctor
$\C_d \times \DD_{\underline \lambda} \to \DD_{\underline \lambda}$  by restriction, which descends to a bilinear map
\[
\circ : \KK_d \times \mathscr V_{\underline \lambda} \to \mathscr V_{\underline \lambda}.
\]
Let $\BB_{\underline \lambda}$ be the set of all isomorphism classes of simple perverse sheaves appearing in  $\DD_{\underline \lambda}$.
We then have
\[
a\circ b =\sum_{c\in \BB_{\underline \lambda}} s_{a, b}^c c, \quad \mbox{where} \; s_{a, b}^c\in \mbb N[v, v^{-1}],
\]
for any $a\in \BB_d$ and $b\in \BB_{\underline \lambda}$. From this, we have 

\begin{cor}
If  Conjectures ~\ref{I-semisimple} and ~\ref{conjecture-algebra}     hold, then the action of the canonical basis elements in $\dot{\U}$ 
on the canonical basis elements in $V_{\underline \lambda}$ has structure constants  in $\mbb N[v,v^{-1}]$ with respect to the canonical basis in $V_{\underline \lambda}$.
\end{cor}

 It has been proved in ~\cite{Zheng08} that the generators of the quantum group attached to $\Gamma$ act positively on the canonical basis of the representation $V_{\underline \lambda}$.

\section{Relation with the work  ~\cite{Zheng08}}

\subsection{Functor $\T$}

Let  $\FF^-_{\Omega, \G}(D, V^1, V^2)$ be the category of functors from the category $\DD^-_{\G}(\mbf E_{\Omega}(D, V^1))$ to $\DD^-_{\G}(\mbf E_{\Omega}(D, V^2))$. 
Similar to the functors $P_{st!}$ and $P_{st}^*$, we define ($\forall s=1, 2$)
\begin{equation}
\label{Pi}
\begin{split}
P_{s!} = Q\circ Qp_{s!} \circ Q_! , \;  P_s^* = Q\circ Qp_s^* \circ Q_!.
\end{split}
\end{equation}

Define a functor
\[
\Theta_{\Omega}: \DD^-_{\G}(\mbf E_{\Omega}(D, V^1, V^2)) \to \FF^-_{\Omega, \G}(D, V^1, V^2) 
\]
by  $\Theta_{\Omega}(K)= P_{2!} (K\otimes P_1^*(-))$ for any object $K$ in $  \DD^-_{\G}(\mbf E_{\Omega}(D, V^1, V^2))$ and the functors $P_{2!}$ and $P_1^*$ are defined 
in (\ref{Pi}).

\begin{prop}
\label{theta-convolution}
$\Theta_{\Omega}(K\cdot L) = \Theta_{\Omega}(L) \Theta_{\Omega}(K)$ for any objects $K$ in $\DD^-_{\G}(\mbf E_{\Omega}(D, V^1, V^2))$
and $L$  in $\DD^-_{\G}(\mbf E_{\Omega}(D, V^2, V^3))$.
\end{prop}

\begin{proof}
By definition, we have 
\[
\T(L)   \T(K) (M) 
= P'_{2!}(L\otimes (P'_1)^* \T(K) (M)) = P'_{2!}(L\otimes (P'_1)^*  P_{2!} (K\otimes P_1^*(M)),
\]
where $P'_{2!}$ and $(P'_1)^*$ are corresponding to the maps $p'_2$ and $p_1'$ in  the following cartesian diagram
\[
\begin{CD}
\mbf E_{\Omega}(D, V^1, V^3) @< p_{13}<< \mbf E_{\Omega}(D, V^1, V^2, V^3) @>p'_{12}>> \mbf E_{\Omega}(D, V^1, V^2)\\
@V \tilde p_2VV @Vp_{23}VV @Vp_2VV\\
\mbf E_{\Omega}(D, V^3) @<p'_2<< \mbf E_{\Omega}(D, V^2, V^3) @>p'_1>> \mbf E_{\Omega}(D, V^2).
\end{CD}
\]
By an argument similar to (\ref{associative-2}), we have 
$(P'_1)^*  P_{2!}= P_{23!}  (P'_{12})^*$. So
\[
\T(L)   \T(K) (M)=P'_{2!}(L\otimes P_{23!} (P'_{12})^*(K\otimes P_1^*(M)).
\]
By an argument similar to (\ref{associative-3}), we have 
$P_{23!}( A\otimes P_{23}^*(B) )= P_{23!} (A) \otimes B$. Thus,
\begin{equation}
\label{T-LHS}
\begin{split}
\T(L)   \T(K) (M)
&=P'_{2!} P_{23!} ( P_{23}^* (L) \otimes (P'_{12})^*(K\otimes P_1^*(M)))\\
&=P'_{2!} P_{23!} ( P_{23}^* (L) \otimes (P_{12}')^*(K)\otimes (P_{12}')^* P_1^*(M)).
\end{split}
\end{equation}
Similarly, we have 
\begin{equation}
\label{T-RHS}
\begin{split}
\T(K\cdot L) 
= \tilde P_{2!} P_{13!} ((P_{12}')^*(K) \otimes P_{23}^*(L) \otimes P_{13}^* \tilde P_1^* (M)),
\end{split}
\end{equation}
where $\tilde P^*_1$ comes from the projection $\mbf E_{\Omega}(D, V^1, V^3)\to \mbf E_{\Omega}(D, V^1)$.
The lemma follows by comparing (\ref{T-LHS}) with (\ref{T-RHS}) and  the following identity
\[
P'_{2!} P_{23!} = \tilde P_{2!} P_{13!} \quad \mbox{and}\quad 
(P_{12}')^* P_1^*=P_{13}^* \tilde P_1^*,
\]
which can be proved by a similar way as (\ref{associative-4}) and (\ref{associative-6}). 
\end{proof}

Define a functor of equivalence
\[
\Psi_{\Omega}^{\Omega'}: \FF^-_{\G, \Omega}(D, V^1, V^2) \to \FF^-_{\G, \Omega'}(D, V^1, V^2)
\]
by $\Psi_{\Omega}^{\Omega'} ( F) = \Phi_{\Omega}^{\Omega'} F a^* \Phi_{\Omega'}^{\Omega}$, where
$a$ is the map of multiplication by $-1$ along the fiber of the vector bundle $\mbf E_{\Omega} $ over $\mbf E_{\Omega\cap \Omega'}$ . 
Its inverse is given by $\Psi_{\Omega'}^{\Omega} (-) = a^* \Phi_{\Omega'}^{\Omega} (-) \Phi_{\Omega}^{\Omega'}$,
since $\Phi_{\Omega'}^{\Omega} \Phi_{\Omega}^{\Omega'} = a^*$. Moreover,  we have

\begin{lem}
$\Psi_{\Omega}^{\Omega'}$ commutes with the composition:
$\Psi_{\Omega}^{\Omega'} (F_2\circ F_1) = \Psi_{\Omega}^{\Omega'}(F_2) \circ \Psi_{\Omega}^{\Omega'}(F_1)$ for any
$F_1\in \FF^-_{\G,\Omega}(D, V^1, V^2) $ and $F_2\in \FF^-_{\G,\Omega}(D, V^2, V^3)$.
\end{lem}

Let $^1\DD^-_{\G}(\mbf E_{\Omega}(D, V^1,V^2))$ 
(resp. $^1\DD^-_{\G}(\mbf E_{\Omega}(D, V^s))$, $s=1, 2$) 
be the full subcategory of the category $\DD^-_{\G}(\mbf E_{\Omega}(D, V^1,V^2)) $ (resp. $\DD^-_{\G}(\mbf E_{\Omega}(D, V^i))$)
consisting of all objects such that $a^* (K)\simeq K$.
Let $^1 \FF^-_{\G, \Omega}(D, V^1, V^2)$ denote the category of functors from the category
$^1\DD^-_{\G}(\mbf E_{\Omega}(D, V^1))$ to $^1\DD^-_{\G}(\mbf E_{\Omega}(D, V^2))$. 
We have

\begin{lem}
\label{Theta-Phi-Psi}
The following diagram commutes
\[
\begin{CD}
^1\DD^-_{\G}(\mbf E_{\Omega}(D, V^1,V^2)) @>\Theta_{\Omega}>> ^1 \FF^-_{\G, \Omega}(D, V^1, V^2)\\
@V\Phi_{\Omega}^{\Omega'} VV @V\Psi_{\Omega}^{\Omega'} VV\\
^1\DD^b_{\G}(\mbf E_{\Omega'}(D, V^1,V^2)) @>\Theta_{\Omega'}>> ^1\FF_{\G, \Omega'}(D, V^1, V^2).
\end{CD}
\]
\end{lem}

\begin{proof}
For any $K\in\;  ^1\DD^-_{\G}(\mbf E_{\Omega}(D, V^1,V^2))$ and $K_1\in \; ^1\DD^-_{\G}(\mbf E_{\Omega}(D, V^1))$, we have 
\begin{equation*}
\begin{split}
\Psi_{\Omega}^{\Omega'} \T(K) (K_1) 
= \Phi_{\Omega}^{\Omega'} \T(K) a^* \Phi_{\Omega'}^{\Omega} (K_1) 
= \Phi_{\Omega}^{\Omega'} P_{2!} (K\otimes a^* P_1^*\Pi_{1!} ((\Pi'_1)^*(K_1)\otimes \mathcal L_1))[d_1],
\end{split}
\end{equation*}
where $P_{2!}$ and  $P_1^*$ are from (\ref{Pi}), $\Pi_{1!}$ and $(\Pi_1')^*$ come from the following projections
\[
\begin{CD}
\mbf E_{\Omega}(D, V^1) @<\pi_1<< \mbf E_{\Omega\cup \Omega'} (D, V^1) @>\pi_1'>> \mbf E_{\Omega'}(D, V^1);
\end{CD}
\]
$d_1$ is the rank of $\pi_1$ and $\mathcal L_1$ is defined in (\ref{L}).
Consider the following cartesian diagram
\[
\begin{CD}
\mbf E_{\Omega\cup \Omega'} (D, V^1)\times \mbf E_{\Omega}(D, V^2) @>\tilde \pi_1>> \mbf E_{\Omega}(D, V^1, V^2)\\
@V\tilde p_1VV @Vp_1 VV\\
\mbf E_{\Omega\cup \Omega'}(D, V^1) @>\pi_1>> \mbf E_{\Omega}(D, V^1).
\end{CD}
\]
By an argument similar to (\ref{associative-2}), we have $P_1^* \Pi_{1!} =\tilde \Pi_{1!} \tilde P_1^*$. So
\begin{equation*}
\begin{split}
\Psi&_{\Omega}^{\Omega'} \T(K) (K_1) 
=\Phi_{\Omega}^{\Omega'} P_{2!} (K\otimes a^*  \tilde \Pi_{1!} \tilde P_1^* ((\Pi'_1)^*(K_1)\otimes \mathcal L_1))[d_1]\\
&=\Phi_{\Omega}^{\Omega'} P_{2!} \tilde \Pi_{1!}(\tilde \Pi_1^* (K)\otimes a^*  \tilde P_1^* (\Pi'_1)^*(K_1)\otimes a^* \tilde P^*_1\mathcal L_1))[d_1]\\
&=R_{2!} (R_1^* P_{2!} \tilde \Pi_{1!} (\alpha) \otimes \mathcal L_2)[d_2],
\end{split}
\end{equation*}
where $R_{2!}$ and $R_1^*$ come from the following projections
\[
\begin{CD}
\mbf E_{\Omega}(D, V^2) @<r_1<< \mbf E_{\Omega\cup \Omega'}(D, V^2) @>r_2>> \mbf E_{\Omega'}(D, V^2),
\end{CD}
\]
$d_2$ is the rank of $r_1$,  $\alpha =\tilde \Pi_1^* (K)\otimes a^*  \tilde P_1^* (\pi'_1)^*(K_1)\otimes a^* \tilde P^*_1\mathcal L_1)[d_1]$, and 
$\mathcal L_2$ is defined in (\ref{L}). 
The following cartesian diagram
\[
\begin{CD}
\mbf E_{\Omega\cup \Omega'} (D, V^1, V^2) @>t_1>> \mbf E_{\Omega\cup \Omega'}(D, V^1)\times \mbf E_{\Omega}(D, V^2)\\
@Vs_1VV @Vp_2\tilde \pi_1VV\\
\mbf E_{\Omega\cup \Omega'}(D, V^2) @>r_1>> \mbf E_{\Omega}(D, V^2),
\end{CD}
\]
gives rise to the identity
$R_1^* P_{2!}\tilde \Pi_{1!} =S_{1!}T_1^*$. So we have 
\begin{equation}
\label{Psi-LHS}
\begin{split}
\Psi&_{\Omega}^{\Omega'} \T(K) (K_1) =R_{2!}(S_{1!}T_1^*(\alpha)\otimes \mathcal L_2)[d_2]
=R_{2!} S_{1!} (T_1^* (\alpha) \otimes S_1^* \mathcal L_2)[d_2]\\
&=R_{2!} S_{1!} (T_1^*\tilde \Pi_1^* (K)\otimes T_1^*  \tilde P_1^* (\Pi'_1)^* a^* (K_1)\otimes a^* T_1^*  \tilde P^*_1\mathcal L_1\otimes S_1^* \mathcal L_2)[d_1+d_2].
\end{split}
\end{equation}
On the other hand, we have
\begin{equation}
\label{Psi-RHS}
\begin{split}
\Theta_{\Omega'} \Phi_{\Omega}^{\Omega'}(K) (K_1) 
&=P_{2!}' (\Phi_{\Omega}^{\Omega'} (K)\otimes (P_1')^*(K_1))\\
&=P_{2!}' (M_{12!}' (M_{12}^*(K)\otimes \mathcal L_{12})[r_{12}]\otimes (P_1')^*(K_1))\\
&=P_{2!}' M'_{12!} (M_{12}^*(K) \otimes \mathcal L_{12} \otimes (M_{12}')^* (P_1')^* (K_1) ) [r_{12}],
\end{split}
\end{equation}
where $P_{2!}'$, $(P'_1)^*$ come from the following projections
\[
\begin{CD}
\mbf E_{\Omega'}(D, V^1) @<p_1'<< \mbf E_{\Omega'}(D, V^1, V^2) @>p_2'>> \mbf E_{\Omega'}(D, V^2),
\end{CD}
\]
and $M_{12!}'$, $M_{12}^*$, $\mathcal L_{12}$  and  $r_{12}$  are from \ref{Fourier}.
By comparing (\ref{Psi-LHS}) with (\ref{Psi-RHS}), the lemma follows from the following observations:
$r_2 s_1=p_2' m_{12}'$, $\tilde \pi_1 t_1=m_{12}$, $\pi_1' \tilde p_1 t_1=p_1'm_{12}'$,  $p_1t_1=p_1'$, $s_1=p_2'$ and 
$\mathcal L_{12} =a^* (P_1')^* \mathcal L_1\otimes (P_2')^* \mathcal L_2$.
Note that the last identity can be deduced  from the following  well-known fact. Let $s, p_1, p_2: k\times k\to k$ be the addition, first and second projections, respectively. Then
$s^* \mathcal L_{\chi} = p_1^* \mathcal L_{\chi}\otimes p_2^*\mathcal L_{\chi}$.
\end{proof}

We define the following functors in $\FF^-_{\Omega,\G}(D, V^1, V^2)$:
\begin{align*}
\mathfrak{ I}_{\mu}  &= \Pi_{2!}\Pi_1^*,  &&\mbox{if}\; \dim V^1=\dim V^2=\nu; \\
\mathfrak {F}^{(n)}_{\mu, \mu-n \alpha_i}  &= \Pi_{2!} \Pi_1^* [e_{\mu,n\alpha_i}]  , \; &&\mbox{if}\;  \dim V^1=\nu \;  \mbox{and}\;  \dim V^2=\nu+ ni;\\ 
 \mathfrak {E}^{(n)}_{\mu, \mu +n \alpha_i} &=  \Pi_{1!} \Pi_2^* [f_{\mu,n\alpha_i}] , &&\mbox{if} \; \dim V^1=\nu \;\mbox{and}\; \dim V^2=\nu-ni;
\end{align*}
where the functors $\Pi_{i!}$ and $\Pi_i^*$ are defined in (\ref{Pi}) and $e_{\mu,n\alpha_i}$ and $f_{\mu,n\alpha_i} $ are defined in (\ref{coefficient}).
Note that $\mathfrak I_{\mu} =\mbox{Id}_{\DD^-_{\G}(\mbf E_{\Omega}(D, V^1))}$, the identity functor,  since $\pi_1$ and $\pi_2$ are principal  $\G_{V^1}$-bundles.
We have 

\begin{prop}
\label{Theta-generator}
$\Theta_{\Omega}(\II_{\mu})=\mathfrak{ I}_{\mu}$, 
$\Theta_{\Omega}(\EE^{(n)}_{\mu, \mu-n \alpha_i})  = \mathfrak {F}^{(n)}_{\mu, \mu-n \alpha_i}  $ and
$ \Theta_{\Omega}(\FF^{(n)}_{\mu, \mu +n \alpha_i})=\mathfrak {E}^{(n)}_{\mu, \mu +n \alpha_i} $.
\end{prop}

\begin{proof}
We shall show that $\Theta_{\Omega}(\EE^{(n)}_{\mu, \mu-n \alpha_i})  = \mathfrak {F}^{(n)}_{\mu, \mu-n \alpha_i}  $.
For any $K_1\in \DD^-_{\G}(\mbf E_{\Omega}(D, V^1))$, we have
\begin{equation*}
\begin{split}
\mathfrak F^{(n)}_{\mu,\mu-n\alpha_i} (K_1) &=\Pi_{2!}\Pi_1^*(K_1) [e_{\mu, n\alpha_i}]=
P_{2!} \Pi_{12!} \Pi_{12}^* P_1^* (K_1) [e_{\mu, n\alpha_i}]\\
&=P_{2!} \Pi_{12!} (\bar{\mbb Q}_{l,\mbf Z_{\Omega}}\otimes \Pi_{12}^* P_1^* (K_1) [e_{\mu, n\alpha_i}]
=P_{2!} ( \Pi_{12!} (\bar{\mbb Q}_{l,\mbf Z_{\Omega}})[e_{\mu, n\alpha_i}] \otimes P_1^* (K_1))\\
&=P_{2!} ( Q( \pi_{12!} (\bar{\mbb Q}_{l,\mbf Z_{\Omega}})[e_{\mu, n\alpha_i}] )\otimes P_1^* (K_1))
=\Theta_{\Omega}(\EE^{(n)}_{\mu, \mu-n \alpha_i}) (K_1).
\end{split}
\end{equation*}
The rest can be proved similarly.
\end{proof}

By Lemmas ~\ref{generator-Fourier},  ~\ref{Theta-Phi-Psi}, and Proposition ~\ref{Theta-generator},  we have

\begin{cor}
\label{Psi-generator}
$\Psi_{\Omega}^{\Omega'} (\mathfrak I_{\mu})  =\mathfrak I_{\mu}$, 
$\Psi_{\Omega}^{\Omega'} (\mathfrak {F}^{(n)}_{\mu, \mu-n \alpha_i})=\mathfrak {F}^{(n)}_{\mu, \mu-n \alpha_i}$  and 
$\Psi_{\Omega}^{\Omega'} (\mathfrak {E}^{(n)}_{\mu, \mu+n \alpha_i})=\mathfrak {E}^{(n)}_{\mu, \mu+n \alpha_i}$.
\end{cor}

Actually, we need to show that the complexes in Corollary ~\ref{Psi-generator} are invariant under the functor $a^*$. This can be proved as in ~\cite[10.2.4]{Lusztig93}.

From Proposition ~\ref{theta-convolution}, Corollary ~\ref{Psi-generator} and Proposition ~\ref{U-relations}, we have 

\begin{prop}
\label{Psi-relation}
The functors $\mathfrak I_{\mu}$, $\mathfrak E^{(n)}_{\mu, \mu-n \alpha_i}$ and $\mathfrak F^{(n)}_{\mu, \mu +n \alpha_i}$ 
satisfy the defining relations of $_{\mbb A}\! \dot{\U}$.
\end{prop}

From Corollary ~\ref{Psi-generator}, one sees  that the functors $\mathfrak {F}^{(n)}_{\mu, \mu-n \alpha_i} $ and 
$\mathfrak {E}^{(n)}_{\mu, \mu +n \alpha_i} $ are the functors $\mathfrak F^{(n)}_{\nu, i}$ and $\mathfrak E^{(n)}_{\nu, i}$ in ~\cite{Zheng08}, respectively. 
Proposition ~\ref{Psi-relation} was first proved in ~\cite[2.5.8]{Zheng08}.

Now that the functor $\Theta_{\Omega}$ induces a bifunctor 
\begin{equation}
\label{action}
\circ: \DD^-_{\G}(\mbf E_{\Omega}(D, V^1, V^2) )\times\DD^-_{\G}(\mbf E_{\Omega}(D, V^1) ) \to \DD^-_{\G}(\mbf E_{\Omega}(D, V^2) )
\end{equation}
given by $K\circ K_1= \Theta_{\Omega} (K) (K_1) = P_{2!} (K\otimes P_1^* (K_1))$ 
for any $K\in \DD^-_{\G}(\mbf E_{\Omega}(D, V^1, V^2) )$ and $K_1\in \DD^-_{\G}(\mbf E_{\Omega}(D, V^1) )$.

\begin{rem}
(1). It should be true that the functor $\Theta_{\Omega}$ is fully faithful.

(2). We are not sure if the superscript $1$ in the categories in Lemma ~\ref{Theta-Phi-Psi} can be dropped.

\end{rem}

\section{BLM case}

\label{TypeA}

In this section, we put the following extra  assumptions on the graph $\Gamma$ and the $I$-graded vector space $D$ in Section ~\ref{framed}. 
\begin{itemize}
\item  The graph $\Gamma$ is of type $\mbf A_N$: $1 - 2 - \cdots - N$. 
\item  The space $D$ concentrates on the vertex $N$, i.e., $D_i=0$ for $1\leq i\leq N-1$ and $D_N$ is  a $d$-dimensional vector space over $k$.  
\end{itemize}

\subsection{Relation with ~\cite{BLM90}}
\label{sln}

We fix an orientation $\Omega$ of $\Gamma$ as follows:
$ 1 \to 2 \to \cdots \to N$.
Let 
\begin{equation}
\label{sln-doublequiver}
\Omega^2:  1 \to 2 \to \cdots \to N \to (N+1)  \leftarrow N' \leftarrow \cdots \leftarrow 2' \leftarrow 1'.
\end{equation}
To a pair  $( V, V')$ of $I$-graded vector spaces, the space $\mbf E_{\Omega}(D, V, V')$ defined in  section ~\ref{framed} 
is  the  representation variety of $\Omega^2$ with $V_a$ attached to the vertex $a$ and $V_a'$ to the vertex $a'$ and $D$ to the vertex $N+1$.

Let us fix some notations. We will use $x_{a+1, a}$ to denote elements in $\Hom(V_a, V_{a+1})$.  In particular, the element $q_N\in \Hom (V_N, D_N)$ is denoted by
$x_{N+1, N}$ in this section.
For a pair $(i, j)$ such that $i\leq j$, $1\leq  i \leq N$ and   $1\leq j \leq N+1$, we fix a representative  $S_{i, j}$  for the indecomposable representation of $\Omega^2$ of dimension equal to $1$ 
at the vertices $i, (i+1),\cdots, j$,  and equal to $0$ otherwise. The notation $S_{i', j'}$ is defined similarly.
For $1\leq i, j \leq N+1$, let $T_{i, j'}$ denote  the indecomposable representation of the quiver $\Omega^2$ such that the dimension  of $T_{i, j'}$ 
equals $1$ at the vertices $i, i+1, \cdots, N, (N+1), N', (N-1)', \cdots j'$, and zero otherwise. When $i=j$, we simply write $S_i$ for $S_{i, i}$, and  $T_i$ for $T_{i, i'}$.

Let      
\[
U=\{ (X, X') \in \mbf E_{\Omega}(D, V, V') | x_{a+1, a},  x'_{a+1,a},  \; \mbox{are injective,} \; \forall  a=1, \cdots, N\}.
\]
It is clear that $U$ is nonempty only when $\dim V_1 \leq \dim V_2 \leq \cdots \dim V_N \leq \dim D_N=d$ and the same property for $V'$.   The set of  
isomorphism classes, $[V]$, of $I$-graded vector spaces $V$  of such a property is then in bijection with  the set $\mbf S_d$ 
of all  nondecreasing $N+1$ step sequences, $\underline \nu=(0\leq \nu_1 \leq \cdots \leq \nu_N \leq d)$,  of nonnegative  integers., via the map 
\[
[V]\mapsto |V|=(\dim V_1, \dim V_2, \cdots, \dim V_N).
\]  
By abuse of notations, we write $V\in \mbf S_d$ if $\dim V_i\leq \dim V_j$ for $i\leq j$. To any  $\underline \nu \in \mbf S_d$,  we attach the partial flag variety
$\F_{\underline \nu}$ consisting of all flags $F=(0\equiv F_0\subseteq F_1\subseteq \cdots \subseteq F_N \subseteq F_{N+1}\equiv D)$
such that $\dim F_a=\nu_a$ for $a=1,\cdots, N$. 

For any pair $V, V'\in \mbf S_d$ such that $|V|=\underline \nu$ and $|V'|=\underline \nu'$, define a morphism of varieties
\[
u: U\to \F_{\underline{\nu}}\times \F_{ \underline{\nu'}}; \quad (X, X') \mapsto (F, F'), 
\]
where
$F=(0\subseteq \mbox{ im}  (x_{N+1,N}  x_{N, N-1} \cdots x_{2,1} ) \subseteq \cdots \subseteq \mbox{im} (x_{N+1,N}) \subseteq  D_N)$,
and $F'$ is defined similarly. 
It is well-known (\cite{Nakajima94}) that the algebraic group $\G_1=G_V\times G_{V'}$ acts freely on $U$, 
and  $u$ can be identified with  the quotient map $q: U\to \G_1\backslash U$.

The following diagram of morphisms
\[
\begin{CD}
\mbf E_{\Omega} (D, V, V') @<\beta << U @>u>>\F_{\underline{\nu}} \times \F_{\underline{\nu'}}
\end{CD}
\]
induces  a diagram of morphisms of algebraic stacks
\[
\begin{CD}
[\G\backslash \mbf E_{\Omega} (D, V, V) ]  @<\beta<<  [\G\backslash U] @>Qu >>  [G_D \backslash \F_{\underline \nu} \times \F_{\underline \nu'}], 
\end{CD}
\]
which, in turn, gives rise to the following diagram of functors
\[
\begin{CD}
\D^b_{\G} (\mbf E_{\Omega} (D, V, V') ) @>\beta^*>> \D^b_{\G}( U) @<Qu^*<<\D^b_{G_D}( \F_{\underline{\nu}} \times \F_{\underline{\nu'}}).
\end{CD}
\]

\begin{lem}
\label{thick}
Suppose that $K$ is a $\G$-equivariant complex on $\mbf E_{\Omega}(D, V, V')$. Then 
$K\in \mathcal N$ in Section ~\ref{general-localization}  if and only if  $K$ satisfies  that $\mrm{supp}(K) \cap U=\mbox{\O}$.
\end{lem}

\begin{proof}
To each $a$, we fix an orientation of the graph $\Gamma$:
\[
\Omega_a: 1 \to 2 \to \cdots \to (a-1) \leftarrow a \to (a+1) \to \cdots \to N.
\]
Let $W_a^{\Omega}$ (resp. $W_a^{\Omega\cup \Omega_a}$)  be the open subvariety 
of $\mbf E_{\Omega}$ (resp. $\mbf E_{\Omega \cup \Omega_a}$) consisting of all elements such that the component $x_{a+1, a}$ is injective.
We have the following diagram
\[
\begin{CD}
W_a^{\Omega} @<\tau<< W_a^{\Omega\cup \Omega_a} @>\tau'>> W_a^{\Omega_a}\\ 
@V\beta VV @V\beta VV @V\beta VV\\
\mbf E_{\Omega} @<\pi<<  \mbf E_{\Omega \cup \Omega_a} @>\pi'>> \mbf E_{\Omega_a},
\end{CD}
\]
where the $\beta$'s are the open inclusions, and $\tau$ (resp. $\tau'$) is the restriction of $\pi$ (resp. $\pi'$)  to the variety $W_a^{\Omega\cup \Omega_a}$.
Moreover, the squares in the above diagram are cartesian. From this diagram, we have
\[
\beta^* \Phi_{\Omega}^{\Omega_a} (K) = \tau'_! ( \tau^* \beta^* K\otimes \beta^* \mathcal L) [r].
\]
This implies that  
\begin{equation}
\label{sln-equivalent}
\mbox{supp}(\Phi_{\Omega}^{\Omega_a}(K) ) \cap W_a^{\Omega_a}=\mbox{\O} \quad
\mbox{if and only if} \quad
\mbox{supp} (K)\cap W_a^{\Omega}=\mbox{\O}.
\end{equation}
Assume that $K\in \mathcal N_{a}$, then by (\ref{sln-equivalent}), $\mbox{supp}(K)\cap W_a^{\Omega}=\mbox{\O}$. 
Since $W_a^{\Omega} \supseteq U$, we have
$\mbox{supp}(K)\cap U=\mbox{\O}$.  Since $\mathcal N$ is generated by the $\mathcal N_a$'s, we see that any object  $K\in \mathcal N$ has the property that
$\mbox{supp}(K) \cap U=\mbox{\O}$.

Now we shall  show that if $\mbox{supp}(K)\cap U=\mbox{\O}$, then $K\in \mathcal N$. 
Since any object in $\D^b_{\G}(\mbf E_{\Omega})$ can be generated by 
the $\G$-equivariant simple perverse sheaves, 
it is enough to show this statement for $K$  a simple perverse sheaf, which we shall assume from now on. 

Recall that $\mbf E_{\Omega}(D, V, V')$ can be regarded as a representation space of the quiver $\Omega^2$ in (\ref{sln-doublequiver}). 
By Gabriel's theorem,  there is only finitely many $\G$-orbits in $\mbf E_{\Omega}(D, V, V')$  and, moreover,   the stabilizers of the orbits 
in $\mbf E_{\Omega}(D, V, V')$ are connected.
So the $\G$-equivariant  simple perverse sheaves on $\mbf E_{\Omega} (D, V, V')$  are  the intersection cohomology complexes
$\IC_{\G} (\overline{\mathcal O_{(X,X')}})$,
attached to the $\G$-orbit  $\mathcal O_{(X, X')}$ in $\mbf E_{\Omega}(D, V, V')$ containing the   element  $(X, X')$.
Therefore, $K\in \mathcal N$ if the following  claim holds: 

{\bf Claim.} If $X(j) + X'(j') $ is injective for $1\leq j\leq i-1$, and $X(i)+X'(i')$ is not injective for some $i \in [1, N]$, 
then $\IC_{\G} (\overline{\mathcal O_{(X,X')}}) \in \mathcal N_i$.

The claim can be shown as follows. 
Let 
$\mbf i=(i_1,\cdots, i_n)$
be  a sequence of vertices in $\Omega^2$ and $\mbf a=(a_1,\cdots, a_n)$ be a sequence of positive integers such that 
$\sum_{i_l=i} a_l =\dim V_i$ for any vertex $i$.
Let $\tF_{\mbf{i, a}}^{\Omega} $ be the variety of all triples $(X, X', \mbf F) $, where $(X, X')\in \mbf E_{\Omega}(D, V, V')$ and 
$\mbf F=(\mbf F^0=D\oplus V\oplus V' \supset \mbf F^1 \supset\cdots \supset \mbf F^n=0)$ is a flag of graded vector subspaces in $D\oplus V\oplus V'$,  such that
\[
\dim \mbf F^{l-1}/\mbf F^l = a_l i_l
\quad
\mbox{and}
\quad
(X, X') (\mbf F^l) \subset \mbf F^l,
\quad \quad \forall l=1, \cdots, n.
\] 
Consider the projection  to the $(1, 2)$-components:
\begin{equation}
\label{resolution}
\pi_{\mbf{i, a}}^{ \Omega} : \tF_{\mbf{i, a}}^{\Omega} \to \mbf E_{\Omega}(D, V, V'), \quad (X, X', \mbf F) \mapsto (X, X').
\end{equation}
By ~\cite[Theorem 2.2]{R03}, one can choose a particular pair $(\mbf i, \mbf a)$ 
such that the image of $\pi_{\mbf{i, a}}^{ \Omega}$ is the closure $\overline{\mathcal O_{(X, X')}}$ of the orbit $\mathcal O_{(X, X')}$, and the restriction
$(\pi_{\mbf{i, a}}^{ \Omega})^{-1} (\overline{\mathcal O_{(X, X')}}) \to \overline{\mathcal O_{(X, X')}}$ is an isomorphism. 
Thus, the complex $\IC_{\G} (\overline{\mathcal O_{(X,X')}})$ is a direct summand of the semisimple complex
$(\pi_{\mbf{i, a}}^{ \Omega})_! (\bar{\mbb Q}_{l, \tF_{\mbf{i, a}}^{\Omega}})$, up to a shift.

So the complex $\Phi_{\Omega}^{\Omega_i} (\IC_{\G} (\overline{\mathcal O_{(X,X')}})) $ is a direct summand of the complex
$\Phi_{\Omega}^{\Omega_i} ((\pi_{\mbf{i, a}}^{ \Omega})_! (\bar{\mbb Q}_{l, \tF_{\mbf{i, a}}^{\Omega}}))$, up to a shift.
By ~\cite[10.2]{Lusztig93}, the complex 
$\Phi_{\Omega}^{\Omega_i} ((\pi_{\mbf{i, a}}^{ \Omega})_! (\bar{\mbb Q}_{l, \tF_{\mbf{i, a}}^{\Omega}}))$ is isomorphic to the complex
$(\pi_{\mbf{i, a}}^{ \Omega_i})_! (\bar{\mbb Q}_{l, \tF_{\mbf{i, a}}^{\Omega_i}})$, up to a shift,
where $\pi_{\mbf{i, a}}^{\Omega_i}$ and $\tF_{\mbf{i,a}}^{\Omega_i}$ are defined in exactly the same manner as 
$\pi_{\mbf{i,a}}^{\Omega}$ and $\tF_{\mbf{i,a}}^{\Omega}$ with $\mbf E_{\Omega}(D, V, V')$ replaced by $\mbf E_{\Omega_i}(D, V, V')$.
So the support of the complex $\Phi_{\Omega}^{\Omega_i} (\IC_{\G} (\overline{\mathcal O_{(X,X')}})) $ is contained in the image of the morphism
$\pi_{\mbf{i, a}}^{\Omega_i}$.

Observe that the conditions in the claim imply that 
\begin{itemize}
\item either $S_{1, i}$ or $S_{1', i'}$  is a direct summand of the representation $M$ of $\Omega^2$ corresponding to the element $(X, X')$ in the claim; 
\item  the representation $M$ does not contain any direct summand of  the form $S_{1, t}$ and $S_{1', t'}$ for $1\leq t < i$. 
\end{itemize}
From this observation, we see that 
the chosen  sequence $\mbf i=(i_1, \cdots, i_n)$ satisfies that $i_n=i$ or $i'$, from the construction in ~\cite{R03} and the Auslander-Reiten quiver of $\Omega^2$. 
This implies that 
for any $(X, X', \mbf F)\in \tF_{\mbf{i,a}}^{\Omega_i}$, either  $\ker (X(i)) \neq 0$ if $i_n=i$ or $\ker (X'(i'))\neq 0$ if $i_n=i'$. 
Therefore, the image of $\pi_{\mbf{i, a}}^{\Omega_i}$ is contained in the subvariety of $\mbf E_{\Omega_i}(D, V, V')$ 
 consisting of all elements  $(X, X')$ such that $\ker (X(i))\neq 0$ or $\ker (X'(i'))\neq 0$.  The claim follows.
\end{proof}

By Lemma ~\ref{thick}, there is an equivalence of triangulated categories 
\[
\bar Q: \D^b_{\G}(U) \to \mathscr D^b_{\G} (\mbf E_{\Omega}(D, V, V'))\equiv \D^b_{\G}(\mbf E_{\Omega}(D, V, V'))/\mathcal N
\]
such that $Q=\bar Q \beta^*$ and $\beta_! =Q_! \bar Q$. 
For $1\leq i\leq N$, $n\in \mbb N$,  we set
\begin{eqnarray*}
 \Delta_{\underline \nu}  &=& \{ ( F, F') \in \F_{\underline{\nu}} \times \F_{\underline{\nu}}| F=F'\},\\
  Y_{\underline \nu, - n\alpha_i}& =&\{(F, F') \in \F_{\underline{\nu}} \times \F_{\underline{\nu'}} | F_j \subseteq F'_j, \dim F_j'/F_j=n \delta_{ij}, \forall 1\leq j\leq N\},\\
 Y_{\underline \nu, +n\alpha_i} &=&\{ (F, F') \in \F_{\underline{\nu} } \times \F_{\underline{\nu'}} | F_j \supseteq F'_j, \dim F_j/F'_j=n \delta_{ij},\forall 1\leq j\leq N\},
\end{eqnarray*}
if $\underline{\nu'}$ exists.
Let 
\begin{eqnarray}
\label{Generator-1}
\begin{split}
1_{\underline \nu} = \Q_{l,  \Delta_{\underline\nu}}, \;
E^{(n)}_{\underline \nu, -n\alpha_i} 
= \Q_{l,Y_{\underline \nu, -n\alpha_i}} [e_{\underline\nu, -n\alpha_i}] , \;
F^{(n)}_{\underline \nu, + n \alpha_i} &=  \Q_{l,  Y_{\underline \nu, +n\alpha_i}} [f_{\underline \nu, + n \alpha_i}],
\end{split}
\end{eqnarray}
where  $e_{\underline\nu, -n\alpha_i}=n( \nu_{i+1} -(\nu_i+n))$, $f_{\underline \nu, + n \alpha_i}=n( (\nu_i-n)-     \nu_{i-1} ) $,
$\nu_{N+1}=d$ and $\nu_0=0$.
By combining the above analysis, we have the following proposition.

\begin{prop}
\label{equivalent}
We have a sequence of functors of equivalence
\begin{equation}
\label{sln-functors}
\D^b_{G_D}( \F_{\underline{\nu}} \times \F_{\underline{\nu'}}) \overset{Qu^*}{\to} \D^b_{\G}( U)  \overset{\bar Q}{\to} \mathscr D^b_{\G}(\mbf E_{\Omega}(D, V, V')).
\end{equation}
Moreover, 
for any  $\mu =\lambda -\nu$, 
\begin{equation}
\label{sln-generator}
\II_{\mu} =\bar Q Qu^* ( 1_{\underline \nu}), \quad 
\EE^{(n)}_{\mu, \mu-n\alpha_i}= \bar Q Qu^* \left (E^{(n)}_{\underline \nu, -n\alpha_i}\right ) , \quad \mbox{and} \quad 
\FF^{(n)}_{\mu, \mu+n\alpha_i}=\bar Q Qu^* \left  (F^{(n)}_{\underline \nu, +n\alpha_i} \right ).
\end{equation}
where $\II_{\mu}$, $\EE_{\bullet}$ and $\FF_{\bullet}$ are from (\ref{Generator-I}) and  $1_{\underline \nu}$, $E_{\bullet}$ and $F_{\bullet}$'s are from (\ref{Generator-1}).
\end{prop}

Define a convolution product ``$\cdot$'' on the categories $\D^b_{\G}(U_i)$ as follows.
We have the following commutative diagram
\[
\begin{CD}
U @>u_{ij}>> U\\
@V\beta VV @V\beta VV\\
\mbf E_{\Omega}(D, V, V', V'') @>p_{ij}>> \mbf E_{\Omega}(D, V^i, V^j),
\end{CD}
\]
where $p_{ij}$ is the projection to the $(i, j)$-component and $u_{ij}$ is the restriction of $p_{ij}$ to $U$. 
The morphism $u_{ij}$  then defines a morphism $ Q u_{ij} : [ \mbf H \backslash U ] \to  [ \G \backslash U] $.
To any objects $K$  and $L$ in $\D^b_{\G}(U)$, associated an object
\[
K\cdot L = Qu_{13!} (Qu_{12}^* (K) \otimes Qu_{23}^* (L)) \quad  \in \D^b_{\G}(U).
\]
Similarly, we define a convolution product 
\begin{equation}
\label{sln-convolution}
\circ: \D^b_{G_D} (\F_{\underline \nu}\times \F_{\underline \nu'}) \times \D^b_{G_D}(\F_{\underline \nu'}\times \F_{\underline \nu''}) \to \D^b_{G_D}(\F_{\underline \nu}\times \F_{\underline \nu''})
\end{equation}

The following proposition shows that the construction given in ~\cite{BLM90} is compatible with the one given in this paper.

\begin{prop}
\label{sln-convolution-compatible}
The convolution products on $\D^b_{G_D}(\F \times \F)$, $\D^b_{\G}(U)$  and 
$\mathscr D^b_{\G}(\mbf E_{\Omega}(D, V, V'))$ are compatible with the functors
$Qu^*$ and $\bar Q$ in the diagram (\ref{sln-functors}).
\end{prop}

\begin{proof}
First, we show that the convolution products are compatible with the functors $\bar Q$.
Recall that $Q=\bar Q \circ \beta^*$ and  $\beta_!=Q_! \bar Q$. We have
\begin{equation}
\label{compatibility-A}
\bar Q \circ Qu_{ij}^* =\bar Q \circ Qu_{ij}^*  \beta^* \beta_!
=\bar Q  \beta^*  \circ Qp_{ij}^* \beta_! = Q\circ Qp_{ij}^* \circ Q_! \bar Q=P_{ij}^* \bar Q.
\end{equation}
Similarly, 
\begin{equation}
\label{compatibility-B}
\bar Q  \circ Qu_{ij!} = \bar Q \beta^* \beta_! Q u_{ij!}=Q\circ Qp_{ij!} \beta_! = Q\circ Qp_{ij!} Q_! \bar Q= P_{ij!} \bar Q.
\end{equation}
From (\ref{compatibility-A}) and (\ref{compatibility-B}), 
we see that the convolution products on $\D^b_{\G}(U_i)$ and $\mathscr D^b_{\G}(\mbf E_{\Omega}(D, V, V'))$ are compatible with the functor $\bar Q$.

Now, we show that the convolution products are compatible with the functors $Qu^*$.  
We have the following cartesian diagram:
\[
\begin{CD}
[\mbf H \backslash U] @>Qu_{ij} >> [\G \backslash U]\\
@VQuVV @VQuVV\\
[G_D \backslash (\F_{\underline \nu}\times \F_{\underline \nu'} \times \F_{\underline \nu''} )] @>q_{ij}>> [G_D \backslash (\F_{\underline \nu^i} \times \F_{\underline \nu^j})].
\end{CD}
\]
where $q_{ij}$ is a projection and the first $U$ is contained in $\mbf E_{\Omega}(D, V, V', V'')$ such that $|V|=\underline \nu$, $|V'|=\underline \nu'$ and $|V''|=\underline \nu''$.
This cartesian diagram gives rise to the following identities:
\[
Qu^* q_{ij}^* =Qu_{ij}^* Q u^*, \quad \mbox{and} \quad Qu^* q_{ij!}= Qu_{ij!} Qu^*. 
\]
So for any 
$ K\in \D^b_{G_D}(\F_{\underline \nu}  \times \F_{\underline  \nu'})$ and $L\in \D^b_{G_D}(\F_{\underline \nu'} \times \F_{\underline  \nu''})$,
we have 
\begin{equation*}
\begin{split}
Qu^* (K\circ L) 
&= Qu^* q_{13!} (q_{12}^*(K) \otimes q_{13}^* (L) ) 
=Qu_{13!}  Qu^*(q_{12}^* (K) \otimes q_{13}^* (L))\\
&=Qu_{13!} (Qu_{12}^* Qu^*(K) \otimes Qu_{23}^* Qu^*(L)) =Qu^*(K)\cdot Qu^*(L).
\end{split}
\end{equation*} 
Therefore, the convolution products commute with the functor $Qu^*$. The  lemma follows.
\end{proof}

The following theorem follows from   (\ref{sln-generator}) in Proposition ~\ref{equivalent}, Proposition~\ref{sln-convolution-compatible}, and
the results in ~\cite{BLM90} and ~\cite{SV00}.

\begin{thm} 
\label{omega}
 Under the assumption in this section, the conjectures ~\ref{I-semisimple} and ~\ref{conjecture-algebra} hold. 
\end{thm}

\subsection{Singular support}

Let $\overline \Omega$ be the quiver obtained from $\Omega$ by reversing all arrows in $\Omega$. 
Let $\mbf E'_{\Omega} (D, V) =\oplus_{h\in \overline\Omega} \Hom (V_{h'}, V_{h''}) \oplus \oplus_{i\in I} \Hom (D_i, V_i)$ and 
$\mbf E'_{\overline \Omega}(D, V, V') = \mbf E'_{\overline \Omega}(D, V)\times \mbf E'_{\overline \Omega}(D, V')$.
We shall identify  the space 
\[
\mbf E:= \mbf E_{\Omega}(D, V, V')\oplus \mbf E'_{\overline \Omega}(D, V, V')
\]
with the cotangent bundles of $\mbf E_{\Omega}(D, V, V')$.
(See ~\cite[12]{Lusztig91} for more details.) 
Let $\Lambda$ be the closed subvariety of $\mbf E$ defined by the following (ADHM or GP) relations:
\begin{equation*}
\label{GP-1}
x_{a, a+1} x_{a+1, a} = x_{a, a-1} x_{a-1, a}, \quad 
x'_{a', (a+1)'} x'_{(a+1)', a'} = x'_{a', (a-1)'} x'_{(a-1)', a'}, \quad \forall 1\leq a\leq N;
\end{equation*}
\begin{equation*}
\label{GP-2}
x_{N+1, N} x_{N, N+1} =x'_{N+1, N'} x'_{N', N+1}.
\end{equation*}
This is  Lusztig's nilpotent quiver variety in ~\cite[12]{Lusztig91}. Let $\mbox{Nil}$ be the variety of nilpotent elements in $\mbox{End} (D)$. Let $G_D$ acts on $\mbox{Nil}$ by conjugation. 
Define a morphism 
\[
\pi: \Lambda \to \mbox{Nil}
\]
of varieties  by sending elements, say $X$,  in $\Lambda$ to $x_{N+1, N} x_{N, N+1}$. Note that 
the morphism $\pi$ is $\G_1$-equivariant.

Let $\Lambda^s$ be the open subvariety of $\Lambda$ consisting of all elements such that 
$x_{a+1,a}+ x_{a-1,a}$ and $x_{(a+1)', a'}+ x_{(a-1)', a'}$ are injective for $1\leq a\leq N$.
By ~\cite{Nakajima94}, the $\G_1$-action on $\Lambda^s$ is free and, moreover, admits 
a GIT  quotient $\G_1\backslash \Lambda^s$, isomorphic to the generalized Steinberg variety 
\[
\mathcal Z: = \{ (x, F, F') \in \mbox{Nil}\times \F_{\underline \nu}\times \F_{\underline \nu'}| x(F_a)\subseteq F_{a-1},\; x(F'_a)\subseteq F'_{a-1}, \; 1\leq a\leq N+1\}.
\]
We shall identify $\G_1\backslash \Lambda^s$ with $\mathcal Z$.
Due to the fact that $\pi$ is $\G_1$-equivariant, it factors through $\mathcal Z$, i.e., $\pi$ is the composition of the morphisms
 $\Lambda^s \overset{q}{\to} \mathcal Z \overset{\pi'}{\to} \mbox{Nil}$.

Given a simple perverse sheaf $K$ in $\D^b_G(X)$, 
the singular support, SS$(K)$,  of $K$ is defined to be the singular support of the complex $K_X$ in ~\cite{KS90}.
Here we abuse the notion slightly,  we mean  the singular support of its counterpart on $X(\mbb C)$ via the principals in ~\cite[Ch. 6]{BBD82}.

\begin{prop}
\label{uniform-Q}
 $K\in  \mathcal N$ if and only if $\mbox{SS}(K)\cap \Lambda^s =\mbox{\O}$. 
\end{prop}

\begin{proof} 
Suppose that $\mbox{SS}(K) \cap \Lambda^s=\mbox{\O}$. 
From ~\cite{KS90},   $\mbox{SS}(K)\cap \mbf E_{\Omega}(D, V, V')\simeq \mbox{supp} (K)$ where $\mbf E_{\Omega}(D, V,V')$ is identified with the subspace 
$0\oplus \mbf E_{\Omega}(D, V,V')$ of $\mbf E$. So we have
\[
\mbox{\O}
=\mbox{SS}(K) \cap \Lambda^s\supseteq \mbox{SS}(K) \cap \Lambda^s\cap \mbf E_{\Omega}(D, V, V') 
=\mbox{Supp}(K) \cap U
\]
where $U$ is the open subvariety of $\mbf E_{\Omega}(D, V, V')$ defined in the section ~\ref{sln}. 
By Lemma ~\ref{thick}, we see that $K\in  \mathcal N$.

On the other hand, if $\mbox{SS}(K) \cap \Lambda^s\neq \mbox{\O}$, then it is a non empty closed subvariety of $\Lambda^s$.  
Hence, $\pi(\mbox{SS}(K) \cap \Lambda^s)$ is a closed subvariety of $\mbox{Nil}$ due
to the fact that $\pi$ can be decomposed as $\pi' q$ with $\pi'$ a proper map and $q$ a quotient map. Note that $\pi(\mbox{SS}(K) \cap \Lambda^s)$ is also 
$G_D$-invariant since $K$ is a $\G$-equivariant complex.
So $0\in \pi(\mbox{SS}(K) \cap \Lambda^s)$, which implies that $\mbox{SS}(K) \cap \Lambda^s\cap \pi^{-1}(0)\neq \mbox{\O}$. Observe that 
$\Lambda^s\cap \pi^{-1}(0)=U$. Thus, $\mbox{Supp}(K) \cap U\neq \mbox{\O}$.  By Lemma ~\ref{thick},  $K\not \in  \mathcal N$.
The proposition follows.
\end{proof}

\end{document}